\documentclass[10pt,a4paper]{article}

\usepackage{amsmath,amsthm,amsbsy}
\usepackage{amsfonts}
\usepackage{amssymb}
\usepackage{graphicx}

\usepackage[ulem=normalem]{changes}
\usepackage{tikz}

\newtheorem{definition}{Definition}
\newtheorem{proposition}{Proposition}
\newtheorem{theorem}{Theorem}
\newtheorem{lemma}{Lemma}
\newtheorem{corollary}{Corollary}

\theoremstyle{definition}
\newtheorem{example}{Example}

\newcommand{\red}{\textcolor{black}}
\newcommand{\green}{\textcolor{green}}

\def\v{\mathop{\boldsymbol v}\nolimits}
\def\w{\mathop{\boldsymbol w}\nolimits}
\def\x{\mathop{\boldsymbol x}\nolimits}
\def\y{\mathop{\boldsymbol y}\nolimits}
\def\z{\mathop{\boldsymbol z}\nolimits}
\def\CC{\mathop{\boldsymbol C}\nolimits}
\def\M{\mathop{\boldsymbol M}\nolimits}
\def\P{\mathop{\boldsymbol P}\nolimits}
\def\X{\mathop{\boldsymbol X}\nolimits}
\def\bxi{\mathop{\boldsymbol\xi}\nolimits}
\def\bOmega{\mathop{\boldsymbol\Omega}\nolimits}

\def\O{\mathop{\mathcal O}\nolimits}
\def\S{\mathop{\mathcal S}\nolimits}

\def\Tm{\mathop{\boldsymbol T}_{\!m+2}\nolimits}

\def\TT{\mathop{\boldsymbol T}_{\!4}\nolimits}

\begin{document}

\title{The Logarithm Map, its Limits\\ and Fr\'echet Means in Orthant Spaces}

\author{D. Barden\thanks{Girton College, University of Cambridge, Cambridge, CB3 0JG, UK ({\tt d.barden@dpmms.cam.ac.uk}).} \and H. Le\thanks{School of Mathematical Sciences, University of Nottingham, Nottingham, NG7 2RD, UK ({\tt huiling.le@nottingham.ac.uk}).}}

%\date{\today}
%\date{1st April 2016}
\date{}
\maketitle

\begin{abstract}
The first part of the paper studies the expression for, and the properties of, the logarithm map on an orthant space, which is a simple stratified space, with the aim of analysing Fr\'echet means of probability measures on such a space. In the second part, we use these results to characterise Fr\'echet means and to derive various of their properties, including the limiting distribution of sample Fr\'echet means.  
\end{abstract}


\section{Introduction}

Several papers have recently appeared concerning probabilistic and statistical analysis of data on certain stratified spaces (cf. \cite{BHV}, \cite{BLO1}, \cite{HHLMMMNOPS}, \cite{MB} and \cite{HMMN}). One such example is the analysis of phylogenetic trees on the BHV space introduced in \cite {BHV} (cf. \cite{SH}, \cite{MO}, \cite{TN1}, \cite{BLO2}, \cite{KC}, \cite{MOP} and \cite{TN2}). The BHV space $\Tm$ of metric trees with $m+2$ leaves is a stratified CAT(0)-space with each stratum being isometric with a positive Euclidean orthant that is at most $m$-dimensional. It is already clear from these preliminary results that some fundamental statistics exhibit strikingly different features from the corresponding ones on Euclidean spaces or on manifolds and that one faces significant challenges in developing novel tools to analyse them, on account of the non-trivial topological structure of these spaces. It also becomes apparent that, although the topological and geometrical properties of stratified spaces have been extensively studied and are mostly well understood, many of the properties required for probabilistic and statistical analysis of data on these spaces have not been addressed. 

\vskip 6pt
This paper concentrates on orthant spaces introduced in \cite{MOP}, a relatively simple type of stratified space but more general than the space $\Tm$ of phylogenetic trees. The latter has $(2m+1)!!$ $m$-dimensional strata, together with their bounding strata, selected from among the $\begin{pmatrix}M\\m\end{pmatrix}$ positive orthants in $\mathbb{R}^M$ where $M=2^{m+2}-m-4$. In particular, each co-dimension one stratum bounds exactly three top-dimensional strata. Thus not only are the relevant dimensions sparse, but the percentage of the positive orthants occupied by the tree space of each dimension declines exponentially. These constraints, such as the restrictions on the dimension and the number of orthants involved in the space, no longer hold in a general orthant space, although we do have to make one restriction to ensure that it is a CAT(0)-space. We shall recall, in the next section, the concept of an orthant space, introducing the subsidiary concepts and definitions we use to describe the structure of such spaces and, in particular, of their tangent cones at the various points.

A fundamental concept for statistical analysis of non-Euclidean data is that of the Fr\'echet mean, which generalises the concept of the mean of Euclidean data. A point $\x_0$ in a metric space $\M$ is a Fr\'echet mean of a probability measure $\mu$ on $\M$ if, at $\x_0$, the Fr\'echet function of $\mu$ defined by
\begin{eqnarray}
\frac{1}{2}\int_{\M} d(\x,\x')^2\,d\mu(\x')
\label{eqn1}
\end{eqnarray}
attains its global minimum. In order to characterise and locate Fr\'echet means, we need to take directional derivatives of the Fr\'echet function and hence, implicitly, of the distance function. The latter involves the logarithm map $\log_{\x^*}(\x)$ which, analogous to the inverse of the exponential map on manifolds, is the initial tangent vector to the geodesic from $\x^*$ to $\x$. This logarithm map is globally well-defined on CAT(0)-spaces and has been studied, for example, in \cite{BK} and \cite{IN}. However, these results do not cover all the properties required for our analysis, although naturally we do rely on some of their results. On the other hand, an algorithm for finding the geodesic between any two given trees in the tree space $\Tm$ was given in \cite{MO} and, using the analysis behind that algorithm, the expression for the logarithm map $\log_{\x^*}$ was obtained in \cite{BLO2} when $\x^*$ lies in a top-dimensional stratum. Although this expression for $\log_{\x^*}$ could be extended to more general orthant spaces, it is noted in \cite{BLO2} that these results are not adequate to provide a tool for analysing Fr\'echet means when they lie in any stratum of co-dimension at least two. The latter requires a better understanding of the behaviour of the logarithm map as the end points of the geodesics move within and between strata. To this end, we first re-examine geodesics directly from first principles in Section 3, in particular avoiding the implicit assumption that $\x^*$ lies in a top-dimensional stratum. This leads, in Section 4, to an explicit expression, \red{given in Theorem \ref{thm1},} for a version of the logarithm map that we shall use, valid for any point in an orthant space. Since \red{the form of} this expression is determined by the carrier of the geodesic, we analyse possible changes in that carrier, focussing on the set, specified in Definition \ref{def1g}, of points $\x$ at which significant changes occur. This allows us, in Section 5, to derive the \red{directional limits} of the logarithm map as the reference point $\x'$ approaches $\x^*$ from a co-bounding stratum. We also study the projections of these limits, and the limits of the projections, onto the various strata related to the stratum in which $\x^*$ lies. This enables us to prove the existence of, and to identify, certain of their derivatives and directional derivatives.

With this understanding of the logarithm map, the second part of the paper turns its attention to the analysis of Fr\'echet means. In Section 6 we obtain, in Theorem \ref{thm3}, the necessary and sufficient conditions for a point $\x^*$ to be the Fr\'echet mean of a probability measure on the orthant space $\X^m$. Two special sets arise in this analysis. Firstly, one of the criteria in Theorem \ref{thm3} involves an inequality and the set, specified in Definition \ref{def2a}, of vectors in the tangent cone to $\X^m$ at the Fr\'echet mean for which that is an equality is significant. Secondly, there is the set given by Definition \ref{def2}. This is related to a limit of the logarithm map and, in a certain sense, encapsulates the `departure' of this limit from the analogous behaviour of the logarithm map on a Euclidean space. Both of these sets are related to the limiting distribution of sample Fr\'echet means, which we establish in the final Section 7. There, in particular, we relate the limiting distribution with Euclidean Gaussian random variables. The covariance matrices of these random variables are related to the derivative of the projection of the logarithm map and to projections of the limits of the logarithm map.   

\vskip 6pt
Although we do not make it explicit, in view of our previous results for $\Tm$ and the comments in \cite{MOP}, our interest in this paper is primarily in the case that $\x^*$ lies in a stratum of \red{local} co-dimension at least two. The results, when restricted to a \red{locally} top-dimensional or co-dimension one stratum, do generalise those for tree spaces in \cite{BLO2} although the approach here is necessarily more complex in order to encompass all cases.

\section{Orthant spaces}

\red{Throughout this paper, we shall use the term `positive' to mean strictly positive. By an open positive orthant in the Euclidean space $\mathbb{R}^M$ we shall mean, for some subset $E=\left(u^{\phantom{A}}_{l_1},\cdots,u^{\phantom{A}}_{l_m}\right)$ of the standard ordered orthonormal basis $U=\left(u_1,\cdots,u^{\phantom{A}}_M\right)$ of $\mathbb{R}^M$, the relatively open set
\[\mathcal O(E)=\left\{\sum_{i=1}^m\lambda_iu^{\phantom{A}}_{l_i}\mid\lambda_i>0\right\}.\]
We denote by $\mathbb R(E)$ the subspace spanned by $E$, and we shall refer to the $u^{\phantom{A}}_{l_i}\in E$ as the axes of $\mathbb R(E)$ or of $\O(E)$.} Then, an orthant space is a union of open positive orthants in a common Euclidean space with certain natural constraints, \red{as specified in the following definition}, that ensure, for example, that such spaces are also CAT(0). Orthant spaces were first introduced in \cite{MOP} as a generalisation of the tree spaces of \cite{BHV}. 

\begin{definition}
For two given integers $M\geqslant m$, an orthant space $\X^m$ of dimension $m$ is a subspace of the Euclidean space $\mathbb{R}^M$ that is a union of open positive orthants, \red{whose} maximum dimension \red{is} $m$, \red{and has the intrinsic metric induced from the Euclidean metric on $\mathbb R^M$. It satisfies the following conditions:}
\begin{itemize}
\item[$(i)$] for every orthant $\sigma$ in $\X^m$, the orthants in the closure $\overline\sigma$ of $\sigma$ are also included in $\X^m$; 
\item[$(ii)$] if, for any positive orthant $\sigma$ in $\mathbb{R}^M$, all the $2$-dimensional orthants in its closure are in $\X^m$, then $\sigma$ itself is in $\X^m$.
\end{itemize}
\label{def0a}
\end{definition}

\red{The intrinsic metric on $\X^m$ is the length metric as defined in \cite{BH}. It is the metric $d$ for which, for any two points $\x_1$ and $\x_2$ in $\X^m$, the distance $d(\x_1,\x_2)$ is the infimum of the lengths of piecewise linear paths \red{in $\X^m$} joining $\x_1$ to $\x_2$. In particular, a geodesic will also be piecewise linear and linear within each stratum.}

\red{Note that there is no loss of generality in restricting $\X^m$ to contain only positive orthants: given two orthants that differ only in having positive or negative coordinates with respect to one particular axis, the intrinsic metric will be the same as it would be if we replace, say the negative axis, by an axis orthogonal to $\mathbb{R}^M$. Thus, rather than considering $\X^m$ to be a union of arbitrary orthants in $\mathbb{R}^M$, we could consider it to be a union of positive orthants in $\mathbb{R}^{2M}$. Henceforth, we shall assume all our orthants to be open and positive, mentioning their closure explicitly where that is relevant.}

The first condition \red{in the above definition} correlates with the constraints used in the definition for orthant space in \cite{MOP} and the second one restricts attention to the `non-positively curved' orthant spaces in \cite{MOP} (Proposition 6.10). These two conditions were first used by the authors of \cite{BHV} to ensure the CAT(0)-property for tree spaces.  

\vskip 6pt
Throughout the rest of the paper, $\X^m$ will denote an orthant space of fixed dimension $m$ viewed as comprising strata that are orthants of a fixed Euclidean space $\mathbb{R}^M$, where $M$ is not necessarily $2^{m+2}-m-4$ as it would be for tree space. Also, whenever we specify an orthant by a union of subsets of the standard orthonormal basis $U$ of $\mathbb R^M$, that will always be intended as a union of mutually disjoint subsets.

\vskip 6pt
The orthant space $\X^m$ so defined is a Whitney stratified set in the sense of Thom, \cite{RT}, the strata being the various orthants that comprise $\X^m$. \red{Note that, since $\X^m$ is a union of orthants in a fixed Euclidean space $\mathbb{R}^M$, the number of strata in $\X^m$ is always finite. $\X^m$} has the structure of a cone with vertex, or `cone point', the origin $o$ \red{in $\mathbb R^M$}, since each orthant is such a cone without its vertex, but that vertex, the origin, is necessarily included in $\X^m$. \red{In particular, $\{o\}$ is the unique zero-dimensional stratum in $\X^m$.} Note however that \red{our relatively open strata differ from those in \cite{BH}.} 

The CAT(0)-property of the orthant space $\X^m$ results \red{as follows, where all the references are to} \cite{BH}. The intersection $L$ of $\X^m$ with the unit sphere in $\mathbb{R}^M$ is a simplicial complex on account of condition $(i)$ and, since the axes in $\mathbb{R}^M$ are orthogonal, it is an `all-right spherical complex' (Section 7A.10) which, on account of condition $(ii)$, is a `flag complex'. Then, by a theorem of Gromov (Theorem 5.18), $L$ is a CAT(1)-space. The metric on $\X^m$ implied by describing it as the $0$-cone over $L$ (Definition 5.6) is the intrinsic metric so that, by the theorem of Berestowski (Theorem 3.14), $\X^m$ is CAT(0). 

In particular, by the Cartan-Hadamard theorem (cf. \cite{BH}, p.193), there is a unique geodesic between any two points of the orthant space $\X^m$. It follows that each stratum is totally geodesic in the strong sense that, if a geodesic contains two points of a stratum, it must include the entire linear segment in that stratum determined by those two points. On the other hand, although the distance metric for the CAT(0)-structure is induced from the Euclidean metric, the angles along and between curves may differ for the two contexts. For example, a geodesic, defined as a shortest path between its endpoints in either context, will be a piecewise linear curve in $\mathbb{R}^M$, linear in each stratum, with angle $\pi/2$ in the Euclidean subspace metric where it changes stratum. However, for the CAT(0)-structure, that angle is defined to be $\pi$. 

\vskip 6pt
\red{The properties of an orthant space are largely determined by the incidence relations between its various strata. The following definitions capture two such relationships that will be used frequently in the paper.}

\begin{definition}
\red{For subsets $E$ and $F$ of the standard orthonormal basis $U=(u_1,\cdots,u^{\phantom{A}}_M)$ of $\mathbb R^M$, if $E\subseteq F$, then the orthant} ${\cal O}(E)$ is said to bound ${\cal O}(F)$ and ${\cal O}(F)$ to co-bound ${\cal O}(E)$. 
\label{def0b} 
\end{definition}

Note that, unlike the case for tree spaces, strata of lower dimension than $m$ need not bound any higher dimensional strata, \red{in particular they need not bound $m$-dimensional strata.}

\begin{definition}
\red{An orthant $\sigma$ of dimension $k$ in $\X^m$ is said to have co-dimension $m-k$ and, if $m'(\leqslant m)$ is the maximum dimension of orthants that $\sigma$ co-bounds, then $\sigma$ is said to have local co-dimension $m'-k$.}
\end{definition}

\vskip 10pt
\noindent\textbf{The tangent cone}\label{tangent cone}

\vskip 6pt
It is natural for our purposes to follow \cite{BH} and to define the tangent cone to $\X^m$ at a point $\x$ to consist of all initial tangent vectors to smooth curves starting from $\x$, the smoothness possibly only being one-sided at $\x$. Note, however, that this is not the same as the generalised tangent space of \cite{GM}. To describe the tangent cone in more detail we work in $\mathbb{R}^M$. Then, when $\x$ lies in a top-dimensional, or locally top-dimensional, stratum $\sigma$ of dimension $m'(\leqslant m)$, the orthant space $\X^m$ is locally an $m'$-dimensional manifold so that a smooth curve can be extended on both sides of $\x$. Thus, the tangent cone will be the usual tangent space, a subspace of $\mathbb{R}^M$ isometric with $\mathbb{R}^{m'}$ and tangent to $\sigma$. However, if $\x$ lies in a stratum of \red{locally} positive co-dimension, then the orthant space $\X^m$ is no longer locally a manifold. Consequently, the tangent cone at $\x$ is no longer a Euclidean space. For example, if the stratum $\sigma$ has co-dimension one and bounds top-dimensional strata, the tangent cone to $\X^m$ at $\x$ is an open book: \red{it has a closed half space $\mathbb {H}^m$ for each top-dimensional stratum $\tau$ co-bounding $\sigma$, with all the boundary $(m-1)$-dimensional faces identified with each other and with the tangent space to $\sigma$ at $\x$.} 

More generally, the tangent cone at a point $\x$, in a stratum $\sigma=\O(E)$ of co-dimension $l(\geqslant1)$, has a \red{topology and} stratification imitating that of $\X^m$ itself in the neighbourhood of $\x$: for each stratum $\tau=\O(E\cup F)$ of co-dimension $l'<l$ that co-bounds $\sigma$, so that $F$ comprises the \red{basis} vectors that have positive coordinates in $\tau$ but zero coordinates in $\sigma$, there is the \red{closed} stratum $\mathbb{R}(E)\times\red{\overline{\O(F)}}$ in the tangent cone. Then, the tangent cone at $\x$ has \red{its} stratification determined by identifying the various \red{$\mathbb{R}(E)\times\{\bf0\}$ with each other} as well as \red{identifying any} tangent axes shared by pairs of strata that co-bound $\sigma$. In particular, when no strata co-bound $\sigma$, the tangent cone is simply the Euclidean space $\mathbb{R}(E)$.

\begin{definition} 
Let $\sigma=\O(E)$ and $\tau=\O(E\cup F)$ be two strata in $\X^m$ with co-dimensions $l$ and $l'<l$, respectively. The component $\mathbb{R}(E)$ common to all the strata in the tangent cone to $\X^m$ at $\x\in\sigma$ is referred to as the tangent space to $\sigma$ at $\x$. Vectors in the \red{$($open$)$} stratum $\mathbb{R}(E)\times\O(F)$ of the tangent cone at $\x\in\sigma$ with non-zero second component are referred to as vectors tangent to $\tau$ at $\x$. 

The set of unit vectors in $\mathbb{R}(E)\times\O(F)$ is denoted by $\S^{m-l'}_{\tau,\sigma}$ and the \red{sub}set of those in $\{{\bf0}\}\times\O(F)$ by $\S^{l-l'}_{\tau\setminus\sigma}$.
\label{def0c}
\end{definition}

\red{The sets $\S^{m-l'}_{\tau,\sigma}$ and $\S^{l-l'}_{\tau\setminus\sigma}$ are open spherical segments of dimensions $m-l'-1$ and $l-l'-1$ respectively, the latter lying in the space $\mathbb R(F)$ orthogonal to $\mathbb R(E)$.}

\vskip 6pt 
Note that the basis vectors in $E$ do not generally precede those of $F$ in \red{the standard ordered basis} $U$, and so writing the stratum as $\mathbb{R}(E)\times\O(F)$ implies an appropriate permutation of the coordinates. 

\begin{definition}
\red{For any subset $E$ of the standard ordered orthonormal basis $U$ of $\mathbb R^M$, where $E$ does not necessarily inherit its order from $U$, we denote by $\jmath:\mathbb R(E)\rightarrow\mathbb R^M$ the linear transformation permuting coordinates and positioning them appropriately as coordinates, with respect to $U$, of a vector in $\mathbb R^M$.}
\label{def0i}
\end{definition}

\red{We are mainly interested in the restriction of $\jmath$ to  subspaces of $\mathbb R(E)$. For example, if $E=(u_1,u_4)$ and $F=(u_2,u_6)$, then a point $(\x,\y)$ in $\mathbb R(E)\times\O(F)$ with coordinates $((x_1,x_2),(y_1,y_2))$ would have $\jmath(\x,\y)=(x_1,y_1,0,x_2,0,y_2,0,\cdots,0)$ in $\mathbb R^M$.}

\vskip 6pt
Inherited from the CAT(0)-structure of $\X^m$, the tangent cone to $\X^m$ at $\x$, since it is metrically complete, also has a 
CAT(0)-structure (cf. \cite{BH}, Theorem 3.19). \red{While the CAT(0)-metric on $\X^m$ is, by definition, the intrinsic metric, the CAT(0)-metric on the tangent cone to $\X^m$ at $\x$ is defined in terms of the Alexandrov angle. Recall that, for any three points $\x,\x_1,\x_2$ in $\X^m$, the comparison triangle of the geodesic triangle $\Delta(\x,\x_1,\x_2)$ in $\X^m$ formed by $\x,\x_1,\x_2$ is the triangle $\bar\Delta(\x,\x_1,\x_2)$ in the Euclidean plane with vertices $\bar{\x}$, $\bar{\x}_1$, $\bar{\x}_2$ such that the Euclidean distances $d(\bar{\x},\bar{\x}_1)$ etc. match the intrinsic distances $d(\x,\x_1)$ etc. in $\X^m$. Then, the Alexandrov angle $\angle_{\x}(\gamma_1,\gamma_2)$ between the geodesics $\gamma_1$ and $\gamma_2$ starting from $\x$ is defined to be 
\[\angle_{\x}(\gamma_1,\gamma_2)=\lim_{t\rightarrow0}\overline\angle_{\x}(\gamma_1(t),\gamma_2(t)),\]
where $\overline\angle_{\x}(\gamma_1(t),\gamma_2(t))$ is the Euclidean angle at $\bar\x$ of the comparison Euclidean triangle $\bar\Delta(\x,\gamma_1(t),\gamma_2(t))$ (cf. \cite{BH}, Section 1.12). Note that, since geodesics in $\X^m$ are piecewise linear, the above limit is well-defined. Then, the inner product on the tangent cone of $\X^m$ at $\x$ is defined by 
\begin{eqnarray}
\ll\w_1,\w_2\gg=\|\w_1\|\,\|\w_2\|\cos\angle_{\x}(\gamma_1,\gamma_2),
\label{eqn0a}
\end{eqnarray}
where $\w_1$ and $\w_2$ are the initial tangent vectors of $\gamma_1$ and $\gamma_2$. By analogy with vectors in the tangent space to a manifold, the distance $\rho_{\x}(\w_1,\w_2)$ between vectors $\w_1$ and $\w_2$ in the tangent cone at $\x$ is defined to be
\[\rho_{\x}(\w_1,\w_2)=\left\{\|\w_1\|^2+\|\w_2\|^2-2\ll\w_1,\w_2\gg\right\}^{1/2}\]
(cf. \cite{IN}, p144). Note that, although in general $\ll\,\,,\,\,\gg$ differs from the usual Euclidean inner product $\langle\,\,,\,\,\rangle$, a geodesic triangle contained in the closure of a stratum of $\X^m$ is in fact a Euclidean geodesic triangle and its angles are the Euclidean ones. In particular, $\ll\w_1,\w_2\gg=\langle\w_1,\w_2\rangle$ for any $\w_1,\w_2$ in the closure of $\mathbb{R}(E)\times\O(F)$ and then $\rho_{\x}(\w_1,\w_2)=\|\w_1-\w_2\|$.} 

\section{The carriers and supports of geodesics}

In order to analyse the logarithm map, we first need to understand the geodesics. The intersection of a geodesic with a stratum, a Euclidean orthant, will be either a single point or a complete intersection of a Euclidean line with that orthant. 

\begin{definition}
The carrier of a geodesic is the sequence of strata each of whose intersection with the geodesic is a Euclidean line of positive length.
\label{def0e}
\end{definition}

\red{This is essentially the terminology that was introduced in \cite{KV}} in the context of tree spaces. The case of a single point intersection arises between successive strata of the carrier: between the (open) linear segment in one stratum and that in the next, there will be one point in the common bounding stratum of those two strata. This intermediate stratum is not listed in the carrier; it is in fact specified by the adjacent strata as the stratum of highest dimension in the intersection of their closures. Similarly, when a geodesic starts, or ends, in a stratum of positive co-dimension and does not remain in that stratum, but passes immediately to a co-bounding stratum, then the latter will be the first, or last, stratum in the carrier. In such a situation, we shall regard the point in the bounding stratum as having the same set of axes as the co-bounding stratum, albeit with the relevant coordinates zero. That is, we regard it as a point of the closure of the co-bounding stratum. 

\vskip 6pt
To describe the carrier of the geodesic from $\x_1$ to $\x_2$ in more detail, as well as for later analysis, we require the following terminology.  

\begin{definition}
$(i)$ The subsets $E$ and $F$ of $U$ are said to be compatible in the orthant space $\X^m$ if the orthant ${\cal O}(E\cup F)$ is contained in $\X^m$. 

\red{$(ii)$ For a subset $E$ of the standard orthonormal basis $U$ of $\mathbb R^M$, we denote the number of vectors in $E$ by $|E|$.} 
\label{def0d}
\end{definition}

\red{We first identify the set of axes common to the points along a geodesic where, for any $\x\in\X^m$, $E(\x)$ denotes the set of axes in $U$ with respect to which $\x$ has positive coordinates.}

\begin{proposition}
\red{For any $\x_1,\x_2\in\X^m$, the set $E(\x_1,\x_2)$, defined by}
\begin{eqnarray}
\begin{array}{rcl}
E(\x_1,\x_2)&=&\left\{E(\x_1)\cap E(\x_2)\right\}\\
&&\bigcup\,\{e\in E(\x_1)\mid e\hbox{ is compatible with }E(\x_2)\}\\
&&\bigcup\,\{e\in E(\x_2)\mid e\hbox{ is compatible with }E(\x_1)\},
\end{array}
\label{eqn0c}
\end{eqnarray}
\red{forms the set of axes common to all strata along the geodesic between $\x_1$ and $\x_2$.}
\label{prop0c}
\end{proposition}

\begin{proof}
Observe that, for a geodesic, each coordinate function must be linearly interpolated between any two values that are non-zero. It follows that once a particular coordinate, having been positive along the geodesic, becomes zero it must remain so or, having started at zero, once it becomes positive, it must continue monotonically to its final value. In particular, the only basis vectors that can occur with positive coordinate at any point along the geodesic from $\x_1$ to $\x_2$ are those that belong to $\x_1$ or $\x_2$ or to both. Moreover, \red{when} the geodesic from $\x_1$ \red{passes} immediately to a co-bounding stratum, \red{all the new axes in that stratum} must have coordinate zero at $\x_1$ increasing linearly along the geodesic to its value at $\x_2$. \red{Any such additional axis $e$ of the co-bounding stratum is in $E(\x_2)$} and is compatible with all the axes in $\red{E(\x_1)}$; and any such $e$ must occur in this way. \red{Thus, the set of axes common to all strata along the geodesic from $\x_1$ to $\x_2$ is precisely the given set $E(\x_1,\x_2)$.}
\end{proof}

\red{Note that, at one extreme, if $\x_1$ and $\x_2$ both lie in the closure of an orthant $\O(E)$ and not both in the same boundary component, then $E(\x_1,\x_2)=E$. At the other extreme, if $\overline{\O(E(\x_1))}\cap\overline{\O(E(\x_2))}=\emptyset$, then $E(\x_1,\x_2)=\emptyset$. In general, $E(\x_1,\x_2)$ depends only on the orthants in which $\x_1$ and $\x_2$ lie, and is independent of their positions in those orthants.}

\vskip 6pt
The number $k+1$ of orthants in the carrier $\mathcal{C}=(\mathcal{O}_0,\mathcal{O}_1,\cdots,\mathcal{O}_k)$ of the geodesic from $\x_1$ to $\x_2$ will, naturally, depend on both $\x_1$ and $\x_2$. If $\x_1$ lies in a top dimensional stratum it will have $m$ strictly positive coordinates, all of which, assuming that none are also positive in $\x_2$, must become zero somewhere along the geodesic and at least one must become zero on each change of stratum as they cannot vanish within a stratum of the carrier. Thus, there will be $m+1$ strata in the carrier, that is $k=m$, if and only if they vanish one at a time. So, $k<m$ if and only if somewhere along the geodesic at least two coordinates become zero on passing from $\mathcal{O}_i$ to $\mathcal{O}_{i+1}$. When \red{$|E(\x_1,\x_2)|=k_0$},  the maximum value of $k$ would now be $k'=m-k_0$. Similarly, if $\x_1$ were in a stratum of dimension $m_0$, this maximum would be $m_0-k_0$.

\red{From now on, for given $\x_1$ and $\x_2$, we shall denote the set $E(\x_1,\x_2)$ by both $A_0$ and $B_0$ to accord with the following notation. It follows from Proposition \ref{prop0c} that each member of the sequence of strata $\mathcal{C}=(\mathcal{O}_0,\mathcal{O}_1,\cdots,\mathcal{O}_k)$ that comprise the carrier of the geodesic $\gamma$ from $\x_1$ to $\x_2$ has $\mathcal{O}(A_0)=\mathcal O(B_0)$ as a factor. The carrier of $\gamma$ determines further subsets of axes forming} two sequences $(A_1,\cdots,A_k)$ and $(B_1,\cdots,B_k)$, where $A_i$ is the set of all the axes whose coordinates become zero and $B_i$ the set of all those whose coordinates become positive as the geodesic passes from $\mathcal{O}_{i-1}$ to $\mathcal{O}_i$. Thus, the stratum $\mathcal{O}_{i-1}$ is $\O(B_0\cup B_1\cup\cdots\cup B_{i-1}\cup A_i\cup\cdots\cup A_k)$ \red{and $\mathcal{O}_i=\O(B_0\cup B_1\cup\cdots\cup B_i\cup A_{i+1}\cup\cdots\cup A_k)$}, with $\mathcal{O}_0$ determined by $A_0\cup A_1\cup\cdots\cup A_k$. Clearly, the intermediate stratum between $\O_{i-1}$ and $\O_i$, their common boundary component, is 
\[\O(B_0\cup B_1\cup\cdots\cup B_{i-1}\cup A_{i+1}\cup\cdots\cup A_k).\]
Thus, in particular, 
\begin{enumerate}
\item[$(a)$]
\textit{the sets $B_i$ and $A_j$ of axes are non-empty for all \red{positive} $i$ and $j$, and compatible in $\X^m$ for $\red{0\leqslant} i<j$}; 
\item[$(b)$] $\gamma$ \textit{passes successively with positive length through the orthants $\O_i$ except that it may meet at most one of $\O_0$ and $\O_k$ in a single point};
\item[$(c)$]
\textit{$A_i\cap A_j=\emptyset$ and $B_i\cap B_j=\emptyset$ for all} $i\not=j$. 
\end{enumerate}
The property $(c)$ follows from the facts that $A_1\cup\cdots\cup A_k$ is disjoint from $B_1\cup\cdots\cup B_k$ and that an axis once removed cannot be removed again, or once introduced cannot be introduced again.

\begin{definition}
\red{For any two points $\x_1$ and $\x_2$ in $\X^m$, the support of the geodesic $\gamma$ from $\x_1$ to $\x_2$ is defined to be the pair $(\mathcal{A},\mathcal{B})$ of sequences of sets of axes,
\[\mathcal{A}=(A_0,A_1,\cdots,A_k)\quad\hbox{ and }\quad\mathcal{B}=(B_0,B_1,\cdots,B_k),\]
where $\gamma$ passes successively through the orthants
\begin{eqnarray}
\mathcal{O}_i=\mathcal{O}(A_0\cup B_1\cup\cdots\cup B_i\cup A_{i+1}\cup\cdots\cup A_k),\qquad i=0,1,\cdots,k,
\label{eqn0b}
\end{eqnarray}
that form the carrier of $\gamma$.}
\label{def0f}
\end{definition}

\red{In the context of tree spaces, the definition of the support of a geodesic given here is equivalent to that of the minimal support given in \cite{MOP}.}

\begin{example}
\red{For a geodesic passing successively through the orthants} 
\begin{eqnarray*}
\O_0&=&\O(e_0,e_1,e_2,e_3,e_4,e_5,e_6),\\
\O_1&=&\O(e_0,e_1,f_2,e_3,e_4,e_5,e_6),\\
\O_2&=&\O(e_0,e_1,f_2,f_3,e_5,e_6),\\
\O_3&=&\O(e_0,e_1,f_2,f_3,f_4,e_6),\\
\O_4&=&\O(e_0,e_1,f_2,f_3,f_4,f_5,f_6),
\end{eqnarray*}
\red{the relevant sequences $\mathcal{A}=(A_0,A_1,\cdots,A_4)$ and $\mathcal{B}=(B_0,B_1,\cdots,B_4)$ forming the support would have members $A_0=B_0=\{e_0,e_1\}$, the basis vectors common to all five orthants; $A_1=\{e_2\}$, $B_1=\{f_2\}$; $A_2=\{e_3,e_4\}$, $B_2=\{f_3\}$; $A_3=\{e_5\}$, $B_3=\{f_4\}$; $A_4=\{e_6\}$ and $B_4=\{f_5,f_6\}$.} 
\end{example}

\red{If both $\x_1$ and $\x_2$ lie in the closure of the same orthant, then the geodesic between them is clearly the Euclidean line segment. To understand geodesics in general and, later, to describe and analyse various properties of the logarithm map, we require the orthogonal projections onto the various strata of $\X^m$, where the orthogonality is with respect to the Euclidean inner product on $\mathbb R^M$.} 

\begin{definition}
\red{For $\x\in\X^m$ and $E\subset U$ such that the orthant $\sigma=\O(E)$ is contained in $\X^m$, $P_E(\x)$ denotes the orthogonal projection of $\x$ onto $\O(E)$, that is the vector, or when relevant its coordinate vector, formed by the components of $\x$ in the directions of the unit vectors in $E$.}
\label{def0g}
\end{definition}

\red{In terms of projections, we have the following characterisation of the supports of geodesics when $\x_1$ and $\x_2$ do not lie in the closure of the same orthant.}

\begin{proposition}
\red{Let $\x_1$ and $\x_2$ be two given points in $\X^m$. Suppose that $\mathcal A=(A_0,A_1,\cdots,A_k)$ and $\mathcal B=(B_0,B_1,\cdots,B_k)$ are two sequences of sets of axes such that the $\O_i$ defined by \eqref{eqn0b} are all contained in $\X^m$, where $k>0$ and where all subsets $A_i$ and $B_j$ are mutually disjoint and non-empty, except for $A_0=B_0=E(\x_1,\x_2)$ which may be empty. Then, $(\mathcal{A},\mathcal{B})$ is the support of the geodesic from $\x_1$ to $\x_2$ if and only if} 
\begin{itemize}
\item[$(i)$] \red{for $k>1$ and for all $0<i<k$,} 
\begin{eqnarray}
\red{\frac{\|P_{A_i}(\x_1)\|}{\|P_{B_i}(\x_2)\|}<\frac{\|P_{A_{i+1}}(\x_1)\|}{\|P_{B_{i+1}}(\x_2)\|};}
\label{eqn2g}
\end{eqnarray}
\item[$(ii)$] \red{for all $0<i\leqslant k$ and all non-trivial partitions $C_{i1}\cup C_{i2}$ for $A_i$ and $D_{i1}\cup D_{i2}$ for $B_i$, if the orthant 
\begin{eqnarray}
\O'=\O(B_0\cup B_1\cup\cdots\cup B_{i-1}\cup D_{i1}\cup C_{i2}\cup A_{i+1}\cup\cdots\cup A_k)
\label{eqn0h}
\end{eqnarray} 
is contained in $\X^m$, then
\begin{eqnarray}
\frac{\|P_{C_{i1}}(\x_1)\|}{\|P_{D_{i1}}(\x_2)\|}\geqslant\frac{\|P_{C_{i2}}(\x_1)\|}{\|P_{D_{i2}}(\x_2)\|}.
\label{eqn2f}
\end{eqnarray}}
\end{itemize}
\label{prop0a}
\end{proposition}

\red{Compared with the result in \cite{OP} (Theorem 2.5) in the case of tree spaces, this result confirms the claim in Section 6 of \cite{MOP} that the results on tree spaces also hold for orthant spaces. However, the condition $(ii)$ above is necessarily stronger than that there. This is due to the fact that, in general orthant spaces, the condition that $C_{i2}$ is compatible with $D_{i1}$ does not necessarily guarantee that the orthant $\O'$ given by \eqref{eqn0h} is contained in $\X^m$.}

\begin{proof}
\red{Assuming that $(\mathcal A,\mathcal B)$ is the support of the geodesic from $\x_1$ to $\x_2$,} we focus on three consecutive strata of the carrier, $\mathcal{O}_{i-1},\mathcal{O}_i,\mathcal{O}_{i+1}$, where $\O_j$ are as defined in \eqref{eqn0b}, projecting the geodesic onto the subspace $\mathbb R(A_i\cup A_{i+1}\cup B_i\cup B_{i+1})$. As the geodesic passes from $\mathcal{O}_{i-1}$ to $\mathcal{O}_i$, the coordinates along the axes in $A_i$ become zero and those in $B_i$ start to grow. Then, on passing from $\mathcal{O}_i$ to $\mathcal{O}_{i+1}$, the coordinates of axes in $A_{i+1}$ become zero and those in $B_{i+1}$ grow. Consider the projection of the geodesic onto the three planar quadrants $\Pi_{i-1}$ determined by the vectors \red{$P_{A_i}(\x_1)$ and $P_{A_{i+1}}(\x_1)$, $\Pi_i$ determined by $P_{A_{i+1}}(\x_1)$ and $P_{B_i}(\x_2)$ and $\Pi_{i+1}$ determined by $P_{B_i}(\x_2)$ and $P_{B_{i+1}}(\x_2)$} as in Figure \ref{fig2g}. 
\begin{figure}
\begin{center}
\begin{tikzpicture}[scale=1]
\draw [<->] (-1.9,0) -- (1.6,0);
\draw [<->] (0,-1) -- (0,1.8);
\node at (0,2.1) {$\scriptsize{{P_{\scriptstyle{A_{i+1}}}(\x_1)}}$};
\node at (0,-1.5) {$\scriptsize{P_{B_{i+1}}(\x_2)}$};
\node at (2.3,0) {$\scriptsize{P_{A_i}(\x_1)}$};
\node at (-2.6,0) {$\scriptsize{P_{B_i}(\x_2)}$};
\node at (-1.2,0.8) {$\Pi_i$};
\node at (1.3,0.8) {$\Pi_{i-1}$};
\node at (-1.2,-1.15) {$\Pi_{i+1}$};
\draw [dashed, red] (1.5,1.7) -- (0,0) -- (-1.8,-0.9);
\draw [red] (1.5,1.7) -- (-1.8,-0.9);
\fill [red] (1.5,1.7) circle (1.5pt) node[above,right, red] {$\scriptsize{p(\x_1)}$};
\fill [red] (-1.8,-0.9) circle (1.5pt) node[left, red] {$\scriptsize{p(\x_2)}$};
\draw [red] (0.4,0) arc [radius=0.45, start angle=0, end angle= 40]; 
\node [red] at (0.6,0.2) {$\scriptsize{\theta}$};
\draw [red] (-0.35,0) arc [radius=0.7, start angle=180, end angle=195];
\node [red] at (-0.6,-0.15) {$\scriptsize{\phi}$};
\end{tikzpicture}
\end{center}
\caption{Projection of the geodesic}
\label{fig2g}
\end{figure}
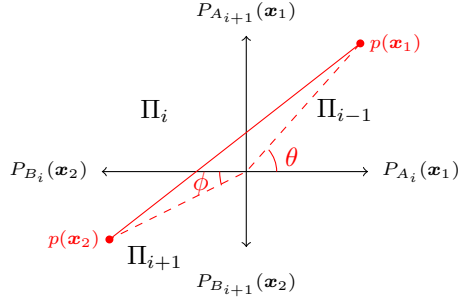
This is an isometric representation of the relevant quadrants except that, in $\mathbb{R}^M$, all four vectors are mutually orthogonal. Then,  $\O_i$ is in the carrier if and only if the projection of the geodesic passes through \red{the interior of} $\Pi_i$. That is if, and only if, the angle $\theta$ that the vector $p(\x_1)=P_{A_i\cup A_{i+1}}(\x_1)$ makes with \red{$P_{A_i}(\x_1)$} in $\Pi_{i-1}$ is greater than the angle $\phi$ that $p(\x_2)=P_{B_i\cup B_{i+1}}(\x_2)$ makes with 
\red{$P_{B_i}(\x_2)$} in $\Pi_{i+1}$, as expressed by \eqref{eqn2g}.

\red{Similarly, if $\O'$ is contained in $\X^m$, the failure of \eqref{eqn2f} would ensure that the geodesic passed through $\O'$, with positive length, between $\O_{i-1}$ and $\O_i$.}
%\red{Similarly, consider the projection of the geodesic onto the three planar quadrants $\Pi_{i-1}'$ determined by $P_{C_{i1}}(\x_1)$ and $P_{C_{i2}}(\x_1)$, $\Pi_i'$ determined by $P_{C_{i2}}(\x_1)$ and $P_{D_{i1}}(\x_2)$, and $\Pi_{i+1}'$ determined by $P_{D_{i1}}(\x_2)$ and $P_{D_{i2}}(\x_2)$. The failure of the inequality \eqref{eqn2f} would imply that the analogous angles $\theta$ and $\phi$ satisfy $\theta<\phi$, so that the straight line segment joining $P_{A_i}(\x_1)$ and $P_{B_i}(\x_2)$ would be shorter than the projection of the geodesic. Since $\O'$ is in $\X^m$, this linear segment would be the projection of the piecewise linear curve in $\X^m$ joining $\x_1$ and $\x_2$ obtained by changing the coordinates in $\mathbb R(A_i\cup B_i)$ of the geodesic, according to this linear segment, while leaving all other coordinates the same. This resulting curve would pass through the orthant $\O'$ between $\O_{i-1}$ and $\O_i$ and be shorter than the geodesic. Hence, \eqref{eqn2f} must hold.}

\red{To show that conditions $(i)$ and $(ii)$ determine the support of the geodesic from $\x_1$ to $\x_2$, we first note that, as seen above, $(i)$ ensures that the geodesic must pass through the orthant $\O_i$ between $\O_{i-1}$ and $\O_{i+1}$. Since $\X^m$ is a cone, it is simply connected and any piecewise linear path from $\x_1$ to $\x_2$ can be transformed by homotopy to a geodesic by a sequence of `simple moves' whereby, for each move, two consecutive linear segments of the path are replaced by a single linear segment. Since the geodesic is linear within orthants, that can only occur between consecutive orthants and
%This results in an extra orthant being inserted in the carrier. As above, that occurs precisely when $(ii)$ fails. Thus, 
condition $(ii)$ guarantees that there is no extra orthant in the carrier between $\O_{i-1}$ and $\O_i$.}
\end{proof} 

\red{As noted previously, if $\x_1$ and $\x_2$ both lie in the closure of an orthant, then $k=0$ and the geodesic between $\x_1$ and $\x_2$ is always a Euclidean segment. Then, when $\x_2$ varies within the orthant in which it lies, the support of the geodesic from $\x_1$ to $\x_2$ remains the same. However, in general, the support may change. The above characterisation of the support of a geodesic implies the following sufficient condition for the support to remain locally constant.}

\begin{corollary}
\red{Suppose that the hypotheses of Proposition $\ref{prop0a}$ hold. If, for all $0<i\leqslant k$ and for all relevant partitions of $A_i$ and $B_i$ as in $(ii)$ of that proposition, the inequality \eqref{eqn2f} is strict then,  for all $\x$ in a sufficiently small neighbourhood of $\x_2$ in its stratum, $(\mathcal A,\mathcal B)$ remains the support for the geodesic from $\x_1$ to $\x$.}
\label{cor0d}
\end{corollary}

\begin{proof}
\red{Since $\x$ varies within the stratum in which $\x_2$ lies, the set $A_0=B_0$ remains unchanged. For the other sets in the support, by continuity, the strict inequalities \eqref{eqn2g} and, we are assuming, \eqref{eqn2f} continue to hold for $\x$ in a sufficiently small neighbourhood of $\x_2$ within its stratum. Hence, the required result follows from Proposition \ref{prop0a}.}
\end{proof}

\section{The logarithm map}  

Analogous to an inverse of the exponential map on a Riemannian manifold, the logarithm map on $\X^m$ is defined as follows.

\begin{definition}
The logarithm map at $\x^*\in\X^m$ is the map $\log_{\x^*}(\x)$ from $\X^m$ to the tangent cone to $\X^m$ at $\x^*$, the image of $\x$ being the initial tangent vector, with norm $d(\x^*,\x)$, to the geodesic from $\x^*$ to $\x$. 
\end{definition}

The logarithm map is globally well-defined since, as already mentioned, the Cartan-Hadamard theorem implies that there is a unique geodesic between any two points $\x^*$ and $\x$ of $\X^m$. If that geodesic has an initial segment in a stratum containing $\x^*$ it will certainly have an initial tangent vector. If it has only $\x^*$ in the initial stratum, it must then have an open segment $\gamma(0,\epsilon)$, with $\gamma(0)=\x^*$, in a co-bounding stratum. Then it will still have a one-sided derivative at $\x^*$ which suffices to define the logarithm map. 

\vskip 6pt
\red{With the description of the carrier, as well as the results on the support, of a geodesic in the previous section, we are now in a position to derive and analyse its initial tangent vector, or equivalently $\log_{\x^*}(\x)$.} As in \cite{BLO1} and \cite{BLO2} for the space of trees, \red{our analysis will mainly involve} a modified version of the logarithm map. For this, since the tangent cones at various points in $\sigma$ are all parallel, we may parallel translate them to the cone point $o$, the origin in $\mathbb{R}^M$, to produce a common isometric copy $\CC_\sigma$. Then, since the coordinate vector of the point $\x^*$, which we also denote by $\x^*$, lies in the common factor $\mathbb{R}(E)$ of all the strata of $\CC_\sigma$, it makes sense to add it to $\log_{\x^*}(\x)$ and the result
\[\Phi(\x;\x^*)=\log_{\x^*}(\x)+\x^*\]
will also lie in $\CC_\sigma$. We shall refer to $\Phi$ as the \red{\textit{translated logarithm map} to distinguish it from the logarithm map itself. All the vectors $\Phi(\x;\x^*)$ being in the same space implies that the translated logarithm maps are directly comparable as $\x^*$ varies within an orthant and such comparability will be necessary later. Moreover, the difference between the two maps is such that all our analysis of $\Phi$ can easily be translated to that of the logarithm map itself.} 

Note that, although the origin corresponds to the cone point $o$ of the orthant space $\X^m$, $\CC_\sigma$ is not the tangent cone to $\X^m$ at $o$, neither being contained in the other, unless $\sigma=\{o\}$. \red{Note also that, when $\X^m$ is a tree space and $\x^*$ lies in a top-dimensional stratum, $\Phi(\x;\x^*)$ was called the modified logarithm map and was denoted by $\Phi_{\x^*}(\x)$ in \cite{BLO2}, and the permutation map $\pi$ there corresponds to the linear transformation $\jmath$ given by Definition $\ref{def0i}$.}

\red{The next theorem gives the expression for the translated logarithm map $\Phi(\,\cdot\,;\x^*)$ in terms of the projections, specified in Definition \ref{def0g}, onto various sets of axes appearing in the support of the geodesic from $\x^*$ to $\x$.} 

\begin{theorem} 
For any two points $\x^*$ and $\x$ in $\X^m$, let the sequences $\mathcal{A}=(A_0,\cdots,A_k)$ and $\mathcal{B}=(B_0,\cdots,B_k)$ of sets of axes form the support of the geodesic from $\x^*$ to $\x$. Then, the \red{translated} logarithm map $\Phi(\,\cdot\,;\x^*)$ at $\x^*$ is given by
\begin{eqnarray}
\red{\Phi(\x;\x^*)=\jmath\left(P_{B_0}(\x),-\frac{\|P_{B_1}(\x)\|}{\|P_{A_1}(\x^*)\|}P_{A_1}(\x^*),\cdots,-\frac{\|P_{B_k}(\x)\|}{\|P_{A_k}(\x^*)\|}P_{A_k}(\x^*)\right),}
\label{eqn0g}
\end{eqnarray}
\red{where $\jmath$ is the linear transformation given by Definition $\ref{def0i}$.}
\label{thm1}
\end{theorem}

In particular, $\Phi(\,\cdot\,;\lambda\x^*)=\Phi(\,\cdot\,;\x^*)$ for any constant $\lambda>0$.

\vskip 6pt
\red{Recall that, if $k=0$, then $\x$ and $\x^*$ lie the closure of an orthant and the geodesic from $\x^*$ to $\x$ is a line segment in $\mathbb R^M$. In this case, $\Phi(\x;\x^*)=\jmath(\x)$. If $k>0$ and if $|A_i|=|B_i|=1$ for $1\leqslant i\leqslant k$, the form of the expression for $\Phi(\x;\x^*)$ is also similar to that of the corresponding translated logarithm map in a Euclidean space, after changing the axes $B_i$ to $-A_i$.} 

\begin{proof}
The orthogonal projection of $\gamma$ onto $\mathcal{O}(A_0)$ determines the component of the initial tangent vector to $\gamma$ that is tangent to $\mathcal{O}(A_0)$, namely 
\begin{eqnarray}
\red{v_0=P_{B_0}(\x)-P_{A_0}(\x^*)}\in\mathbb{R}(A_0).
\label{eqn0}
\end{eqnarray}
For the remaining coordinates, since the sets $A_i$ and $B_j$ above are all mutually disjoint, it follows that, for each $i$, the subspace $\mathbb{R}(A_i\cup B_i)$ is orthogonal to all $\mathbb{R}(A_j)$ and $\mathbb{R}(B_j)$ for $j\not=i$, so that the coordinates of the geodesic $\gamma$ that are positive with respect to the axes in $\mathbb{R}(A_i\cup B_i)$ are just those of the projection $\gamma_i$ of $\gamma$ onto that subspace.  If $s_i$ is the parameter such that $\gamma(s_i)\in\mathcal{O}_{i-1}\cap\mathcal{O}_i$, then $\red{P_{A_i}(\gamma(s))}\in\O(A_i)$ declines linearly from $\red{P_{A_i}(\gamma(0))=P_{A_i}(\x^*)}$ to $\red{P_{A_i}(\gamma(s_i))}=\bf{0}$. Then, the coordinates $\red{P_{B_i}(\gamma(s))}\in\O(B_i)$ increase linearly from zero at $\gamma(s_i)$ to $\red{P_{B_i}(\gamma(1))=P_{B_i}(\x)}$. Thus, the projected geodesic $\gamma_i$ \red{lies in the union of the orthogonal orthants $\O(A_i)$ and $\O(B_i)$ and hence} has length $\red{\|P_{A_i}(\x^*)\|+\|P_{B_i}(\x)\|}$. The initial tangent vector to $\gamma_i$ is parallel to \red{$-P_{A_i}(\x)$ and so is}
\[v_i=-\red{\frac{\|P_{A_i}(\x^*)\|+\|P_{B_i}(\x)\|}{\|P_{A_i}(\x^*)\|}P_{A_i}(\x^*)}.
\]
Hence, the initial tangent vector to $\gamma$ with norm $d(\x^*,\x)$ is represented by $(v_0,v_1,\cdots,v_k)$. However, this ordering of the coordinates, with those in $\mathbb{R}(A_i)$ preceding those of $\mathbb{R}(A_{i+1})$ for each $i$, requires \red{the linear transformation $\jmath$ to} obtain its representation with respect to the standard basis in $\mathbb{R}^M$. Then, the logarithm map at $\x^*$ will be
\[\log_{\x^*}:\x\mapsto\jmath\left(v_0,v_1,\cdots,v_k\right),\]
so that equation \eqref{eqn0g} follows from the coordinates $v_i$ since the coordinates of $\x^*$ are $\jmath\left(\red{P_{A_0}(\x^*),P_{A_1}(\x^*),\cdots,P_{A_k}(\x^*)}\right)$.
\end{proof}

\red{In the following, when we say that the expression \eqref{eqn0g} for $\Phi(\x_2;\x^*)$ takes the same form as the corresponding expression for $\Phi(\x_1;\x^*)$, we mean that the expression for $\Phi(\x_2;\x^*)$ can be obtained by replacing $\x_1$ by $\x_2$ in the expression \eqref{eqn0g} for $\Phi(\x_1;\x^*)$. Clearly, the form of the expression for $\Phi(\x;\x^*)$ will depend on the support $(\mathcal A,\mathcal B)$ of the geodesic from $\x^*$ to $\x$, noting that the roles that $\mathcal A$ and $\mathcal B$ play are not symmetric. The following example illustrates this feature where, although $\x$ lies in the same orthant in the second and third cases, the forms for $\Phi(\x;\x^*)$, as a function of $\x$, differ in the two cases. However, along the boundary between the light and dark grey regions, the two forms give the same result.}

\begin{example}\label{ex2}
\red{Consider $\X^2$ in $\mathbb R^5$, which was called $Q_5$ in \cite{BLO1}, consisting of five orthants as shown in Figure \ref{fig1}, where all five axes are mutually orthogonal.}
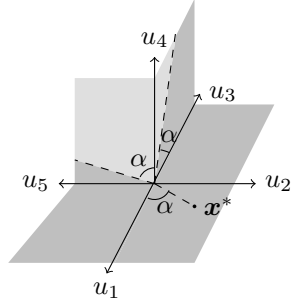
\begin{figure}
\begin{center}
\begin{tikzpicture} [scale=0.7]
\path[fill=gray!50] (0.5,0) -- (1.25,1.5) -- (1.25,3.5) -- (0.9,2.8) -- cycle;
\path[fill=gray!25] (0.5,0) -- (0.9,2.8) -- (0.5,2) -- cycle;
\path[fill=gray!25] (0.5,0) -- (0.5,2) -- (-1,2) -- (-1,0.45) -- cycle;
\path[fill=gray!50] (0.5,0) -- (-1,0.45) -- (-1,0) -- cycle;
\path[fill=gray!50] (0.5,0) -- (1.25,1.5) -- (2.75,1.5) -- (2,0) -- cycle;
\path[fill=gray!50] (0.5,0) -- (-0.25,-1.5) -- (1.25,-1.5) -- (2,0) -- cycle;
\path[fill=gray!50] (0.5,0) -- (-0.25,-1.5) -- (-2.25,-1.5) -- (-1,0) -- cycle;
\draw[<->] (-0.4,-1.7) node[below] {$u_1$} -- (1.35,1.7) node [above,right] {$u_3$};
\draw[<->] (-1.3,0) node [left] {$u_5$}  -- (2.4,0) node[right] {$u_2$};
\draw[->] (0.5,0) -- (0.5,2.4) node[above] {$u_4$};
\path[draw, dashed] (0.5,0) -- (1.25,-0.43);
\fill (1.25,-0.43) circle (1pt) node[below, right] {$\x^*$};
\path[draw, dashed] (0.5,0) -- (0.9,2.9);
\path[draw, dashed] (0.5,0) --  (-1,0.45);
\draw(0.5,0)+(245:0.3) arc (245:335:0.3) node at (0.7,-0.45) {$\alpha$};
\draw(0.5,0)+(90:0.3) arc (90:155:0.3) node at (0.2,0.45) {$\alpha$};
\draw(0.5,0)+(65:0.65) arc (65:80:0.65) node at (0.78,0.9) {$\alpha$};
\end{tikzpicture}
\end{center}
\caption{\red{$\X^2$ in $\mathbb R^5$ consisting of five orthants}}
\label{fig1}
\end{figure}
\red{The tangent cone to $\X^2$ at $\x^*=(x_1^*,x_2^*,0,0,0)$ indicated in Figure \ref{fig1} is the $(u_1,u_2)$-plane and that at the cone point $o$ is $\X^2$ itself. While $\Phi(\x;o)=\x$ for all $\x$, the expression for $\Phi(\x;\x^*)$ takes different forms depending the position of $\x$. For example, for any $\x=(0,x_2,x_3,0,0)$ in the orthant $\O(u_2,u_3)$,
\[\Phi(\x;\x^*)=(-x_3,x_2,0,0,0);\]
for $\x=(0,0,x_3,x_4,0)$ in the dark grey region of $\O(u_3,u_4)$, i.e. if the coordinates of $\x$ satisfy $x_4/x_3<\tan(\alpha)=x^*_2/x^*_1$, then 
\[\Phi(\x,\x^*)=(-x_3,-x_4,0,0,0);\]
However, if $\x=(0,0,x_3,x_4,0)$ lies in the light grey region of $\O(u_3,u_4)$, i.e. if the coordinates of $\x$ satisfy $x_4/x_3>\tan(\alpha)=x^*_2/x^*_1$, then 
\[\Phi(\x,\x^*)=-\frac{\|\x\|}{\|\x^*\|}(x^*_1,x^*_2,0,0,0).\] 
}
\red{In particular, for all $\x$ in the light grey region of $\X^2$, the vectors $\Phi(\x;\x^*)$ have the same direction $-\frac{1}{\|\x^*\|}(x^*_1,x^*_2,0,0,0)$ and the only difference between them lies in the length of this vector.}
\end{example}

The potential variation of the form of the expression \eqref{eqn0g} for the \red{translated} logarithm map, arising from the changes in the supports of the geodesics, is one of the main obstructions to generalising the theory for manifolds to orthant spaces, or more general stratified spaces. \red{To study this variation, we first note the following result, which is a direct consequence of Corollary \ref{cor0d}.}

\begin{corollary}
\red{If the support of the geodesic from $\x^*$ to $\x$ satisfies the conditions of Corollary $\ref{cor0d}$, then there is a neighbourhood $\mathcal N$ of $\x$ within its stratum such that, for any $\x'\in\mathcal N$, the form of the expression \eqref{eqn0g} for $\Phi(\x';\x^*)$ takes the same form of that for $\Phi(\x;\x^*)$.}
\label{cor0c}
\end{corollary}

\red{We now characterise, in terms of the two conditions on the support of a geodesic given in Proposition \ref{prop0a}, changes in the form of the expression \eqref{eqn0g} for $\Phi(\x;\x^*)$ when $\x$ varies locally. Although the roles played by these two conditions in determining the support of a geodesic are different, to some extent, they play a similar role in the change of the form of that expression. Replacing the inequality \eqref{eqn2g} or \eqref{eqn2f} by equality determines a quadratic co-dimension one hyper-surface. When two or more such hyper-surfaces meet, their normals are linearly independent so they intersect in surfaces of co-dimension at least two. Thus, it suffices to consider a point lying in a single such hyper-surface. Then, points on either side of the hyper-surface will have different supports for their geodesics from $\x^*$, but that will not always result in a change in the form the expression for $\Phi(\x;\x^*)$.}

\begin{proposition}
\red{Let $\x^*$ and $\x_0$ be two given points in $\X^m$, and let $(\mathcal{A},\mathcal{B})$ be the support of the geodesic from $\x^*$ to $\x_0$, where $\mathcal A=(A_0,A_1,\cdots,A_k)$ and $\mathcal B=(B_0,B_1,\cdots,B_k)$, and where $k>1$. Assume that $\x$ moves from $\x_0$, within its stratum, to a first point $\x_1$ such that, for $i=i_0>0$, the inequality \eqref{eqn2g}, with $\x_1,\x_2$ replaced by $\x^*,\x_1$ respectively, becomes an equality while all the other inequalities \eqref{eqn2g} and \eqref{eqn2f} remain strict. Then, the support $(\mathcal A',\mathcal B')$ of the geodesic from $\x^*$ to $\x_1$ has
\[\mathcal A'=(A_0,A_1,\cdots,A_{i_0-1},A_{i_0}\cup A_{i_0+1},A_{i_0+2},\cdots,A_k)\]
and similarly for $\mathcal B'$.}

\red{If the orthant}
\begin{eqnarray*}
\red{\O''=\O(B_0\cup B_1\cup\cdots\cup B_{i_0-1}\cup B_{i_0+1}\cup A_{i_0}\cup A_{i_0+2}\cup\cdots\cup A_k)}
%\label{eqn2b}
\end{eqnarray*}   
\red{is contained in $\X^m$, then there is a neighbourhood $\mathcal N$ of $\x_1$ within its stratum such that, for all $\x\in\mathcal N$, the form of the expression \eqref{eqn0g} for $\Phi(\x;\x^*)$ is identical with that for $\Phi(\x_0;\x^*)$.}

\red{If $\O''$ is not an orthant of $\X^m$ then, in any neighbourhood $\mathcal N$ of $\x_1$ within its stratum, there are $\x'$ and $\x''$ such that the form for $\Phi(\x';\x^*)$ is the same as that for $\Phi(\x_0;\x^*)$ and that for $\Phi(\x'';\x^*)$ is determined by the support $(\mathcal A',\mathcal B')$. When $\mathcal N$ is sufficiently small, there are no other possibilities.}
\label{prop0d}
\end{proposition}

\begin{proof}
\red{By Corollary \ref{cor0c}, the form of the expression \eqref{eqn0g} will remain constant, as long as the inequalities \eqref{eqn2g} and \eqref{eqn2f} remain strict. However, for $\x=\x_1$, on account of the equality \eqref{eqn2g} for $i_0$ at $\x=\x_1$, the angles $\theta$ and $\phi$, in the projected diagram of Figure \ref{fig2f}, will be equal where the projections are as specified in the proof of Proposition \ref{prop0a}.} 
\begin{figure}
\begin{center}
\begin{tikzpicture}[scale=0.7]
\draw [<->] (-4.4,0) -- (-0.9,0);
\draw [<->] (-2.5,-1.1) -- (-2.5,1.5);
\node at (-2.5,1.8) {$\scriptsize{P_{A_{i_0+1}}(\x^*)}$};
\node at (-2,-1.5) {$\scriptsize{P_{B_{i_0+1}}(\x_0)}$};
\node at (0,0) {$\scriptsize{P_{A_{i_0}}(\x^*)}$};
\node at (-5.4,0) {$\scriptsize{P_{B_{i_0}}(\x_0)}$};
%\node at (-3.7,0.8) {$\Pi_{i_0}$};
%\node at (-1.2,0.8) {$\Pi_{i_0-1}$};
%\node at (-3.7,-1.15) {$\Pi_{i_0+1}$};
%\node at (-1.5,-0.8) {$\Pi_0$};
\draw [red] (-1,1.4) -- (-2.5,0);
\draw [dashed, red] (-1,1.4) -- (-4.3,-1);
\draw [red] (-2.5,0) -- (-4,-1.4);
\fill [red] (-4,-1.4) circle (1.5pt);
\node [red] at (-4.5,-1.6) {$\scriptsize{p(\x_1)}$};
\fill [red] (-1,1.4) circle (1.5pt) node[above,right, red] {$\scriptsize{p(\x^*)}$};
\fill [red] (-4.3,-1) circle (1.5pt) node[left, red] {$\scriptsize{p(\x_0)}$};
\fill [red] (-3.5,-1.5) circle (1.5pt) node[left,below, red] {$\scriptsize{p(\x_2)}$};
\draw [dashed, red] (-1,1.4) -- (-3.5,-1.5);
\draw [red] (-2.1,0) arc [radius=0.5, start angle=0, end angle= 35]; 
\node [red] at (-1.85,0.15) {$\scriptsize\theta$};
\draw [red] (-2.85,0) arc [radius=0.6, start angle=180, end angle=205];
\node [red] at (-3,-0.2) {$\scriptsize\phi$};
\node at (-2.5,-2.8) {$(a)$};
\path[fill=gray!30] (5.5,0) -- (8,0) -- (8,-2) -- (5.5,-2) -- cycle;
\draw [<->] (3.6,0) -- (7.1,0);
\draw [<->] (5.5,-1.1) -- (5.5,1.5);
\node at (5.5,1.8) {$\scriptsize{P_{A_{i_0+1}}(\x^*)}$};
\node at (6,-1.5) {$\scriptsize{P_{B_{i_0+1}}(\x_0)}$};
\node at (8,0) {$\scriptsize{P_{A_{i_0}}(\x^*)}$};
\node at (2.6,0) {$\scriptsize{P_{B_{i_0}}(\x_0)}$};
\draw [red] (7,1.4) -- (5.5,0);
\draw [dashed, red] (7,1.4) -- (3.7,-1);
\draw [red] (5.5,0) -- (4,-1.4);
\fill [red] (4,-1.4) circle (1.5pt);
\node [red] at (3.5,-1.6) {$\scriptsize{p(\x_1)}$};
\fill [red] (7,1.4) circle (1.5pt) node[above,right, red] {$\scriptsize{p(\x^*)}$};
\fill [red] (3.7,-1) circle (1.5pt) node[left, red] {$\scriptsize{p(\x_0)}$};
\fill [red] (4.5,-1.5) circle (1.5pt) node[left,below, red] {$\scriptsize{p(\x_2)}$};
\draw [dashed, red] (5.5,0) -- (4.5,-1.5);
\draw [red] (5.9,0) arc [radius=0.5, start angle=0, end angle= 35]; 
\node [red] at (6.15,0.15) {$\scriptsize{\theta}$};
\draw [red] (5.15,0) arc [radius=0.6, start angle=180, end angle=205];
\node [red] at (5,-0.2) {$\scriptsize{\phi}$};
\node at (5.5,-2.8) {$(b)$};
\end{tikzpicture}
\end{center}
\caption{Change of carrier}
\label{fig2f}
\end{figure}
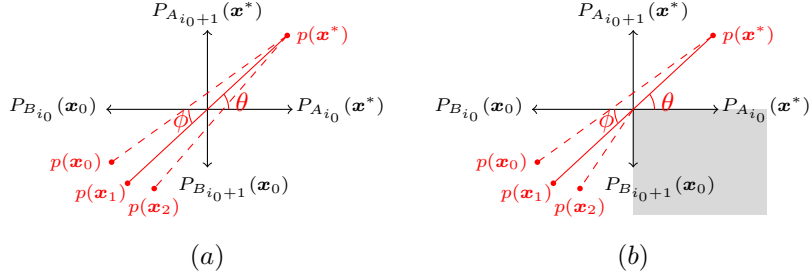
\red{Consequently at $\x_1$, $\O_{i_0}$ will drop out of the carrier, where $\O_i$ is defined by \eqref{eqn0b}, and, by the continuity of geodesics, the support of the geodesic from $\x^*$ to $\x_1$ will be $(\mathcal A',\mathcal B')$.}

\red{Now, let $\x$ continue to move past $\x_1$ to $\x_2$, remaining sufficiently close to $\x_1$ and having projection $p(\x_2)=P_{B_{i_0}\cup B_{i_0+1}}(\x_2)$ in Figure \ref{fig2f} lying on the opposite side to $p(\x_0)=P_{B_{i_0}\cup B_{i_0+1}}(\x_0)$ of the ray from the origin to $p(\x_1)=P_{B_{i_0}\cup B_{i_0+1}}(\x_1)$. If $\O''$ is contained in $\X^m$, the projection of the geodesic from $\x^*$ to $\x_2$ would be, as in Figure \ref{fig2f}$(a)$, the `straight' line from $p(\x^*)=P_{A_{i_0}\cup A_{i_0+1}}(\x^*)$ to $p(\x_2)$ passing through the planar quadrant $\Pi_0$ determined by $\P_{A_{i_0}}(\x^*)$ and $P_{B_{i_0}}(\x_0)$. This would imply replacing $\O_{i_0}$ in the carrier by $\O''$ with the resulting support for the geodesic from $\x^*$ to $\x_2$ being $(\mathcal A'',\mathcal B'')$, where 
\[\mathcal A''=(A_0,A_1,\cdots,A_{i_0-1},A_{i_0+1},A_{i_0},A_{i_0+2},\cdots,A_k)\]
and similarly for $\mathcal B''$. In this case, the application of the linear transformation $\jmath$ in the expression for $\Phi(\,\cdot\,;\x^*)$ implies that, for such $\x_2$, the form of the expression \eqref{eqn0g} for $\Phi(\x_2;\x^*)$ is identical with that for $\Phi(\x_0;\x^*)$.}

\red{Assume now that $\O''$ is not an orthant of $\X^m$. There might still be an intermediate orthant between $\O_{i_0-1}$ and $\O_{i_0+1}$ arrived at by non-trivial partitions $A_{i_0}=C_1\cup C_2$, $A_{i_0+1}=D_1\cup D_2$, $B_{i_0}=E_1\cup E_2$ and $B_{i_0+1}=F_1\cup F_2$ such that the orthant
\[\widetilde{\O}=\O(B_0\cup\cdots\cup B_{i_0-1}\cup E_1\cup F_1\cup C_2\cup D_2\cup A_{i_0+2}\cup\cdots\cup A_k)\]
is contained in $\X^m$ and provides a shorter path between $\O_{i_0-1}$ and $\O_{i_0+1}$. In which case, by Proposition \ref{prop0a}$(i)$, we must have
\[\frac{\|P_{C_1\cup D_1}(\x^*)\|}{\|P_{E_1\cup F_1}(\x_2)\|}<\frac{\|P_{C_2\cup D_2}(\x^*)\|}{\|P_{E_2\cup F_2}(\x_2)\|}.\]
This would result in 
\[\frac{\|P_{C_1\cup D_1}(\x^*)\|}{\|P_{E_1\cup F_1}(\x_2)\|}<\frac{\|P_{A_{i_0}\cup A_{i_0+1}}(\x^*)\|}{\|P_{B_{i_0}\cup B_{i_0+1}}(\x_2)\|}<\frac{\|P_{C_2\cup D_2}(\x^*)\|}{\|P_{E_2\cup F_2}(\x_2)\|}\]
and, taking the limit as $\x_2\rightarrow\x_1$,}
\begin{eqnarray}
\red{\frac{\|P_{C_1\cup D_1}(\x^*)\|}{\|P_{E_1\cup F_1}(\x_1)\|}\leqslant\frac{\|P_{A_{i_0}\cup A_{i_0+1}}(\x^*)\|}{\|P_{B_{i_0}\cup B_{i_0+1}}(\x_1)\|}\leqslant\frac{\|P_{C_2\cup D_2}(\x^*)\|}{\|P_{E_2\cup F_2}(\x_1)\|}.}
\label{eqn2a}
\end{eqnarray}
\red{On the other hand, the closures of the orthants $\O_{i_0-1}$ and $\widetilde{\O}$ being in $\X^m$ ensure that all 2-dimensional orthants in the closure of 
\[\widetilde{\O}^*=\O(B_0\cup\cdots\cup B_{i_0-1}\cup E_1\cup C_2\cup A_{i_0+1}\cup\cdots\cup A_k)\]
are in $\X^m$ and hence, by Definition \ref{def0a}, so too is $\widetilde{\O}^*$ itself. Then, by the assumption of uniqueness of the equality at $\x_1$ of the proposition, we must have by Proposition \ref{prop0a}$(ii)$ that} 
\begin{eqnarray}
\red{\frac{\|P_{C_1}(\x^*)\|}{\|P_{E_1}(\x_1)\|}>\frac{\|P_{A_{i_0}}(\x^*)\|}{\|P_{B_{i_0}}(\x_1)\|}>\frac{\|P_{C_2}(\x^*)\|}{\|P_{E_2}(\x_1)\|}.}
\label{eqn2c}
\end{eqnarray}
\red{Similarly, by considering the orthant $\O(B_0\cup\cdots\cup B_{i_0}\cup F_1\cup D_2\cup A_{i_0+2}\cup\cdots\cup A_k)$, we get}
\begin{eqnarray}
\red{\frac{\|P_{D_1}(\x^*)\|}{\|P_{F_1}(\x_1)\|}>\frac{\|P_{A_{i_0+1}}(\x^*)\|}{\|P_{B_{i_0+1}}(\x_1)\|}>\frac{\|P_{D_2}(\x^*)\|}{\|P_{F_2}(\x_1)\|}.}
\label{eqn2d}
\end{eqnarray}
\red{Since, by assumption, $\O_{i_0}$ drops out of the carrier at $\x_1$, we also have
\[\frac{\|P_{A_{i_0}}(\x^*)\|}{\|P_{B_{i_0}}(\x_1)\|}=\frac{\|P_{A_{i_0+1}}(\x^*)\|}{\|P_{B_{i_0+1}}(\x_1)\|},\]
so that, combining \eqref{eqn2c} and \eqref{eqn2d}, we have
\[\frac{\|P_{C_1\cup D_1}(\x^*)\|}{\|P_{E_1\cup F_1}(\x_1)\|}>\frac{\|P_{A_{i_0}\cup A_{i_0+1}}(\x^*)\|}{\|P_{B_{i_0}\cup B_{i_0+1}}(\x_1)\|}>\frac{\|P_{C_2\cup D_2}(\x^*)\|}{\|P_{E_2\cup F_2}(\x_1)\|},\]
contradicting \eqref{eqn2a}.}

\red{Thus, if $\O''$ is not contained in $\X^m$, the projection of the geodesic from $\x^*$ to $\x_2$ continues to pass through the origin, as shown in Figure \ref{fig2f}$(b)$, and the carrier remains as it was for $\x_1$, where the support is $(\mathcal A',\mathcal B')$ given above. In this case, the form of the expression \eqref{eqn0g} for $\Phi(\x_2;\x^*)$ clearly differs from that for $\Phi(\x_0;\x^*)$.}
\end{proof}

\red{Note that the equality \eqref{eqn2g} for $i=i_0$ at $\x_1$ and the mutual orthogonality of all the axes together imply that}
\begin{eqnarray*}
&&\left(\frac{\|P_{B_{i_0}}(\x_1)\|}{\|P_{A_{i_0}}(\x^*)\|}P_{A_{i_0}}(\x^*),\frac{\|P_{B_{i_0+1}}(\x_1)\|}{\|P_{A_{i_0+1}}(\x^*)\|}P_{A_{i_0+1}}(\x^*)\right)\\
&=&\frac{\|P_{B_{i_0}\cup B_{i_0+1}}(\x_1)\|}{\|P_{A_{i_0}\cup A_{i_0+1}}(\x^*)\|}P_{A_{i_0}\cup A_{i_0+1}}(\x^*).
\end{eqnarray*}
\red{This confirms that the form of the expression for $\Phi(\x_0;\x^*)$ is still valid for $\Phi(\x_1;\x^*)$, as expected by the continuity of geodesics. Similarly, the form of the expression for $\Phi(\x_2;\x^*)$ is still valid for $\Phi(\x_1;\x^*)$ whether or not the orthant $\O_{i_0}$ has been replaced by $\O''$.}

\vskip 6pt
\red{A similar argument to that for the proof of Proposition \ref{prop0d} gives the following complementary result.}

\begin{proposition}
\red{Let $\x^*$ and $\x_1$ be two given points in $\X^m$, and let $(\mathcal{A},\mathcal{B})$ be the support of the geodesic from $\x^*$ to $\x_1$, where $\mathcal A=(A_0,A_1,\cdots,A_k)$ and $\mathcal B=(B_0,B_1,\cdots,B_k)$, and where $k>0$. Assume that all inequalities \eqref{eqn2g} and \eqref{eqn2f}, with $\x_1,\x_2$ replaced by $\x^*,\x_1$ respectively, are strict except that, for $i=i_0>0$ and unique non-trivial partitions $C_{i_01}\cup C_{i_02}$ for $A_{i_0}$ and $D_{i_01}\cup D_{i_02}$ for $B_{i_0}$, \eqref{eqn2f} is an equality and that the corresponding orthant $\O'$ given by \eqref{eqn0h} with $i=i_0$ is contained in $\X^m$.}

\red{If the orthant  
\begin{eqnarray}
\O'''=\O(B_0\cup\cdots\cup B_{i_0-1}\cup D_{i_02}\cup C_{i_01}\cup A_{i_0+1}\cup\cdots\cup A_k)
\label{eqn2b}
\end{eqnarray}
is contained in $\X^m$, there is a neighbourhood $\mathcal N$ of $\x_1$ within its stratum such that the form of the expression for $\Phi(\x;\x^*)$ is the same, for all $\x\in \mathcal N$. Then, the common form of the expression for $\Phi(\x;\x^*)$ is determined by $(\mathcal A',\mathcal B')$, where 
\[\mathcal A'=(A_0,\cdots,A_{i_0-1},C_{i_01},C_{i_02},A_{i_0+1},\cdots,A_k)\]
and similarly for $\mathcal B'$.}

\red{If $\O'''$ is not an orthant of $\X^m$ then, in any neighbourhood $\mathcal N$ of $\x_1$ within its stratum, there are $\x'$ and $\x''$ such that the form for $\Phi(\x';\x^*)$ is the same as that for $\Phi(\x_1;\x^*)$ and that for $\Phi(\x'';\x^*)$ is determined by $(\mathcal A',\mathcal B')$. When $\mathcal N$ is sufficiently small, there are no other possibilities.}
\label{prop0e}
\end{proposition}

The carrier of the geodesic from $\x^*$ to $\x$ will also change when $\x$ moves from one stratum to another which necessarily involves, as initial, final or intermediate stratum, a stratum of locally positive co-dimension. The set of all such strata, together with the quadratic hyper-surfaces determined by \red{equalities in each of the relevant equations \eqref{eqn2g}}, form the defining boundaries for the (\textit{pre})-\textit{vistal polyhedral subdivision}, with respect to $\x^*$, in \cite{MOP}. The points in any component of the complement of these surfaces all have the same carrier. However, \red{for our analysis,} we shall only be concerned with changes in the forms of the expressions taken by $\log_{\x^*}(\x)$, or equivalently by $\Phi(\x;\x^*)$, \red{when $\x$ or $\x^*$ vary within their strata} rather than changes in the underlying carrier. \red{For this, we note that the results in Propositions \ref{prop0d} and \ref{prop0e} where the changed support must be used to obtain the correct expression for $\Phi(\x;\x^*)$ are reflections of each other where an othant is removed or introduced, respectively, in the carrier. Thus, we may encapsulate as follows the hyper-surfaces across which, though not at which, it is necessary to take account of the change of support to obtain the correct value for the logarithm map.}

\begin{definition}
Given a point $\x^*\in\X^m$, $\mathcal{D}_{\x^*}$ denotes the set that consists of all points $\x\in\X^m$ for which the support $(\mathcal{A},\mathcal{B})$, where $\mathcal{A}=(A_0,\cdots,A_k)$ and $\mathcal{B}=(B_0,\cdots,B_k)$, of the geodesic from $\x^*$ to $\x$ has the property that, for one or more \red{$i=i_0>0$, there are non-trivial partitions $A_{i_0}=C_{i_01}\cup C_{i_02}$ and $B_{i_0}=D_{i_01}\cup D_{i_02}$ with 
\begin{eqnarray}
\frac{\|P_{C_{i_01}}(\x^*)\|}{\|P_{D_{i_01}}(\x)\|}=\frac{\|P_{C_{i_02}}(\x^*)\|}{\|P_{D_{i_02}}(\x)\|},
\label{eqn2e}
\end{eqnarray}
where the corresponding orthant $\O'$ of \eqref{eqn0h} is contained in $\X^m$, but $\O'''$ of \eqref{eqn2b} is not.} 
\label{def1g}
\end{definition}

In view of the symmetry that reverses the geodesics at the same time as it reverses the order of the strata and interchanges the roles of the sequences $\mathcal A$ and $\mathcal B$ of edge sets in the support, the definition is symmetric: $\x\in\mathcal{D}_{\x^*}$ if and only if $\x^*\in\mathcal{D}_{\x}$. Since each stratum is a Euclidean orthant, it is preserved under multiplication by $\lambda>0$ in $\mathbb{R}^M$ which also multiplies the length of each curve by $\lambda$. Then, since the geodesic $\gamma$ joining $\x^*$ to $\x$ is the shortest curve through the strata of $\X^m$ from $\x^*$ to $\x$, it follows that $\gamma$ is mapped onto the geodesic from $\lambda\x^*$ to $\lambda\x$. In particular, these two geodesics have the same carrier. Thus, $\mathcal{D}_{\lambda\x^*}=\lambda\mathcal{D}_{\x^*}$ and, since the equations \eqref{eqn2e} are homogeneous, $\mathcal{D}_{\x^*}=\lambda\mathcal{D}_{\x^*}$.

\vskip 6pt
The pseudo-partition of $\X^m$ with respect to $\x^*$ determined by $\mathcal{D}_{\x^*}$ gives rise to a polyhedral subdivision of each stratum by restriction. It is coarser than the (pre)-vistal subdivision of \cite{MOP} and, if $\X^m$ is a tree space and if $\x^*$ lies in a top-dimensional stratum, it is equivalent to the polyhedral subdivision defined in \cite{BLO2}. 
 
\section{Limits, projections and derivatives}

We now turn to certain limits and projections of the \red{translated} logarithm map that, in particular, will enable us to calculate the directional derivatives we require. 

\vskip 6pt
Firstly, we obtain an expression for the limit of the \red{translated} logarithm map as the reference point $\x^*$ moves along a geodesic. For a vector $\w$ in the tangent cone \red{to $\X^m$} at $\x^*$, write $\x^*(\lambda,\w)$ for the point distant $\lambda\|\w\|$ along the geodesic $\gamma$ starting at $\x^*$ with initial tangent vector $\w$. Then, we have the following result.

\begin{theorem}
Let $\sigma=\mathcal{O}(E)$ be a stratum of $\X^m$, $\x^*\in\sigma$ and $\x$ be a fixed choice of point anywhere in $\X^m$.
\begin{enumerate}
\item[$(i)$] If $\w\in\mathbb{R}(E)$ is tangent to $\sigma$ at $\x^*$, then
\[\lim_{\lambda\rightarrow0+}\Phi(\x;\x^*(\lambda,\w))=\Phi(\x;\x^*).\]
\item[$(ii)$] If $\sigma$ bounds $\tau=\mathcal{O}(E\cup F)$ in $\X^m$  and $\w_\tau\in\mathbb{R}(E)\times\O(F)$ is tangent to $\tau$ at $\x^*$, then the limit 
\begin{eqnarray}
\Psi(\x,\w_\tau;\x^*)=\lim\limits_{\lambda\rightarrow0+}\Phi(\x;\x^*(\lambda,\w_\tau))
\label{eqn3f}
\end{eqnarray} 
exists. Moreover, there exist $\epsilon>0$ and sequences $\mathcal{A}=(A_0,A_1,\cdots,A_k)$ and $\mathcal{B}=(B_0,B_1,\cdots,B_k)$ of sets of axes such that, for each $\lambda\in(0,\epsilon)$, $(\mathcal{A},\mathcal{B})$ forms the support of the geodesic from $\x^*(\lambda,\w_\tau)$ to $\x$. In terms of these $\mathcal A$ and $\mathcal B$,
\begin{eqnarray}
\red{\Psi(\x,\w_\tau;\x^*)=\jmath\left(P_{B_0}(\x),-\frac{\|P_{B_1}(\x)\|}{\|W_1\|}W_1,\cdots,-\frac{\|P_{B_k}(\x)}{\|W_k\|}W_k\right),}
\label{eqn3g}
\end{eqnarray}
where \red{$W_i=P_{A_i\cap E}(\x^*)$, unless $P_{A_i\cap E}(\x^*)=0$, in which case $W_i=P_{A_i\cap F}(\w_\tau)$,} the projection of $\w_\tau$ on $\mathbb{R}(A_i)$, and $\jmath$ \red{is the linear transformation given by Definition $\ref{def0i}$}.
\end{enumerate}
\label{thm2}
\end{theorem}

\red{For $\x^*\in\sigma\subseteq\X^m$, $\Psi(\x,\w;\x^*)$ defined by \eqref{eqn3f} is the limit of the translated logarithm map of $\Phi(\x;\x')$ as $\x'\rightarrow\x^*$ from the direction $\w$. When the direction $\w$ is clear in the context we shall, in the following, call $\Psi(\x,\w;\x^*)$ simply the directional limit of $\Phi(\x;\x')$.}

\begin{proof}
($i$) This follows from the uniform continuity of geodesics with respect to their end points (cf. \cite{BH}, pp195-196) and also from a minor modification of the proof of ($ii$) below.

($ii$) \red{Note that, since $\w_\tau\in\mathbb{R}(E)\times\O(F)$, $\x^*$ and $\x^*(\lambda,\w_\tau)$ lie in different strata.} Writing $\gamma_\lambda$ for the geodesic from $\x^*(\lambda,\w_\tau)$ to $\x$, as $\x^*(\lambda,\w_\tau)$ moves along $\gamma$ the support of $\gamma_\lambda$ can only change when $\gamma$ meets transversally one or more of the hyper-surfaces where the carrier of the geodesic to $\x$ changes. This can only happen at discrete points along $\gamma$ so, for some $\epsilon>0$ and $0<\lambda\leqslant\epsilon$, the carriers of the geodesics $\gamma_\lambda$ will be independent of $\lambda$. Let $(\mathcal{A},\mathcal{B})$ be the support of $\gamma_\epsilon$ from $\x^*(\epsilon,\w_\tau)$ to $\x$, where $\mathcal{A}=(A_0,A_1,\cdots,A_k)$ and $\mathcal{B}=(B_0,B_1,\cdots,B_k)$. Then, \red{$A_0\cup A_1\cup\cdots\cup A_k=E\cup F$ and,} for $0<\lambda\leqslant\epsilon$, the integer $k$ and the support $(\mathcal{A},\mathcal{B})$ will remain constant for the expression
\begin{eqnarray*}
\Phi(\x;\x^*(\lambda,\w_\tau))&=&\jmath\left(\red{P_{B_0}(\x), -\frac{\|P_{B_1}(\x)\|}{\|P_{A_1}(\x^*(\lambda,\w_\tau))\|}P_{A_1}(\x^*(\lambda,\w_\tau)),\cdots,}\right.\\
&&\quad\qquad\left.\red{-\frac{\|P_{B_k}(\x)\|}{\|P_{A_k}(\x^*(\lambda,\w_\tau))\|}P_{A_k}(\x^*(\lambda,\w_\tau))}\right),
\end{eqnarray*}
replacing $\x^*$ in \eqref{eqn0g} by $\x^*(\lambda,\w_\tau)$. Then, since the $\x^*(\lambda,\w_\tau)$ lie in $\tau$ for all sufficiently small positive $\lambda$, the vectors $\Phi(\x;\x^*(\lambda,\w_\tau))$ all lie in $\mathbf{C}_\tau$ so that it makes sense to take the limit as $\lambda\rightarrow0+$, \red{where $\mathbf C_\tau$ is the common translated cone of the tangent cone at $\x^*(\lambda,\w_\tau)$ as introduced in Section 2}.

To evaluate it, we take the limit in the above expression for $\Phi(\x;\x^*(\lambda,\w_\tau))$. Since $\x^*\in{\mathcal O}(E)$, $\x^*(\lambda,\w_\tau) =\x^*+\lambda \w_\tau$ for sufficiently small $\lambda>0$ and it follows that \red{$P_{A_i}(\x^*(\lambda,\w_\tau)) = P_{A_i\cap E}(\x^*) + \lambda P_{A_i}(\w_\tau)$}. So the limit as $\lambda\rightarrow0+$ of this term is \red{$P_{A_i\cap E}(\x^*)$} if that is non-zero. If it is zero, then $A_i\cap E=\emptyset$ since \red{$\|P_{\{e\}}(\x^*)\|>0$} for all $e\in E$. Then \red{$P_{A_i}(\w_\tau)$, the projection of $\w_\tau$ on $\mathbb{R}(A_i)$ is, in fact, $P_{A_i\cap F}(\w_\tau)$.}
\end{proof}

If $\sigma$ has co-dimension $l$ and $\tau$ co-dimension $l'$ then, when $l-l'=1$ and so $|F|=1$, there is no $i>0$ such that $|A_i|>1$ and \red{$P_{A_i\cap E}(\x^*)=0$} as all the axes involved in the carrier that are not in $E\cup F$ are in $A_0=B_0$. If further $l=1$ and $l'=0$, that is, $\sigma$ is a stratum of \red{local} co-dimension one and $\tau$ co-bounding $\sigma$ is a \red{locally} top-dimensional stratum, then $\Psi(\,\cdot\,,\w_\tau;\x^*)$ obtained here is identical with the map resulting from the `folding map' composed with $\Phi(\,\cdot\,;\x^*)$ used in \cite{BLO2} when $\X^m$ is a tree space, noting that $\w_\tau$ in this case is unique up to a positive scalar multiple. 

\begin{example}\label{ex3}
\red{Consider the orthant space $\X^2$ in Example \ref{ex2}. Take $\sigma=\{o\}$ and $\tau=\O(u_1,u_2)$. Recall from Example \ref{ex2} that the tangent cone to $\X^2$ at $o$ is $\X^2$ itself. Take $\w_\tau=\x^*$, where $\x^*\in\tau$ is indicated in Figure \ref{fig1}. Then, $\Psi(\x,\w_\tau;o)=\Phi(\x;\x^*)$ for any $\x\in\X^2$. Since the light grey region in Figure \ref{fig1} may change if $\x^*$ changes, $\Phi(\,\cdot\,;\x^*)$ may change as a map when $\x^*$ changes. Hence, the directional limit $\Psi(\,\cdot\,,\w_\tau;o)$ of $\Phi(\,\cdot\,;\lambda\w_\tau)$ from the direction $\w_\tau$ as $\lambda\rightarrow0$, as a map, also depends on $\w_\tau$.}
\end{example}

For $\w_\tau$ as given in Theorem \ref{thm2}$(ii)$, \red{write} $\w_\tau^\perp$ for the component of $\w_\tau$ orthogonal to $\sigma$, that is, the component in $\{\textbf{0}\}\times\O(F)\subset\mathbb{R}(E)\times\O(F)$. \red{Then, the following consequences of Theorem \ref{thm2} imply that, although the directional limit $\Psi(\x,\w_\tau;\x^*)$ generally depends on $\w_\tau$, for given $\x$ and $\x^*$, as noted in Example \ref{ex3} above, it remains constant in some circumstances. In particular, to consider the changes of $\Psi(\x,\w_\tau;\x^*)$ as $\x$ varies, it suffices to restrict attention to $\w_\tau\in\S_{\tau\setminus\sigma}^{l-l'}$, recalling that $\S_{\tau\setminus\sigma}^{l-l'}$ is the open unit spherical segment of $\{{\bf0}\}\times\O(F)$ given by Definition \ref{def0c}.}

\begin{corollary}
With the notation and hypotheses of Theorem $\ref{thm2}(ii)$,  
\begin{enumerate}
\item[$(i)$] $\Psi(\,\cdot\,,\lambda\w_\tau;\x^*)=\Psi(\,\cdot\,,\w_\tau;\x^*)$ for all $\lambda>0$;
\item[$(ii)$] $\Psi(\x,\w_\tau;\x^*)=\Psi(\x,\w_\tau^\perp;\x^*)$.
\end{enumerate}
\label{cor0a}
\end{corollary}

\begin{proof}
$(i)$ is obvious \red{from the expression \eqref{eqn3g}} and $(ii)$ is immediate since $\sigma=\O(E)$, $\tau=\O(E\cup F)$ and only the $F$-coordinates of $\w_\tau$ are potentially involved in \eqref{eqn3g}.
\end{proof}

When $\x^*$ lies in a stratum $\sigma$ of positive co-dimension that is not locally top-dimensional, the vector $\log_{\x^*}(\x)$, \red{and so $\Phi(\x;\x^*)$,}  will usually have non-zero components both tangent to $\sigma$ and orthogonal to it. In order to discuss \red{the projections, onto these components, of the translated logarithm map and of its directional limits, as well as to discuss their derivatives}, we \red{extend the notation $P$ for projection maps on $\X^m$ given by Definition \ref{def0g} to include projection maps on \red{ tangent cones, or their translated cones}. However, since we are more interested in the orthant itself rather than the axes determining it, we shall use $P_\sigma$ instead of $P_E$, where $\sigma=\O(E)$. In particular, for any stratum $\tau=\O(E\cup F)$ co-bounding $\sigma$ in $\X^m$, $P_\sigma$ and $P_{\tau\setminus\sigma}$ respectively are the projections onto the two factors of the corresponding stratum $\mathbb{R}(E)\times\O(F)$ in the common translated cone $\CC_\sigma$, or equivalently in the tangent cone at a point of $\sigma$, depending on the context.} 

For $\x^*$ in $\sigma=\O(E)$ or in $\tau=\O(E\cup F)$ co-bounding $\sigma$, we shall denote $P_\sigma(\log_{\x^*}(\x))$ by $\log^\sigma_{\x^*}(\x)$ and $P_\sigma(\Phi(\x;\x^*))$ by $\Phi_\sigma(\x;\x^*)$. Note that, on $\CC_\sigma$, $P_\sigma$ so defined is the tangential projection onto $\sigma$ and $P_{\tau\setminus\sigma}$ is one of several possible normal projections. In particular, for $\w_\tau\in\mathbb R(E)\times\O(F)$, $\w_\tau^\perp=P_{\tau\setminus\sigma}(\w_\tau)$. We shall \red{further} extend the notation $P_\sigma$ to include top-dimensional, or locally top-dimensional, strata by taking it to be the identity in that case, so that in particular $\Phi_\sigma(\x;\x^*)=\Phi(\x;\x^*)$ if $\sigma$ is a top-dimensional, or locally top-dimensional, stratum.

\vskip 6pt
For $\x^*$ in $\sigma$ of \red{locally} positive co-dimension, the non-zero components of $\log_{\x^*}(\x)$ orthogonal to $\sigma$ correspond to axes \red{with respect to which $\x^*$ has zero coefficient and $\x$ has non-zero coefficient}. Hence, these axes are in $A_0=B_0=E(\x^*,\x)$, \red{the set of axes common to all strata in the carrier of the geodesic between these two points,} so that they correspond to components of $v_0$ in \eqref{eqn0}. This implies, in particular, that $\Phi_\sigma(\x;\x^*)$ is given by \eqref{eqn0g} with \red{$P_{B_0}(\x)$ there replaced by $P_{B_0\cap E}(\x)$.} Then, since the restriction to each stratum of the set $\mathcal{D}_{\x}$ given by Definition \ref{def1g} is relatively closed, the \red{form of the} expression for $\Phi(\x;\x')$ will remain constant for $\x'$ varying in a neighbourhood of $\x^*$ in $\sigma$ when $\x^*$ is restricted to avoid $\mathcal{D}_{\x}$. Hence, the proof of Lemma 4 in \cite{BLO2} of the differentiability of $\Phi(\x;\x^*)$ with respect to $\x^*$ for the case that $\X^m$ is a tree space and $\x^*$ lies in a top-dimensional stratum will give the following generalisation of that result to the derivative of $\Phi_\sigma(\x;\x^*)$ \red{with respect to $\x^*$}. Since the proof is similar to that for Lemma 4 in \cite{BLO2}, we omit it here.

\begin{proposition}
Let $\x$ and $\x^*$ be fixed points in $\X^m$ with $\x^*$ in the stratum $\sigma=\mathcal{O}(E)$ and $\x\not\in\mathcal{D}_{\x^*}$, \red{where the set $\mathcal{D}_{\x^*}$ is given by Definition $\ref{def1g}$}. Then, the map
\[\sigma\rightarrow\mathbb{R}(E);\quad x'\mapsto\Phi_\sigma(\x;\x')\]
is differentiable with respect to $\x'$ at $\x^*$ with derivative given by
\begin{eqnarray}
\begin{array}{rcl}
&&M_{\x^*}^\sigma(\x)\\
&=\!\!\!&J^\top\hbox{\rm diag}\!\left\{{\mathbf0}_{|B_0\cap E|},\red{-\|P_{B_1}(\x)\|M^\dagger_{P_{A_1}(\x^*)},\cdots\!,-\|P_{B_k}(\x)\|M^\dagger_{P_{A_k}(\x^*)}}\!\right\}J
\end{array}
\label{eqn4a}
\end{eqnarray}
where the sequences $\mathcal{A}=(A_0,\cdots,A_k)$ and $\mathcal{B}=(B_0,\cdots,B_k)$ form the support of the geodesic from $\x^*$ to $\x$ and $J$ is the matrix representation of the \red{linear transformation $\jmath$ given by Definition $\ref{def0i}$}, and where, for $\y=(y_1,\cdots,y^{\phantom{A}}_l\!\!)\not=0$,
\begin{eqnarray}
M^\dagger_{\y}=\frac{1}{\|\y\|}I_l-\frac{1}{\|\y\|^3}\y^\top\y
\label{eqn10p}
\end{eqnarray}
is the derivative of the map $\y\mapsto\frac{1}{\|\y\|}\y$. 
\label{prop1}
\end{proposition}

\red{Note that, if $l>1$, $\|\y\|\,M^\dagger_{\y}$ is the projection onto the hyper-plane in $\mathbb{R}^l$ orthogonal to $\y$ and, when $l=1$, $M^\dagger_{y_1}=0$. Hence, if $k=0$ or if $k>0$ and $|A_i|=1$ for all $1\leqslant i\leqslant k$, then the derivative of $\Phi_\sigma(\x;\x')$, with respect to $\x'$, at $\x'=\x^*$ is zero. Recall that the corresponding translated logarithm map in the Euclidean space is the identity map, independent of $\x'$, and so its derivative with respect to $\x'$ is identically zero. Hence, in a broad sense, Proposition \ref{prop1} captures where and how the derivative of $\Phi_\sigma(\x;\x')$ differs from that of the corresponding translated Euclidean logarithm map.}

\vskip 6pt
Returning to \red{the directional limit} $\Psi(\x,\w_\tau;\x^*)$ \red{of $\Phi(\x;\x')$} with $\x^*\in\sigma=\O(E)$, where $\tau=\O(E\cup F)$ co-bounds $\sigma$ and $\w_\tau$ is in $\mathbb{R}(E)\times\O(F)$, since $\Psi(\x,\w_\tau;\x^*)$ is in $\CC_\tau$, both projections $\Psi_\tau(\x,\w_\tau;\x^*)=P_\tau(\Psi(\x,\w_\tau;\x^*))$ and $P_\sigma(\Psi(\x,\w_\tau;\x^*))$ are well defined. In particular, $\Psi_\tau(\,\cdot\,,\w_\tau;\x^*)$ is a map from $\X^m$ onto $\mathbb{R}(E\cup F)$. Then, we also have the following consequences of Theorem \ref{thm2}, \red{giving the relationships between the projections of the directional limit of the translated logarithm map and the directional limit of the projections of the translated logarithm map}. 

\begin{corollary}
With the notation and hypotheses of Theorem $\ref{thm2}(ii)$,  
\begin{enumerate}
\item[$(i)$] $\lim\limits_{\lambda\rightarrow0+}\Phi_\tau(\x;\x^*(\lambda,\w_\tau))=\lim\limits_{\lambda \rightarrow0+}\Phi_\tau(\x;\x^*(\lambda,\w_\tau^\perp))=\Psi_\tau(\x,\w_\tau;\x^*)$;
\item[$(ii)$] $P_\sigma\left(\Psi_\tau(\x,\w_\tau;\x^*)\right)=\Phi_\sigma(\x;\x^*)$.
\end{enumerate}
\label{cor00}
\end{corollary}

\begin{proof}
The equality of the extreme terms in $(i)$ follows since the $W_i$ in \eqref{eqn3g} are determined by the axes in $E\cup F$, so that it does not matter whether we project on $\O(E\cup F)$ before or after taking the limit, and the remaining term \red{$P_{B_0\cap(E\cup F)}(\x)$}        remains constant throughout the limiting process. The equality with the central term in $(i)$ follows from Corollary \ref{cor0a}$(ii)$: $\Psi_\tau(\x,\w_\tau;\x^*)=\Psi_\tau(\x,\w_\tau^\perp;\x^*)$, which is $\lim\limits_{\lambda\rightarrow0+}\Phi_\tau(\x;\x^*(\lambda,\w_\tau^\perp))$ by the case already established. 

Note that, since projection onto $\mathbb{R}(E)\subset\mathbb{R}(E\cup F)$ is unaffected by first projecting onto $\mathbb{R}(E\cup F)$, $(ii)$ is equivalent to 
\begin{eqnarray}
P_\sigma(\Psi(\x,\w_\tau;\x^*))=\Phi_\sigma(\x;\x^*).
\label{eqn4g}
\end{eqnarray}
To establish \eqref{eqn4g}, we need to allow for the fact that the geodesics $\gamma_\lambda$ \red{from $\x^*(\lambda,\w_\tau)$ to $\x$} and the geodesic $\gamma_0$ from $\x^*$ to $\x$ may have different carriers. We assume that $\lambda$ is restricted to the range $0<\lambda<\epsilon$ such that the initial segments of $\gamma_\lambda$ all lie in $\zeta=\mathcal{O}(E\cup F\cup G)$, where possibly $G=\emptyset$ and so $\zeta=\tau$, and let $K$ be the set of axes \red{with respect to which the initial segment of $\gamma_0$ has positive coordinates}. Then, $K\supseteq E\cup G$. Now, $e\in E\cup F\cup G$ if, and only if, for each $\lambda$ and some maximal $\delta(\lambda)>0$, \red{$\|P_{\{e\}}(\gamma_\lambda(s))\|>0$} for $s\in(0,\delta(\lambda))$. From the uniform continuity of geodesics with respect to their endpoints, %(cf. \cite{BH}, pp195-196), 
it is clear that we must have $\delta(\lambda)\rightarrow\delta_0\geqslant0$ as $\lambda\rightarrow0$. If $\delta_0>0$, then \red{$\|P_{\{e\}}(\gamma_0(s))\|>0$} for $s\in(0,\delta_0)$ and so $e\in K$. Conversely, $e\in K\cap(E\cup F\cup G)$ implies that \red{$\|P_{\{e\}}(\gamma_0(s))\|>0$} for $s\in(0,\delta(0))$ and we must have $\delta_0=\delta(0)$. 

Thus, for any axis $e$ in $K\cap(E\cup F\cup G)$, the projections \red{$P_{\{e\}}(\gamma_\lambda(s))$ and $P_{\{e\}}(\gamma_0(s))$} of the initial segments of these geodesics all lie in the closure of the stratum $\mathcal{O}(E\cup F\cup G)$. The uniform continuity of these geodesics, and so of their projections, with respect to their endpoints, together with their linearity within that closed stratum, implies that the components \red{$P_{\{e\}}\left(\dot\gamma_\lambda(0)\right)$ converge to $P_{\{e\}}\left(\dot\gamma_0(0)\right)$} as $\lambda\rightarrow0$. In particular, since $E\subseteq K$, this is valid for any axis $e$ in $E$, which establishes \eqref{eqn4g}.
\end{proof}

The comments made prior to Proposition \ref{prop1} regarding the \red{form of the} expression for $\Phi_\sigma(\x;\x^*)$ can be generalised to apply to $\Psi_\tau(\x,\w_\tau;\x^*)$: using the notation in Theorem \ref{thm2}$(ii)$ for $\Psi(\x,\w_\tau;\x^*)$ we have that
\begin{eqnarray}
\Psi_\tau(\x,\w_\tau;\x^*)=\jmath\left(\red{P_{B_0\cap(E\cup F)}(\x),-\frac{\|P_{B_1}(\x)\|}{\|W_1\|}W_1,\cdots,-\frac{\|P_{B_k}(\x)\|}{\|W_k\|}W_k}\right).
\label{eqn4b}
\end{eqnarray}

\vskip 6pt
Recall that $\S^{l-l'}_{\tau\setminus\sigma}$ denotes the set of unit vectors in $\{{\bf0}\}\times\O(F)\subset\mathbb{R}(E)\times\O(F)$ that comprises all unit vectors that are tangent to $\tau$ and orthogonal to $\sigma$. If $l-l'=1$, $\S^{l-l'}_{\tau\setminus\sigma}$ comprises a single point. When $l-l'>1$, for any fixed $\x\in\X^m$, the pseudo-partition of $\X^m$ determined by $\mathcal{D}_{\x}$ induces a polyhedral subdivision of $\S^{l-l'}_{\tau\setminus\sigma}$ where, in each cell of the induced polyhedral subdivision, the \red{form of the} expression \eqref{eqn3g} for $\Psi(\x,\,\cdot\,;\x^*)$, and so the \red{form of the} expression for $\Psi_\tau(\x,\,\cdot\,;\x^*)$, remains the same. In particular, this implies that, for fixed $\x$, $\Psi_\tau(\x,\w_\tau;\x^*)$ is a continuous function of $\w_\tau\in\S^{l-l'}_{\tau\setminus\sigma}$. In fact, the directional derivatives of $\Psi_\tau(\x,\w_\tau;\x^*)$ with respect to $\w_\tau$ also exist in directions $\v$ in the tangent space to $\S^{l-l'}_{\tau\setminus\sigma}$ at $\w_\tau$ that we denote by $\mathcal{T}_{\w_\tau}(\S^{l-l'}_{\tau\setminus\sigma})$. These derivatives have the property given in the following proposition, where we note that $\mathbb{R}(E)\times\O(F)\subset\mathbb{R}(E)\times\mathbb{R}(F)$ so that, for fixed $\x$ and $\x^*$, $\w_\tau$ and $\Psi_\tau(\x,\w_\tau;\x^*)$ lie in the same Euclidean space.  

\begin{proposition}
Let the stratum $\sigma=\mathcal{O}(E)$ of co-dimension $l(\geqslant2)$ bound, in $\X^m$, the stratum $\tau=\O(E\cup F)$ of co-dimension $l'(<l-1)$. Fix $\x,\x^*\in\X^m$ with $\x^*\in\sigma$. Then, as a function of $\w_\tau\in\S^{l-l'}_{\tau\setminus\sigma}$, the directional derivative $D$ of $\Psi_\tau(\x,\w_\tau;\x^*)$ at $\w_\tau$ in the direction $\v\in\mathcal{T}_{\w_\tau}(\S^{l-l'}_{\tau\setminus\sigma})$ exists and satisfies
\[\langle\w_\tau,\,D\Psi_\tau(\x,\w_\tau;\x^*)(\v)\rangle=0.\]
\label{prop2}
\end{proposition}

\begin{proof}
Without loss of generality, we may assume that $\|\v\|=1$. Consider the geodesic on $\S_{\tau\setminus\sigma}^{l-l'}$ given by $\alpha(s)=\w_\tau\cos s+\v\sin s$. Write $w_1$ for a vector whose coordinates comprise a subset of those of $\w_\tau$, and $v_1$, $\alpha_1$ for the corresponding components of $\v$ and $\alpha$ respectively. Then, the initial tangent vector of the function $f(s)=\frac{\alpha_1(s)}{\|\alpha_1(s)\|}$ is $\dot f(0)=v_1M^\dagger_{w_1}$, where $M^\dagger_{\y}$ is given by \eqref{eqn10p}. Clearly, $\langle w_1,\dot f(0)\rangle=0$, since the image of $M^\dagger_{w_1}$ is orthogonal to $w_1$. 

On the other hand, it follows from the argument in the proof of Theorem \ref{thm2} that, for all sufficiently small $s$, \red{the expression for} $\Psi_\tau(\x,\alpha(s);\x^*)$ all have the same \red{form} provided that, when $\w_\tau$ lies on the boundary of a cell of the induced polyhedral subdivision on $\S^{l-l'}_{\tau\setminus\sigma}$, we use for $\w_\tau$ the expression valid for $s>0$. Thus, we may use the expression for $\Psi_\tau(\x,\w_\tau;\x^*)$ given by \eqref{eqn4b} to express $D\Psi_\tau(\x,\w_\tau;\x^*)(\v)$ in the form $\v M_{\x^*,\x}(\w_\tau)$, where 
\begin{eqnarray}
M_{\x^*,\x}(\w_\tau)=J^\top\hbox{\rm diag}\left\{\textbf{0},\red{-\|P_{B_{l_1}}(\x)\|M^\dagger_{W_{l_1}},\cdots, -\|P_{B_{l_j}}(\x)\|}M^\dagger_{W_{l_j}},\textbf{0}\right\}J
\label{eqn0p}
\end{eqnarray}
and where, using the notation of Theorem \ref{thm2}, \red{$W_{l_i}=P_{A_{l_i}\cap F}(\w_\tau)$} are just those components in the expression for $\Psi_\tau(\x,\w_\tau;\x^*)$ for which \red{$P_{A_{l_i}\cap E}(\x^*)=0$  and $|A_{l_i}\cap F|>1$}. \red{Since $\|y\|M^\dagger_y$ is the projection onto the hyperplane orthogonal to $y$ in the Euclidean space where $y$ lies as noted after the statement of Proposition \ref{prop1},} the result follows.
\end{proof}

The proof of Proposition \ref{prop2} also shows that, if $\w_\tau$ lies in the interior of a single cell of the induced polyhedral subdivision of $\S^{l-l'}_{\tau\setminus\sigma}$, then $\Psi_\tau(\x,\w'_\tau;\x^*)$ is differentiable with respect to $\w'_\tau$ at $\w_\tau$. However, if $\w_\tau$ lies in the boundary of a cell of the induced polyhedral subdivision, \red{this no longer holds, although directional derivatives still exist}.  

\vskip 6pt
The directional derivative of $\langle\w_\tau,\,\Psi_\tau(\x,\w_\tau;\x^*)\rangle$, as a function of $\w_\tau$, now follows from Proposition \ref{prop2}.

\begin{corollary}
Assume that all assumptions in Proposition $\ref{prop2}$ hold. Then, for any $\v\in\mathcal{T}_{\w_\tau}(\S^{l-l'}_{\tau\setminus\sigma})$, the derivative $D$ in the direction $\v$ of $\langle\w_\tau,\,\Psi_\tau(\x,\w_\tau;\x^*)\rangle$ at $\w_\tau$ is given by
\[D\left(\langle\w_\tau,\,\Psi_\tau(\x,\w_\tau;\x^*)\rangle\right)(\v)=\langle\v,\Psi_\tau(\x,\w_\tau;\x^*)\rangle.\]
\label{cor0}
\end{corollary}

\begin{proof}
The second term in the expansion
\begin{eqnarray*}
&&D\left(\langle\w_\tau,\,\Psi_\tau(\x,\w_\tau;\x^*)\rangle\right)(\v)\\
&=&\langle D\w_\tau(\v),\Psi_\tau(\x,\w_\tau;\x^*)\rangle+\langle\w_\tau,\,D\Psi_\tau(\x,\w_\tau;\x^*)(\v)\rangle
\end{eqnarray*}
vanishes by Proposition \ref{prop2}. The result then follows since the directional derivative $D\w_\tau(\v)$ is given by the derivative at $s=0$ of the geodesic $\alpha(s)=\w_\tau\cos s+\v\sin s$.
\end{proof}

\section{Characterisation of Fr\'echet means}

In the remainder of this paper, we use the knowledge obtained so far on the \red{translated} logarithm map to investigate Fr\'echet means of probability measures on $\X^m$. So, from now on we assume that $\mu$ is a probability measure on $\X^m$ and that its Fr\'echet function defined by \eqref{eqn1}, where $\M=\X^m$, is finite at one point. The latter ensures that the Fr\'echet function of $\mu$ is finite everywhere. 

\vskip 6pt
Since the squared distance on a CAT(0)-space is a convex function with respect to each of its variables, it follows that the Fr\'echet mean of $\mu$ is unique and that the condition for $\x^*$ to be the Fr\'echet mean of $\mu$, that is, the condition for $\x^*$ to satisfy
\begin{eqnarray*}
\int_{\X^m} d(\x^*,\x)^2\,d\mu(\x)<\int_{\X^m} d(\x',\x)^2\,d\mu(\x),\qquad\hbox{ for any }\x'\not=\x^*,
\end{eqnarray*}
is equivalent to this inequality holding in any neighbourhood of $\x^*$. Then, since the Fr\'echet function of $\mu$ is differentiable at $\x^*$ if $\x^*$ lies in a top-dimensional, or locally top-dimensional, stratum, the above condition for such $\x^*$ to be the Fr\'echet mean of $\mu$ is equivalent to the condition that
\begin{eqnarray}
\int_{\X^m}\log_{\x^*}(\x)\,d\mu(\x)=0,
\label{eqn00}
\end{eqnarray}
similar to the condition for Fr\'echet means in Riemannian manifolds of non-positive curvature.

When $\x^*$ lies in a stratum $\sigma$ of \red{locally} positive co-dimension, the squared distance $d(\x^*,\x)^2$ is no longer differentiable at $\x^*$ for any fixed $\x$. Nevertheless, it has directional derivatives along all possible directions and then the above condition becomes that, at $\x^*\in\sigma$, the Fr\'echet function of $\mu$ has non-negative directional derivatives along all possible directions. The fact that $\X^m$ is a CAT(0)-space also implies that the derivative at $\x^*$ in the direction $\w$ of the distance function $d_{\x}=d(\cdot,\x)$ can be expressed as 
\[(Dd_{\x}(\x^*))(\w)=-\frac{1}{d_{\x}(\x^*)}\ll \w,\log_{\x^*}(\x)\gg,\]
where $\ll\,\,,\,\,\gg$ is defined by \eqref{eqn0a} (cf. \cite{BK}, (2.5), p417). Thus, the criterion for a point $\x^*$ lying in a stratum $\sigma$ of \red{locally positive} co-dimension to be the Fr\'echet mean of $\mu$ is equivalent to the condition that
\begin{eqnarray}
\int_{\X^m}\ll\w,\log_{\x^*}(\x)\gg\,d\mu(\x)\leqslant0
\label{eqn1p}
\end{eqnarray}
for all tangent vectors $\w$ at $\x^*$. 

For any vector $\w$ at $\x^*$ which is tangent to $\sigma$, the fact that $-\w$ is also tangent to $\sigma$ at $\x^*$ implies that the inequality \eqref{eqn1p} must be an equality for all such $\w$. From this it follows that 
\begin{eqnarray}
\int_{\X^m}\log^\sigma_{\x^*}(\x)\,d\mu(\x)=0,
\label{eqn2p}
\end{eqnarray}
analogous to the condition \eqref{eqn00}. On the other hand, for any given stratum $\tau$ co-bounding $\sigma$ and any vector $\w$ at $\x^*$ tangent to $\tau$, it is possible to link the derivative, at $\x^*$, of the Fr\'echet function in the direction $\w$ with $\Psi_\tau(\,\cdot\,,\w;\x^*)$, \red{the projection of the directional limit of $\Phi(\,\cdot\,;\x^*)$}. To show this, we need the following limiting property of the directional derivatives on general CAT(0)-spaces.

\begin{lemma}
Let $X$ be a ${\rm CAT}(0)$-space, and let $x_0$ and $x$ be two distinct fixed points in $X$. For some $\epsilon>0$, assume that $\gamma:[0,\epsilon)\rightarrow X$ is a geodesic with $\gamma(0)=x$ and $\dot\gamma(0)=v_x$. Then, if $\{x_i\,:\,i\geqslant1\}$ is a sequence of points along $\gamma$ convergent to $x$, the derivative $D$ at $x$ in the direction $v_x$ of the distance function $d_{x_0}=d(x_0,\cdot)$ has the property that
\[Dd_{x_0}(v_x)=\lim_{i\rightarrow\infty}Dd_{x_0}(v_{x_i}),\]
where $v_{x_i}$ denotes the tangent vector at $x_i$ of the geodesic $\gamma$.
\label{lem2}
\end{lemma}

\begin{proof}
\red{For $x,y,z\in X$, denote by $\angle_x(y,z)$ the Alexandrov angle at $x$ between the geodesics from $x$ to $y$ and $z$ respectively.} Since $Dd_{x_0}(v_x)=-\ll v_x,\log_x(x_0)\gg/d_{x_0}(x)=-\red{\|v_x\|}\cos\angle_x(x',x_0)$ where $x'$ is a point on the geodesic $\gamma$, it is sufficient to show that, for a fixed point $x'$ chosen on $\gamma$, $\angle_x(x',x_0)=\lim\limits_{i\rightarrow\infty}\angle_{x_i}(x',x_0)$.
 
For this, we write $\gamma_{a,b}$ for the (unique) geodesic segment joining $a$ and $b$, for any two distinct points $a$ and $b$ in $X$. Then, given sequences of points $a_i\rightarrow a$, $b_i\rightarrow b$ and $c_i\rightarrow c$ in $X$, it follows from the Cartan-Hadamand theorem that the geodesic segments $\gamma_{a_i,b_i}$ and $\gamma_{a_i,c_i}$ converge uniformly, as maps, to $\gamma_{a,b}$ and $\gamma_{a,c}$ respectively. %(cf. \cite{BH}, pp195-196). 
From this it follows that $\angle_a(b,c)\geqslant\limsup_{i\rightarrow\infty}\angle_{a_i}(b_i,c_i)$ (cf. \cite{BBI}, Theorem 4.3.11, p.119). Applying this to the sequence of geodesic triangles $\Delta(x'x_ix_0)$, we obtain 
\begin{eqnarray}
\angle_x(x',x_0)\geqslant\limsup_{i\rightarrow\infty}\angle_{x_i}(x',x_0).
\label{eqn3p}
\end{eqnarray} 

On the other hand, using (4.3) p.124 of \cite{BBI}, we have 
\[\limsup_{i\rightarrow\infty}\overline\angle_{x_i}(x,x_0)\leqslant\pi-\angle_x(x',x_0),\]
where, \red{as in Section 2}, $\overline\angle$ denotes the corresponding comparison angle in $\mathbb{R}^2$. Then, since $\overline\angle_{x_i}(x,x_0)\geqslant\angle_{x_i}(x,x_0)$, the above implies that
\begin{eqnarray*}
\angle_x(x',x_0)&\leqslant&\pi-\limsup_{i\rightarrow\infty}\overline\angle_{x_i}(x,x_0)\\
&\leqslant&\pi-\limsup_{i\rightarrow\infty}\angle_{x_i}(x,x_0)\\
&=&\liminf_{i\rightarrow\infty}\left\{\pi-\angle_{x_i}(x,x_0)\right\}.
\end{eqnarray*}
However, since $X$ has non-positive curvature, if $x_i$ lies between $x$ and $x'$ on the geodesic segment $\gamma_{x,x'}$, then $\angle_{x_i}(x',x_0)+\angle_{x_i}(x_0,x)\geqslant\pi$ (cf. \cite{BBI}, p117, line 5). Hence, 
\[\angle_x(x',x_0)\leqslant\liminf_{i\rightarrow\infty}\angle_{x_i}(x',x_0).\]
This, together with \eqref{eqn3p}, gives that
\[\angle_x(x',x_0)=\lim_{i\rightarrow\infty}\angle_{x_i}(x',x_0),\]
so that the required result follows.
\end{proof}

Recalling that $\Phi(\x;\x^*)=\log_{\x^*}(\x)+\x^*$, the criteria \eqref{eqn1p} and \eqref{eqn2p} for a point $\x^*$ to be the Fr\'echet mean of $\mu$ may now be recast, the former in terms of the standard Euclidean inner product and $\Psi_\tau(\x,\w;\x^*)$, \red{the projection of the directional limit of $\Phi(\x;\x^*)$}, when $\x^*$ lies in a stratum $\sigma$ of positive co-dimension and $\w$ is tangent to a co-bounding stratum $\tau$.

\begin{theorem}
Let $\sigma$ be a stratum in $\X^m$ of co-dimension $l(\geqslant0)$. The necessary and sufficient conditions for a given point $\x^*\in\sigma$ to be the Fr\'echet mean of $\mu$ are
\begin{enumerate}  
\item[$(i)$] for any stratum $\tau$ in $\X^m$ of co-dimension $l'$, $0\leqslant l'<l$, co-bounding $\sigma$ and any $\w_\tau\in\S_{\tau\setminus\sigma}^{l-l'}$,
\begin{eqnarray}
\left\langle\w_\tau,\int_{\X^m}\Psi_\tau(\x,\w_\tau;\x^*)\,d\mu(\x)\right\rangle\leqslant0,
\label{eqn4p}
\end{eqnarray}
\red{where $\S_{\tau\setminus\sigma}^{l-l'}$ is given by Definition $\ref{def0c}$;}
\item[$(ii)$] for all $l\geqslant0$,
\begin{eqnarray}
\x^*=\int_{\X^m}\Phi_\sigma(\x;\x^*)\,d\mu(\x).
\label{eqn5p}
\end{eqnarray}
\end{enumerate}
\label{thm3}
\end{theorem}

Note that case $(i)$ may only occur if $l>0$, but need not occur then. Note also that, if $\X^m$ is a tree space, the special case $l=0$ of this result is the same as that of Lemma 3 of \cite{BLO2}; and the special case $l=1$, so that $l'=0$, is equivalent to that given by Lemma 5 of \cite{BLO2}: on the one hand, $\mathcal{S}^{l-l'}_{\tau\setminus\sigma}$ contains a single unit vector and, on the other hand, as we noted earlier, $\Psi_\tau(\,\cdot\,,\w_\tau;\x^*)=\Psi(\,\cdot\,,\w_\tau;\x^*)$ is identical with the composition of the `folding map' with $\Phi(\,\cdot\,;\x^*)$ in \cite{BLO2}.  

\begin{proof}
Noting that ($ii$) is precisely \eqref{eqn2p}, it is sufficient to show that ($i$) is equivalent to \eqref{eqn1p} for any tangent vector $\w$ that is not tangent to $\sigma=\O(E)$. For this, we fix any stratum $\tau=\O(E\cup F)$, of co-dimension $l'$, co-bounding $\sigma=\O(E)$ and take $\w=\w_\tau\in\mathbb{R}(E)\times\O(F)$. Then, it follows from Lemma \ref{lem2} that \eqref{eqn1p} is equivalent to
\begin{eqnarray}
\lim_{\lambda\rightarrow0+}\int_{\X^m}\ll\w_\tau,\log_{\x^*(\lambda,\w_\tau)}(\x)\gg\,d\mu(\x)\leqslant0,
\label{eqn5a}
\end{eqnarray}
where $\x^*(\lambda,\w_\tau)\red{=\x^*+\lambda\w_\tau}$, as defined prior to Theorem \ref{thm2}. Since $\w_\tau$ is tangent to $\tau$ at $\x^*(\lambda,\w_\tau)$ for sufficiently small $\lambda>0$ and, for any given $\x$, $\log_{\x^*(\lambda,\w_\tau)}(\x)$ is tangent either to $\tau$ or to one of the strata that co-bound $\tau$, we have 
\[\ll\w_\tau,\log_{\x^*(\lambda,\w_\tau)}(\x)\gg=\langle\w_\tau,\log_{\x^*(\lambda,\w_\tau)}(\x)\rangle.\] 
However, 
\begin{eqnarray*}
\left\langle\w_\tau,\log_{\x^*(\lambda,\w_\tau)}(\x)\right\rangle&=&\left\langle\w_\tau,\Phi(\x;\x^*(\lambda,\w_\tau))-\x^*(\lambda,\w_\tau)\right\rangle\\
&=&\left\langle\w_\tau,\Phi_\tau(\x;\x^*(\lambda,\w_\tau))-\x^*(\lambda,\w_\tau)\right\rangle.
\end{eqnarray*}
Hence, by Corollary \ref{cor00}$(i)$ and then Corollary \ref{cor0a}$(ii)$, \eqref{eqn5a} is equivalent to
\[\int_{\X^m}\left\langle\w_\tau,\Psi_\tau(\x,\w_\tau^\perp;\x^*)\right\rangle\,d\mu(\x)\leqslant\langle\w_\tau,\x^*\rangle,\]
where $\w_\tau^\perp=P_{\tau\setminus\sigma}(\w_\tau)$. Decomposing $\w_\tau$ as $\w_\tau=\w_\sigma+\w_\tau^\perp$, where $\w_\sigma=P_\sigma(\w_\tau)$, leads to 
\begin{eqnarray*}
\left\langle\w_\tau,\Psi_\tau(\x,\w^\perp_\tau;\x^*)\right\rangle
&=&\left\langle\w_\sigma,\Psi_\tau(\x,\w^\perp_\tau;\x^*)\right\rangle+\left\langle\w^\perp_\tau,\Psi_\tau(\x,\w^\perp_\tau;\x^*)\right\rangle\\
&=&\left\langle\w_\sigma,\Phi_\sigma(\x;\x^*)\right\rangle+\left\langle\w^\perp_\tau,\Psi_\tau(\x,\w^\perp_\tau;\x^*)\right\rangle,
\end{eqnarray*}
where the second equality follows from Corollary \ref{cor00}$(ii)$. The required result now follows by noting $(ii)$, noting that $\langle\w_\tau,\x^*\rangle=\langle\w_\sigma,\x^*\rangle=0$ and noting that, by applying the projection $P_\tau$ to the result of Corollary \ref{cor0a}$(i)$, $\Psi_\tau(\x,\w^\perp_\tau;\x^*)=\Psi_\tau(\x,\w^\perp_\tau/\|\w^\perp_\tau\|;\x^*)$. 
\end{proof}

\red{From now on, we assume that $\bxi$ is a random variable defined on a probability space $(\bOmega,\mathcal{F},{\bf P})$ with values in $\X^m$ and that $\mu$ is the distribution (measure) of $\bxi$, i.e. $\mu(B)={\bf P}(\bxi^{-1}(B))$ for any Borel set $B$ in $\X^m$.} When the stratum containing the Fr\'echet mean $\x^*$ of the probability measure $\mu$ on $\X^m$ is of \red{locally} positive co-dimension, \eqref{eqn4p} being an equality has a significant influence on the nature of the distributions of the Euclidean random variables $\Psi_\tau(\bxi,\w_\tau;\x^*)$\red{, which will be seen in Propositions \ref{prop5a}, \ref{prop5b} and \ref{prop6}. We shall also see, in Proposition \ref{prop4}, its link with the long term behaviour of sample Fr\'echet means}. 

\begin{definition}
For the stratum $\sigma$ of co-dimension $l(\geqslant1)$, in which the Fr\'echet mean $\x^*$ of $\mu$ lies, and the stratum $\tau$, of co-dimension $l'$, co-bounding $\sigma$, the subset $\Theta_{\tau,\sigma}(\x^*;\mu)$ of $\S^{l-l'}_{\tau\setminus\sigma}$ is defined as 
\begin{eqnarray}
\Theta_{\tau,\sigma}(\x^*;\mu)=\left\{\w_\tau\in\S^{l-l'}_{\tau\setminus\sigma}\mid\hbox{the inequality \eqref{eqn4p} for }\w_\tau\hbox{ is an equality}\right\},
\label{eqn7e}
\end{eqnarray}
\red{where $\S_{\tau\setminus\sigma}^{l-l'}$ is given by Definition $\ref{def0c}$.}
\label{def2a}
\end{definition}

The convexity of the directional derivative $D(d_{\x}^2)(\w)$ in $\w$ (cf. \cite{BK}, pp416-417) ensures that $\Theta_{\tau,\sigma}(\x^*;\mu)$ is a convex subset of $\S^{l-l'}_{\tau\setminus\sigma}$ and that
\begin{eqnarray}
\Theta_\sigma(\x^*;\mu)=\bigcup_{\tau\supset\sigma}\Theta_{\tau,\sigma}(\x^*;\mu)
\label{eqn7f}
\end{eqnarray}
is a convex subset of $\bigcup\limits_{\tau\supset\sigma}\mathcal{S}_{\tau\setminus\sigma}^{l-l'}\subseteq\CC_\sigma$. If $l-l'=1$, $\S^{l-l'}_{\tau\setminus\sigma}$ consists of a single unit vector so that $\Theta_{\tau,\sigma}(\x^*;\mu)$ is either $\S^{l-l'}_{\tau\setminus\sigma}$ itself or an empty set. In general, if the closure of $\Theta_{\tau,\sigma}(\x^*;\mu)$ is contained in $\S^{l-l'}_{\tau\setminus\sigma}$, the fact that $\langle\w_\tau,\,\Psi_\tau(\x,\w_\tau;\x^*)\rangle$ is continuous in $\w_\tau\in\S^{l-l'}_{\tau\setminus\sigma}$ implies that $\Theta_{\tau,\sigma}(\x^*;\mu)$ itself must be closed. 

\vskip 6pt
The following result gives a relationship between the Fr\'echet mean $\x^*$ of $\mu$ and the Euclidean mean of $\Psi_\tau(\bxi,\w_\tau;\x^*)$. Here, and henceforth, by interior we intend the relative interior, that is, interior with respect to the subspace topology.

\begin{proposition}
Let the stratum $\sigma$ of co-dimension $l(\geqslant2)$ bound, in $\X^m$, the stratum $\tau$ of co-dimension $l'(<l-1)$. Assume that the Fr\'echet mean $\x^*$ of $\mu$ lies in $\sigma$ and that ${\rm int}(\Theta_{\tau,\sigma}(\x^*;\mu))\not=\emptyset$. Then, for any $\w_\tau\in\Theta_{\tau,\sigma}(\x^*;\mu)$, 
\begin{eqnarray}
\int_{\X^m}\left\{\Psi_\tau(\x,\w_\tau;\x^*)-\Phi_\sigma(\x;\x^*)\right\}d\mu(\x)=0.
\label{eqn8p}
\end{eqnarray}
\label{prop5a}
\end{proposition}

Note that, if $l'=l-1$, equality \eqref{eqn8p} holds automatically since its left hand side is a 1-dimensional vector so that the equality follows from the assumption that $\w_\tau\in\Theta_{\tau,\sigma}(\x^*;\mu)$. 

\begin{proof}
By the continuity of $\Psi_\tau(\x,\w_\tau;\x^*)$ in $\w_\tau$, we may assume that $\w_\tau\in\hbox{int}(\Theta_{\tau,\sigma}(\x^*;\mu))$. Then equality holds in \eqref{eqn4p} in a neighbourhood of $\w_\tau$, so that 
\[D\left(\left\langle\w_\tau,\int_{\X^m}\Psi_\tau(\x,\w_\tau;\x^*)\,d\mu(\x)\right\rangle\right)(\v)=0,\quad\forall\v\in\mathcal{T}_{\w_\tau}(\S^{l-l'}_{\tau\setminus\sigma}).\]
By Corollary \ref{cor0}, this implies that
\[\left\langle\v,\,\int_{\X^m}\Psi_\tau(\x,\w_\tau;\x^*)\,d\mu(\x)\right\rangle=0,\quad\forall\v\in\mathcal{T}_{\w_\tau}(\S^{l-l'}_{\tau\setminus\sigma}).\]
On the other hand, it follows from $\int_{\X^m}\Phi_\sigma(\x;\x^*)\,d\mu(\x)=\x^*$ and $\langle\w_\tau,\x^*\rangle=0$ that 
\[\left\langle\w_\tau,\,\int_{\X^m}\Phi_\sigma(\x;\x^*)\,d\mu(\x)\right\rangle=0,\quad\forall\w_\tau\in\S^{l-l'}_{\tau\setminus\sigma}.\]
Hence, taking the directional derivative of the left hand side as a function of $\w_\tau\in\S^{l-l'}_{\tau\setminus\sigma}$, we have
\[\left\langle\v,\,\int_{\X^m}\Phi_\sigma(\x;\x^*)\,d\mu(\x)\right\rangle=0\]
for all $\v\in\mathcal{T}_{\w_\tau}(\S^{l-l'}_{\tau\setminus\sigma})$. Noting that the left hand side of \eqref{eqn8p} is a vector lying in the $(l-l')$-dimensional Euclidean space containing $\S^{l-l'}_{\tau\setminus\sigma}$, the fact that $\w_\tau\in\Theta_{\tau,\sigma}(\x^*;\mu)$, together with the above, implies that the required result holds for any $\w_\tau\in\hbox{int}(\Theta_{\tau,\sigma}(\x^*;\mu))$. 
\end{proof}

One immediate consequence of Proposition \ref{prop5a} is the following.

\begin{corollary} 
Assume that the conditions given in Proposition $\ref{prop5a}$ are satisfied. If $\sigma=\O(E)$ and $\tau=\O(E\cup F)$ then, for all $\w_\tau\in\Theta_{\tau,\sigma}(\x^*;\mu)$,
\begin{eqnarray}
\x^*=\int_{\X^m}\Psi_\tau(\x,\w_\tau;\x^*)\,d\mu(\x).
\label{eqn11p}
\end{eqnarray}
That is, the point $\x^*\in\sigma$, as a point in $\mathbb{R}(E\cup F)$, is the Euclidean mean of each of the Euclidean random variables \red{$\Psi_\tau(\bxi,\w_\tau;\x^*)$} for such $\w_{\tau}$. 
\label{cor0b}
\end{corollary}

\red{If a stratum $\sigma=\mathcal{O}(E)$ of co-dimension $l(\geqslant1)$ bounds, in $\X^m$, the stratum $\tau$ of co-dimension $l'(<l)$ and if $\x^*\in\sigma$, then it follows from the proof of Proposition \ref{prop2} that the maps $\Psi_\tau(\,\cdot\,,\w^1_\tau;\x^*)$ and $\Psi_\tau(\,\cdot\,,\w^2_\tau;\x^*)$ from $\X^m$ to $\mathbb{R}(E\cup F)$ are generally not identical for any given distinct $\w^i_\tau\in\S^{l-l'}_{\tau\setminus\sigma}$, $i=1,2$. With the insight obtained from that proof, to characterise the places where they differ we introduce the subset $\Sigma_{\tau,\sigma}(\x^*;\w_\tau)$ of $\X^m$ as follows. It will be clear later, in the proof of Proposition \ref{prop6}, that the set of $\x\in\X^m$ where $\Psi(\x,\w_\tau^1;\x^*)\not=\Psi(\x,\w_\tau^2;\x^*)$ is contained in the set $\Sigma_{\tau,\sigma}(\x^*;\w_\tau^1)\cup\Sigma_{\tau,\sigma}(\x^*;\w_\tau^2)\cup\mathcal D_{\x^*}$. Thus, in particular, for $\bxi$ lying outside of the latter set, the Euclidean random variables $\Psi(\bxi,\w_\tau^1;\x^*)$ and $\Psi(\bxi,\w_\tau^2;\x^*)$ are identical. This fact will be used in the derivation of the limiting distribution of sample Fr\'echet means in the next section.} 

\begin{definition}
Let the stratum $\sigma=\mathcal{O}(E)$ of co-dimension $l(\geqslant1)$ bound, in $\X^m$, the stratum $\tau=\O(E\cup F)$ of co-dimension $l'(<l)$. For $\x^*\in\sigma$ and $\w_\tau\in\S^{l-l'}_{\tau\setminus\sigma}$, a point $\x\in\X^m$ is called singular with respect to $(\x^*,\w_\tau)$, if at least one $A_i$ with $A_i\cap E=\emptyset$ has \red{$|A_i\cap F|>1$, where $i\geqslant1$ and the sequences $\mathcal A=(A_0,A_1,\cdots,A_k)$ and $\mathcal B=(B_0,B_1,\cdots,B_k)$ form the support of the geodesics from $x^*+\lambda\w_\tau$ to $\x$ for all sufficiently small $\lambda>0$.} The set $\Sigma_{\tau,\sigma}(\x^*;\w_\tau)$ consists of all points $\x$ that are singular with respect to $(\x^*,\w_\tau)$. 
\label{def2}
\end{definition}

\red{For example, in the orthant space $\X^2$ of Example \ref{ex3}, using the notation there, $\Sigma_{\tau,\{o\}}(o;\w_\tau)$ is the closure of the light grey region in Figure \ref{fig1}.} It follows from comparison of the corresponding expressions \eqref{eqn3g} and \eqref{eqn4b} that the singularity of $\x$ with respect to $(\x^*,\w_\tau)$ has the same effect on $\Psi_\tau(\x,\w_\tau;\x^*)$ as it does on $\Psi(\x,\w_\tau;\x^*)$. \red{In particular, in terms of the matrix $M_{\x^*,\x}(\w)$ given by \eqref{eqn0p}, we can express $\Sigma_{\tau,\sigma}(\x^*;\w_\tau)$ defined above as
\[\Sigma_{\tau,\sigma}(\x^*;\w_\tau)\setminus\mathcal D_{x^*}=\{\x\in\X^m\mid M_{\x^*,\x}(\w_\tau)\not=0\}.\]}

Note that $\Sigma_{\tau,\sigma}(\x^*;\w_\tau)=\emptyset$ if $l-l'=1$, since then $\mathcal{S}^{l-l'}_{\tau\setminus\sigma}$ contains a single unit vector $\w_\tau$ which leads to the impossibility that \red{$|A_i\cap F|>1$}. Generally, if $l-l'>1$, which implies that $l\geqslant2$, $\Sigma_{\tau,\sigma}(\x^*;\w_\tau)$ could be relatively substantial. Nevertheless, we have the following result on the measure of $\Sigma_{\tau,\sigma}(\x^*;\w_\tau)$.% for $\w_\tau\in{\rm int}(\Theta_{\tau,\sigma}(\x^*;\mu))$. 

\begin{proposition}
Let the stratum $\sigma$ of co-dimension $l(\geqslant2)$ bound, in $\X^m$, the stratum $\tau$ of co-dimension $l'(<l-1)$. Assume that the Fr\'echet mean $\x^*$ of $\mu$ lies in $\sigma$ and that $\w_\tau\in{\rm int}(\Theta_{\tau,\sigma}(\x^*;\mu))$, where $\Theta_{\tau,\sigma}(\x^*;\mu)$ is defined by \eqref{eqn7e}. Then, $\mu\left(\Sigma_{\tau,\sigma}(\x^*;\w_\tau)\right)=0$. 
\label{prop5b}
\end{proposition}

\begin{proof}
Let $\alpha(s)$ be a unit speed geodesic in $\S^{l-l'}_{\tau\setminus\sigma}$, write $\v(s)=\dot\alpha(s)$ and define $h(s)=\left\langle\v(s),\int_{\X^m}\Psi_\tau(\x,\alpha(s);\x^*)\,d\mu(\x)\right\rangle$. Since $\S^{l-l'}_{\tau\setminus\sigma}$ is an open subset of a Euclidean sphere, we have $\dot\v(s)=-\alpha(s)$, $\ddot\alpha(s)=-\alpha(s)$ and so, by Proposition \ref{prop2} and its proof,
\begin{eqnarray*}
\dot h(s)&=&\left\langle\v(s),\int_{\X^m}D\Psi_\tau(\x,\alpha(s);\x^*)(\v(s))\,d\mu(\x)\right\rangle\\
&=&\left\langle\v(s),\int_{\Sigma_{\tau,\sigma}(\x^*;\alpha(s))}\v(s)M_{\x^*,\x}(\alpha(s))\,d\mu(\x)\right\rangle,
\end{eqnarray*}
where $M_{\x^*,\x}(\w)$ is given by \eqref{eqn0p}. The expression for $M_{\x^*,\x}(\w)$ implies that, for $\w\in\S^{l-l'}_{\tau\setminus\sigma}$ and any fixed $\x\in\Sigma_{\tau,\sigma}(\x^*;\w)$, $\langle\v,\v M_{\x^*,\x}(\w)\rangle$ can be written in the form
\[-\sum_{i=1}^j\frac{\red{\|P_{B_{l_i}}(\x)\|}}{\|W_{l_i}\|}\left\{\|v^{\phantom{A}}_{l_i}\|^2-\left\langle v_i,\frac{W_{l_i}}{\|W_{l_i}\|}\right\rangle^2\right\}\]
\red{for some $1\leqslant j\leqslant k$,} where $W_{l_i}$ and \red{$P_{B_{l_i}}(\x)$} are those required for the expression \eqref{eqn0p} for $M_{\x^*,\x}(\w)$ in the proof of Proposition \ref{prop2}. This implies that $\dot h(0)$ must be non-positive. Moreover, for any open or closed subset $\mathcal{E}\subseteq\Sigma_{\tau,\sigma}(\x^*;\alpha(0))$ such that $\Psi_\tau(\x,\alpha(0);\x^*)$ has the same expression for all $\x\in\mathcal E$, there is a vector $\v(0)\in\mathcal{T}_{\alpha(0)}(\S^{l-l'}_{\tau\setminus\sigma})$ such that $\langle\v(0),\,\v(0)\, M_{\x^*,\x}(\alpha(0))\rangle<0$ for all $\x\in\mathcal E$. Then, if $\mu(\mathcal{E})\not=0$, the corresponding $h$ satisfies
\[\dot h(0)\leqslant\left\langle\v(0),\int_{\mathcal E}\v(0)M_{\x^*,\x}(\alpha(0))\,d\mu(\x)\right\rangle<0.\]
Clearly, $\Sigma_{\tau,\sigma}(\x^*;\alpha(0))$ can be decomposed as a finite disjoint union of such sets $\mathcal E$. 

If $\w_\tau=\alpha(0)\in\hbox{int}(\Theta_{\tau,\sigma}(\x^*;\mu))$ then, for any $v(0)\in\mathcal{T}_{\w_\tau}(\S^{l-l'}_{\tau\setminus\sigma})$, the corresponding geodesic $\alpha(s)$ lies in $\Theta_{\tau,\sigma}(\x^*;\mu)$ for all sufficiently small $s\geqslant0$. Using a similar argument to that for the proof of Proposition \ref{prop5a}, the corresponding $h(s)$ must be identically zero for all sufficiently small $s\geqslant0$, which implies that $\dot h(0)=0$. Hence, we must have $\mu(\Sigma_{\tau,\sigma}(\x^*;\w_\tau))=0$. 
\end{proof}

If a stratum $\sigma$ bounds $\tau$ in $\X^m$, $\x^*\in\sigma$ and $\w^1_\tau$, $\w^2_\tau$ are two different vectors at $\x^*$ tangent to $\tau$, then \red{it follows from the map $\Psi_\tau(\,\cdot\,,\w^1_\tau;\x^*)$ generally differing from $\Psi_\tau(\,\cdot\,,\w^2_\tau;\x^*)$ that} the distribution of the Euclidean random variable $\Psi_\tau(\bxi,\w^1_\tau;\x^*)$ generally differs from that of $\Psi_\tau(\bxi,\w^2_\tau;\x^*)$. Nevertheless, under the conditions in Proposition \ref{prop5b}, the $\Psi_\tau(\bxi,\w_\tau;\x^*)$ are in fact a.s. identical for $\w_\tau\in{\rm int}(\Theta_{\tau,\sigma}(\x^*;\mu))$. 

\begin{proposition}
Assume that $\bxi$ is a random variable on $\X^m$ with distribution measure $\mu$ having Fr\'echet mean $\x^*$. Assume further that $\mu\left(\mathcal{D}_{x^*}\right)=0$ and that  $\x^*$ lies in the stratum $\sigma=\O(E)$ of co-dimension $l(\geqslant2)$. Let the stratum $\tau$ of co-dimension $l'(<l-1)$ co-bound $\sigma$, in $\X^m$. If ${\rm int}(\Theta_{\tau,\sigma}(\x^*;\mu))\not=\emptyset$, then the distributions of the Euclidean random variables $\Psi_\tau(\bxi,\w_\tau;\x^*)$ are independent of $\w_\tau\in{\rm int}(\Theta_{\tau,\sigma}(\x^*;\mu))$, where the set $\Theta_{\tau,\sigma}(\x^*;\mu)$ is defined by \eqref{eqn7e}.
\label{prop6}
\end{proposition}

Note that the example in the next section makes it clear that the condition $\w_\tau\in{\rm int}(\Theta_{\tau,\sigma}(\x^*;\mu))$ in the statement of Proposition \ref{prop6} cannot be relaxed to $\w_\tau\in\Theta_{\tau,\sigma}(\x^*;\mu)$.

\begin{proof}
First, we show that, for any given distinct $\w^j_\tau\in\S^{l-l'}_{\tau\setminus\sigma}$, $j=1,2$, and for $\x\not\in\Sigma_{\tau,\sigma}(\x^*;\w^1_\tau)\bigcup\Sigma_{\tau,\sigma}(\x^*;\w^2_\tau)\bigcup\mathcal{D}_{\x^*}$, $\Psi_\tau(\x,\w^1_\tau;\x^*)=\Psi_\tau(\x,\w^2_\tau;\x^*)$. Then, it follows from the assumption and Proposition \ref{prop5b} that $\Psi_\tau(\bxi,\w^1_\tau;\x^*)=\Psi_\tau(\bxi,\w^2_\tau;\x^*)$ a.s. Recall \red{from the proof of Theorem \ref{thm2}$(ii)$} that, for fixed $\x\in\X^m$, $\x^*\in\sigma$ and $\w_\tau\in S^{l-l'}_{\tau\setminus\sigma}$, the supports of the geodesics from $\x^*(\lambda,\w_\tau)=\x^*+\lambda\w_\tau$ to $\x$ are the same, for all sufficiently small $\lambda>0$, and that the expression for $\Psi_\tau(\x,\w_\tau;\x^*)$ is determined by this common support. Thus, $\Psi_\tau(\x,\w_\tau^j;\x^*)$ is identical if the geodesics from $\x^*(\lambda,\w_\tau^j)$ to $\x$ have the same support when $\lambda>0$ is sufficiently small. 

Suppose now that the supports $(\mathcal{A}^j,\mathcal{B}^j)$, $j=1,2$, of the geodesics from $\x^*(\lambda,\w^1_\tau)$ and $\x^*(\lambda,\w^2_\tau)$ respectively to $\x$ are different, for all sufficiently small $\lambda>0$. Then, the geodesic $\gamma_\lambda$ between $\x^*(\lambda,\w^1_\tau)$ and $\x^*(\lambda,\w^2_\tau)$ must meet at least one hyper-surface in $\mathcal{D}_{\x}$. If there are more than one, but necessarily finitely many, such hyper-surfaces, by introducing a point on $\gamma_\lambda$ between each pair of consecutive such hyper-surfaces, the change of the supports of the geodesics from points of $\gamma_\lambda$ to $\x$ can be considered inductively to reduce the case to where $\gamma_\lambda$ meets only one such hyper-surface. 

Hence, without loss of generality, we assume that $\gamma_\lambda$ only meets $\mathcal{D}_{\x}$ at a point \red{$\x_\lambda$ on one of the hyper-surfaces in $\mathcal{D}_{\x}$. That is, $\x_\lambda$ satisfies \eqref{eqn2e} for a particular $i_0$ with $\x^*$ being replaced by $\x_\lambda$ and all the other relevant inequalities in Proposition \ref{prop0a}, with $\x_1$ and $\x_2$ replaced by $\x_\lambda$ and $\x$, are strict}. If $\x\not\in\mathcal{D}_{\x^*}$ so that $\x^*\not\in\mathcal{D}_{\x}$, we may assume that the points $\x^*(\lambda,\w^j_\tau)$ lie on the opposite sides of $H$ for all sufficiently small $\lambda>0$. \red{Then, by Proposition \ref{prop0e}}, as $\gamma_\lambda$ moves through \red{$\x_\lambda$}, the supports of the geodesics from $\gamma_\lambda$ to $\x$ change, with the relevant subset \red{$A^1_{i_0}=C_{i_01}\cup C_{i_02}$ of the sequence $\mathcal{A}^1=(A^1_0,\cdots,A^1_k)$ in the support $(\mathcal{A}^1,\mathcal{B}^1)$ on the one side splitting, say, into two subsets $C_{i_01},C_{i_02}$ on the other, and similarly for $B^1_{i_0}$ in $\mathcal B^1$. That is, the support $(\mathcal{A}^2,\mathcal{B}^2)$ of the geodesics from $\x^*(\lambda,\w_\tau^2)$ to $\x$ is related to $(\mathcal{A}^1,\mathcal{B}^1)$ by $\mathcal{A}^2=(A^1_0,\cdots,A^1_{i_0-1},C_{i_01},C_{i_02},A^1_{i_0+1},\cdots,A^1_k)$, and similarly $\mathcal B^2$ to $\mathcal B^1$.} We show now that neither of these subsets $C_{i_01}$ and $C_{i_02}$ can meet $E$. \red{If only one of these two sets meets $E$, say $C_{i_01}$, then since $P_{C_{i_02}}(\x^*(\lambda,\w^1_\tau))\rightarrow0$ as $\lambda\rightarrow0$, it is impossible that there are $\x_\lambda$ such that the corresponding equality \eqref{eqn2e} holds for all sufficiently small $\lambda>0$. Similarly, if both of these sets meet $E$, then the proof of Corollary \ref{cor00}$(ii)$ shows that $P_{C_{i_0s}}(\x^*(\lambda,\w^j_\tau))\rightarrow P_{C_{i_0s}}(\x^*)$, as $\lambda\rightarrow0$, for $j=1,2$. This implies that, for $j=1$, the corresponding strict inequality \eqref{eqn2f} holds for $\x^*$ while, for $j=2$, it is reversed. Hence, that is also impossible.}

Thus, in the case when the supports $(\mathcal{A}^i,\mathcal{B}^i)$ are different, we still have $A^1_j=A^2_j$ for all $j>0$ such that $A^1_j\cap E\not=\emptyset$. 
 
If further $\x\not\in\Sigma_{\tau,\sigma}(\x^*;\w^1_\tau)\bigcup\Sigma_{\tau,\sigma}(\x^*;\w^2_\tau)$, then \red{the change of the support described above cannot happen when both $C_{i_01}\cap E$ and $C_{i_02}\cap E$ are empty, as then $|A_{i_0}^1\cap F|>1$, and so we would have $\x\in\Sigma_{\tau,\sigma}(\x^*;\w^1_\tau)$. Since $A^1_0=A^2_0$, the above implies that we must have $(\mathcal A^i,\mathcal B^i)$ identical for $i=1,2$} and so, for such $\x$, $\Psi_\tau(\x,\w^2_\tau;\x^*)=\Psi_\tau(\x,\w^1_\tau;\x^*)$. 

\vskip 6pt
Next, assume that the two $\w^j_\tau$ are chosen to be sufficiently close that, for any given $\x$ and all sufficiently small $\lambda>0$, the geodesics from $\x^*(\lambda,\w^j_\tau)$ to $\x$ have the same support. Then, if $\w_\tau(\alpha)$, $\alpha\in[0,1]$, is the geodesic between $\w^1_\tau$ and $\w^2_\tau$, an elementary argument on the relevant parameters in the inequalities \eqref{eqn2g} and \eqref{eqn2f} that determine the carrier will show that these parameters are monotonic in $\alpha$ along the geodesic. So, the geodesic from $\x^*(\lambda,\w_\tau(\alpha))$ to $\x$ will have the same support as that for the geodesics from $\x^*(\lambda,\w^j_\tau)$ to $\x$. This implies that $\Sigma_{\tau,\sigma}(\x^*;\w_\tau(\alpha))\subseteq\Sigma_{\tau,\sigma}(\x^*;\w^1_\tau)\bigcup\Sigma_{\tau,\sigma}(\x^*;\w^2_\tau)$, so that $\Psi_\tau(\bxi,\w_\tau(\alpha);\x^*)$ are a.s. independent of $\alpha\in[0,1]$. 

\vskip 6pt
Finally, since $\Theta_{\tau,\sigma}(\x^*;\mu)$ is convex, there is a sequence $\{\w^n_\tau\mid n\geqslant1\}\subset\hbox{int}(\Theta_{\tau,\sigma}(\x^*;\mu))$ such that
\begin{eqnarray}
\hbox{int}(\Theta_{\tau,\sigma}(\x^*;\mu))=\lim_{n\rightarrow\infty}C_n,
\label{eqn12a}
\end{eqnarray}
where $C_n$ is the convex hull in $\S_{\tau\setminus\sigma}^{l-l'}$ of $\{\w^1_\tau,\cdots,\w^n_\tau\}$. The above argument implies that, without loss of generality, we may also assume that $\{\w^n_\tau\mid n\geqslant1\}$ have the property that, for any $\w_\tau\in C_n$,
\[\Sigma_{\tau,\sigma}(\x^*;\w_\tau)\subseteq\bigcup_{i=1}^n\Sigma_{\tau,\sigma}(\x^*;\w^i_\tau).\]
This shows that
\[\mu\left(\bigcup_{\w_\tau\in C_n}\Sigma_{\tau,\sigma}(\x^*;\w_\tau)\right)=\mu\left(\bigcup_{i=1}^n\Sigma_{\tau,\sigma}(\x^*;\w^i_\tau)\right)=0,\]
so that $\Psi_\tau(\bxi,\w_\tau;\x^*)$ are a.s. independent of $\w_\tau\in C_n$. Hence, it follows from \eqref{eqn12a} that
\[\mu\left(\bigcup_{\w_\tau\in{\rm int}(\Theta_{\tau,\sigma}(\x^*;\mu))}\Sigma_{\tau,\sigma}(\x^*;\w_\tau)\right)=\lim_{n\rightarrow\infty}\mu\left(\bigcup_{\w_\tau\in C_n}\Sigma_{\tau,\sigma}(\x^*;\w_\tau)\right)=0,\]
which gives the required result.
\end{proof}

\section{The limiting distribution of sample Fr\'echet means}

In this section, we assume that $\{\bxi_i\,:\,i\geqslant1\}$ is a sequence of \textit{i.i.d.} random variables \red{defined on a common probability space $(\bOmega,\mathcal F,{\bf P})$ with values} in $\X^m$; \red{that $\mu$ is the distribution measure of $\bxi_1$;} and that $\hat\bxi_n$ is the sample Fr\'echet mean of $\bxi_1,\cdots,\bxi_n$. Then, $\hat\bxi_n$ converges to the Fr\'echet mean $\x^*$ of $\mu$ almost surely as $n$ tends to infinity (cf. \cite{HZ}).

\subsection{\red{On the support of the limiting distribution}}

If $\x^*$ lies in a top-dimensional stratum, $\X^m$ is locally an $m$-dimensional manifold. One would expect that the limiting behaviour of sample Fr\'echet means $\hat\bxi_n$ is similar, to some extent, to that of sample Fr\'echet means in a Riemannian manifold as obtained in \cite{BP} and \cite{KL}. In particular, the support of the limiting distribution of $\sqrt{n}\log_{x^*}(\hat\bxi_n)=\sqrt{n}(\hat\bxi_n-\x^*)$ is the entire tangent space to $\X^m$ at $\x^*$, as long as $\hbox{cov}(\Phi(\bxi_1;\x^*))$ has rank $m$. This fact was proved for the case of open books in \cite{HHLMMMNOPS} and for the case of tree spaces in \cite{BLO1} and \cite{BLO2}. We shall see in the following that the argument used in \cite{BLO2} can be generalised to $\X^m$, so that the corresponding conclusion is also valid for orthant spaces.

However, when $\x^*$ lies in a stratum of \red{locally} positive co-dimension, the limiting behaviour of sample Fr\'echet means is generally very different. In the case that $\X^m$ is an open book or a tree space and that the stratum containing $\x^*$ is of the co-dimension one, this phenomenon was observed and studied in \cite{HHLMMMNOPS}, \cite{BLO1} and \cite{BLO2}. Similarly, for general orthant spaces, the strictness or otherwise of the inequality \eqref{eqn4p} affects the limiting behaviour of $\hat\bxi_n$. In particular, when \eqref{eqn4p} is strict, there is a constraint on the support of the limiting distribution. To describe this, we recall that, for $\sigma=\O(E)$ of co-dimension $l$ and $\tau=\O(E\cup F)$ of co-dimension $l'<l$ co-bounding $\sigma$, we are denoting the set of unit vectors in $\mathbb{R}(E)\times\O(F)$ by $\S^{m-l'}_{\tau,\sigma}$ and those in $\{\mathbf{0}\}\times\O(F)$ by $\S^{l-l'}_{\tau\setminus\sigma}$. Then, for $\w_\tau$ in the latter, denote by $\mathcal{H}_{\w_\tau}$ the intersection of the half hyper-plane $\mathbb{R}(E)\times\{c\w_\tau\mid c>0\}$ with $\S^{m-l'}_{\tau,\sigma}$, namely
\[\mathcal{H}_{\w_\tau}=\Big\{\w_\sigma+c\w_\tau\in\S^{m-l'}_{\tau,\sigma}\,\,\Big|\,\,c>0\hbox{ and }\w_\sigma\in\mathbb{R}(E)\times\{\mathbf{0}\}\subset\mathbb{R}(E)\times\O(F)\Big\},\]
and let
\[\red{\Omega^k_n(\w_\tau)=\left\{\omega\in\bOmega\mid\hat\bxi_n(\omega)\in\tau\hbox{ and }d\left(\frac{\hat\bxi_n(\omega)-\x^*}{\|\hat\bxi_n(\omega)-\x^*\|},\,\mathcal{H}_{\w_\tau}\right)\leqslant\frac{1}{k}\right\}.}\]

\begin{proposition} 
Let the stratum $\sigma=\mathcal{O}(E)$ of co-dimension $l(\geqslant1)$ bound, in $\X^m$, the stratum $\tau=\O(E\cup F)$ of co-dimension $l'(<l)$. Assume that the Fr\'echet mean $\x^*$ of $\mu$ lies in $\sigma$ and that $\w_\tau\in\S^{l-l'}_{\tau\setminus\sigma}\red{\setminus\Theta_{\tau,\sigma}(\x^*;\mu)}$, \red{where $\S^{l-l'}_{\tau\setminus\sigma}$ and $\Theta_{\tau,\sigma}(\x^*;\mu)$ are given by Definitions $\ref{def0c}$ and $\ref{def2a}$ respectively}. Then, %with $\mathcal{H}_{\w_\tau}$ and $\Omega^k_n(\w_\tau)$ as just defined, 
\[\lim_{k\rightarrow\infty}\red{{\bf P}}\left(\limsup_n\Omega^k_n(\w_\tau)\right)%=\lim_{k \rightarrow \infty}\red{{\bf P}}\left(\bigcap_{l\geqslant1}\bigcup_{n\geqslant l}\Omega^k_n(\w_\tau)\right)
=0.\]
\label{prop4}
\end{proposition}

\begin{proof}
For $\w_\tau$ as given in the proposition, let 
\[\Omega_{\w_\tau}=\red{\bigcap_{k\geqslant1}\limsup_n\Omega^k_n(\w_\tau)}=\bigcap_{k\geqslant1}\bigcap_{l\geqslant1}\bigcup_{n\geqslant l}\Omega^k_n(\w_\tau).\]
Then, the set $\Omega_{\w_\tau}$ consists of points with the property that, for arbitrary $\epsilon>0$, there exist arbitrarily large $n$ such that $\hat\bxi_n$ lies in $\tau$ and $(\hat\bxi_n-\x^*)/\|\hat\bxi_n-\x^*\|$ is within a distance $\epsilon$ of $\mathcal{H}_{\w_\tau}$. Since $\Omega_n^k(\w_\tau)\supseteq \Omega_n^{k+1}(\w_\tau)$, the required result is equivalent to showing that $\red{{\bf P}}(\Omega_{\w_\tau})=0$. 
 
Without loss of generality, we may assume that, restricted to $\Omega_{\w_\tau}$, $\hat\bxi_n$ lie in $\tau$ for all $n$ and $\w_n=(\hat\bxi_n-\x^*)/\|\hat\bxi_n-\x^*\|\rightarrow\w$ as $n\rightarrow\infty$ for some (random) unit vector $\w\in\mathcal{H}_{\w_\tau}$.

\red{Recall that,} for given $\w_\tau\in\S^{l-l'}_{\tau\setminus\sigma}\red{\setminus\Theta_{\tau,\sigma}(\x^*;\mu)}$, each $\Psi_\tau(\bxi_i,\w_\tau;\x^*)$ is a Euclidean random variable on $\mathbb{R}(E\cup F)$. Then, let 
\[\hat\bxi^{\w_\tau}_n=\frac{1}{n}\sum_{i=1}^n\Psi_\tau(\bxi_i,\w_\tau;\x^*)\]
and write $\Omega_0$ for the subset \red{of $\bOmega$} consisting of points such that $\hat\bxi^{\w_\tau}_n$ converges to
\[\int_{\X^m}\Psi_\tau(\x,\w_\tau;\x^*)\,d\mu(\x).\] 
It follows from the classical Law of Large Numbers that $\red{{\bf P}}(\Omega_0)=1$. Hence, restricted to $\Omega_{\w_\tau}\cap\Omega_0$, the assumption on \red{$\w_\tau$} implies that, for some constant $c<0$, there is an $n_0$ such that, for $n>n_0$, $\langle\w_\tau,\,\hat\bxi^{\w_\tau}_n\rangle<c$. However, the assumption that $\w_n\rightarrow\w\in\mathcal{H}_{\w_\tau}$ implies that $\langle\w_\tau,\,\hat\bxi_n-\x^*\rangle>0$ for all sufficiently large $n$. Putting these two conclusions together, we have that, restricted to $\Omega_{\w_\tau}\cap\Omega_0$, 
\begin{eqnarray}
\left\langle\w_\tau,\hat\bxi_n-\hat\bxi^{\w_\tau}_n\right\rangle=\left\langle\w_\tau,\hat\bxi_n-\x^*\right\rangle-\left\langle\w_\tau,\hat\bxi^{\w_\tau}_n\right\rangle>-c>0
\label{eqn7p}
\end{eqnarray}
as $\langle\w_\tau,\x^*\rangle=0$.

On the other hand, restricted to $\Omega_{\w_\tau}\cap\Omega_0$, $\hat\bxi_n$ is in $\tau$ by the assumption made earlier. Then, it follows from \eqref{eqn5p} that
\[\hat\bxi_n=\frac{1}{n}\sum_{i=1}^n\Phi_\tau(\bxi_i;\hat\bxi_n).\]
Thus, we can express the difference $\hat\bxi_n-\hat\bxi^{\w_\tau}_n$ as
\begin{eqnarray}
\hat\bxi_n-\hat\bxi^{\w_\tau}_n=\frac{1}{n}\sum_{i=1}^n\left\{\Phi_\tau(\bxi_i;\hat\bxi_n)-\Psi_\tau(\bxi_i,\w_\tau;\x^*)\right\}.
\label{eqn7a}
\end{eqnarray}
Decompose $\w_n=(\w_n)_\sigma+(\w_n)^\perp$, where $(\w_n)_\sigma=P_\sigma(\w_n)$ and $(\w_n)^\perp=P_{\tau\setminus\sigma}(\w_n)$. Then, by Corollary \ref{cor0a}$(ii)$, for each $1\leqslant i\leqslant n$, 
\begin{eqnarray}
\begin{array}{rcl}
&&\Phi_\tau(\bxi_i;\hat\bxi_n)-\Psi_\tau(\bxi_i,\w_\tau;\x^*)\\
%&=&\Phi_\tau(\bxi_i;\hat\bxi_n)-\Psi_\tau(\bxi_i,\w_n;\x^*)+\Psi_\tau(\bxi_i,\w_n;\x^*)-\Psi_\tau(\bxi_i,\w_\tau;\x^*)\\
&=&\Phi_\tau(\bxi_i;\hat\bxi_n)-\Psi_\tau(\bxi_i,\w_n;\x^*)+\Psi_\tau(\bxi_i,(\w_n)_\tau;\x^*)-\Psi_\tau(\bxi_i,\w_\tau;\x^*),
\end{array}
\label{eqn7b}
\end{eqnarray}
where $(\w_n)_\tau=(\w_n)^\perp/\|(\w_n)^\perp\|\in\S^{l-l'}_{\tau\setminus\sigma}$. \red{Without loss of generality,} we assume that the carriers of the geodesics from $\hat\bxi_n$ to $\bxi_i$ remain constant. \red{The general case follows from a similar inductive argument to that outlined in the beginning of the proof of Proposition \ref{prop6} and from the fact that $\hat\bxi_n$ converges to $\x^*$ a.s. Then, if $(\mathcal{A},\mathcal{B})$ is the common support of the geodesics, where $\mathcal{A}=(A_0,\cdots,A_k)$ and $\mathcal{B}=(B_0,\cdots,B_k)$, and, for $0<j\leqslant k$, writing $W_j$ for \red{$P_{A_j\cap E}(\x^*)$} if $A_j\cap E\not=\emptyset$ and otherwise \red{$P_{A_j\cap F}(\w_n)$}, Theorem \ref{thm1} tells us} that the $j$th set of components of $(\jmath^{-1})(\Phi(\bxi_i,\hat\bxi_n))$ is the vector \red{$-\frac{\|P_{B_j}(\bxi_i)\|}{\|P_{A_j}(\hat\bxi_n)\|}P_{A_j}(\hat\bxi_n)$} and, for $(\jmath^{-1})(\Psi(\bxi_i,\w_n;\x^*))$, \red{Theorem \ref{thm2}$(ii)$ tells us that} the corresponding vector is $-\frac{\red{\|P_{B_j}(\bxi_i)\|}}{\|W_j\|}W_j$. \red{Hence, the proof of Theorem \ref{thm2}$(ii)$ shows that}, when $A_j\cap E=\emptyset$, \red{these two vectors are identical,} since \red{$P_{A_j}(\hat\bxi_n)=P_{A_j\cap F}(\hat\bxi_n)=P_{A_j\cap F}(\hat\bxi_n-\x^*)$.} \red{While, if} $A_j\cap E\not=\emptyset$, \red{the difference between these two vectors is \red{of the same order as} $\frac{P_{A_j\cap F}(\hat\bxi_n)}{\|P_{A_j}(\hat\bxi_n)\|}$} whose limit, as $n\rightarrow\infty$, is zero since \red{$\|P_{A_j}(\hat\bxi_n)\|\geqslant\|P_{A_j\cap E}(\hat\bxi_n)\|\rightarrow\|P_{A_j\cap E}(\x^*)\|>0$ but $P_{A_j\cap F}(\hat\bxi_n)\rightarrow0$ a.s.} It follows that, as $n\rightarrow\infty$,   
\begin{eqnarray}
\left\langle\w_\tau,\,\Phi_\tau(\bxi_i;\hat\bxi_n)-\Psi_\tau(\bxi_i,\w_n;\x^*)\right\rangle\longrightarrow0\quad\hbox{ a.s.}
\label{eqn7c}
\end{eqnarray}
Moreover, since $\w^\perp=P_{\tau\setminus\sigma}(\w)\not=0$, $\w_n\rightarrow\w$ implies that $(\w_n)_\tau\rightarrow\frac{\w^\perp}{\|\w^\perp\|}=\w_\tau$. Then, it follows from a similar argument to that of the proof of Proposition \ref{prop2} that, for sufficiently large $n$, 
\begin{eqnarray*}
&&\left\langle\w_\tau,\,\Psi_\tau(\bxi_i,(\w_n)_\tau;\x^*)-\Psi_\tau(\bxi_i,\w_\tau;\x^*)\right\rangle\\
&\approx&\left\langle\w_\tau,\,\left(D\Psi_\tau(\bxi_i,\w_\tau;\x^*)\right)\left(\arccos(\langle(\w_n)_\tau,\w_\tau\rangle)\frac{\v_n}{\|\v_n\|}\right)\right\rangle,
\end{eqnarray*}
where $\v_n$ is the component of $(\w_n)_\tau-\w_\tau$ orthogonal to $\w_\tau$, so that as $n\rightarrow\infty$,  
\begin{eqnarray}
\left\langle\w_\tau,\,\Psi_\tau(\bxi_i,(\w_n)_\tau;\x^*)-\Psi_\tau(\bxi_i,\w_\tau;\x^*)\right\rangle\longrightarrow0\quad\hbox{ a.s.,}
\label{eqn7d}
\end{eqnarray}
by Proposition \ref{prop2}. Then, \eqref{eqn7a}, \eqref{eqn7b}, \eqref{eqn7c} and \eqref{eqn7d} together imply that, when it is restricted to $\Omega_{\w_\tau}\cap\Omega_0$, $\langle\w_\tau,\hat\bxi_n-\hat\bxi^{\w_\tau}_n\rangle\rightarrow0$ a.s., as $n\rightarrow\infty$, contradicting \eqref{eqn7p}. \red{Hence, ${\bf P}(\Omega_{\w_\tau}\cap\Omega_0)=0$, so that ${\bf P}(\Omega_{\w_\tau})=0$ as  required.}
\end{proof}

When $l-l'=1$, $\S^{l-l'}_{\tau\setminus\sigma}$ contains a single unit vector, so that we have the following special case. In particular, taking $l=1$ and so $l'=0$ recovers the result of Lemma 6 in \cite{BLO2} for the case of co-dimension one when $\X^m$ is a tree space. 

\begin{corollary}
Let the stratum $\sigma$ of co-dimension $l(\geqslant1)$ bound, in $\X^m$, the stratum $\tau$ of co-dimension $l'=l-1$. Assume that the Fr\'echet mean $\x^*$ of $\mu$ lies in $\sigma$. If the inequality \eqref{eqn4p} corresponding to the unique $\w_\tau\in\S^{l-l'}_{\tau\setminus\sigma}$ is strict then, for all sufficiently large $n$, $\hat\bxi_n$ cannot lie in $\tau$.
\label{cor1}
\end{corollary}

Thus, when $l-l'=1$, the support of the limiting distribution of any appropriately scaled difference $\hat\bxi_n-\x^*$ intersects the stratum $\mathbb{R}(E)\times\O(F)$ in the tangent cone to $\X^m$ at $\x^*$ only if the inequality \eqref{eqn4p} corresponding to the unique $\w_\tau\in\S^{l-l'}_{\tau\setminus\sigma}$ is an equality.  

\vskip 6pt
Similar to the case where $l-l'=1$, Proposition \ref{prop4} has the following consequence on the support of the limiting distribution when $l-l'>1$, where $\mathcal{C}(\Theta)$ denotes the Euclidean cone on $\Theta$. 

\begin{corollary}
Let the stratum $\sigma=\O(E)$ of co-dimension $l(\geqslant2)$ bound, in $\X^m$, the stratum $\tau=\O(E\cup F)$ of co-dimension $l'\leqslant l-2$. Assume that $\x^*\in\sigma$ is the Fr\'echet mean of $\mu$. Then the support of the limiting distribution of an appropriately scaled difference $\hat\bxi_n-\x^*$, if it meets the stratum $\mathbb{R}(E)\times\O(F)$ in the tangent cone to $\X^m$ at $\x^*$, must be contained in $\mathbb{R}(E)\times\mathcal{C}(\Theta_{\tau,\sigma}(\x^*;\mu))$, where $\Theta_{\tau,\sigma}(\x^*;\mu)$ is defined by \eqref{eqn7e}. 
\label{cor2}
\end{corollary}

Hence, the support of the limiting distribution of an appropriately scaled difference $\hat\bxi_n-\x^*$ is contained in $\mathcal{K}_\mu$ where, for the closed sets
\begin{eqnarray}
\mathcal{K}_{\mu,\tau}=\mathbb{R}(E)\times\mathcal{C}\left(\Theta_{\tau,\sigma}(\x^*;\mu)\right)
\label{eqn12}
\end{eqnarray}
in the tangent cone to $\X^m$ at $\x^*$, 
\begin{eqnarray}
\mathcal{K}_\mu=\bigcup_{\tau\text{ co-bounds }\sigma}\mathcal{K}_{\mu,\tau}
\label{eqn13}
\end{eqnarray} 
and where we regard $\sigma$ as co-bounding itself. Nevertheless, the following example shows that
\begin{enumerate}
\item[$(i)$] if it is non-empty, $\mathbb{R}(E)\times\mathcal{C}(\Theta_{\tau,\sigma}(\x^*;\mu))$ is not necessarily an entire stratum $\mathbb{R}(E)\times\O(F)$; 
\item[$(ii)$] even if it is the entire stratum, the support of the limiting distribution of $\sqrt{n}(\hat\bxi_n-\x^*)$ does not necessarily intersect that stratum; and
\item[$(iii)$] it is possible that the support of the limiting distribution, when restricted to the stratum, is only a subset of $\mathbb{R}(E)\times\mathcal{C}(\Theta_{\tau,\sigma}(\x^*;\mu))$.
\end{enumerate}

\begin{example}\label{ex4} 
Consider the orthant space $\X^2$ \red{of Example \ref{ex2}}. Let $\mu$ have mass $1/2$ at the two points $p_1$ and $p_2$ equidistant from the cone point $o$ along a geodesic through that point \red{as illustrated in Figure \ref{fig3}}. 
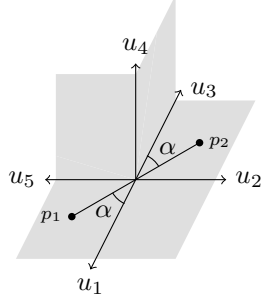
\begin{figure}
\begin{center}
\begin{tikzpicture} [scale=0.7]
\path[fill=gray!25] (0,0) -- (0.75,1.5) -- (0.75,3.5) -- (0.4,2.8) -- cycle;
\path[fill=gray!25] (0,0) -- (0.4,2.8) -- (0,2) -- cycle;
\path[fill=gray!25] (0,0) -- (0,2) -- (-1.5,2) -- (-1.5,0.45) -- cycle;
\path[fill=gray!25] (0,0) -- (-1.5,0.45) -- (-1.5,0) -- cycle;
\path[fill=gray!25] (0,0) -- (0.75,1.5) -- (2.25,1.5) -- (1.5,0) -- cycle;
\path[fill=gray!25] (0,0) -- (-0.75,-1.5) -- (0.75,-1.5) -- (1.5,0) -- cycle;
\path[fill=gray!25] (0,0) -- (-0.75,-1.5) -- (-2.25,-1.5) -- (-1.5,0) -- cycle;
\draw[->] (0,0) -- (1.7,0) node[right] {$u_2$};
\draw[->] (0,0) -- (0,2.2) node[above] {$u_4$};
\draw[->] (0,0) -- (-1.7,0) node[left] {$u_5$};
\draw[->] (0,0) -- (0.85,1.7) node[above, right] {$u_3$};
\draw[->] (0,0) -- (-0.85,-1.7) node[below] {$u_1$};
\draw (-1.2,-0.7) -- (1.2,0.7);
\fill (-1.2,-0.7) circle (2pt) node[below,left] {$\scriptstyle{p_1}$};
\fill (1.2,0.7) circle (2pt)node[above,right] {$\scriptstyle{p_2}$};
\draw(245:0.5) arc (245:210:0.5) node at (-0.6,-0.6) {$\alpha$};
\draw(65:0.5) arc (65:30:0.5) node at (0.6,0.6) {$\alpha$};
\end{tikzpicture}
\end{center}
\caption{Probability measure $\mu$ on $\X^2$ has mass $1/2$ at $p_1$ and $p_2$.}
\label{fig3}
\end{figure}
Then its Fr\'echet mean is at the cone point and the sample Fr\'echet means always lie on this geodesic segment. \red{This, in particular, implies that the support of the limiting distribution of $\sqrt{n}\{\hat\bxi_n-o\}$ is the union of the cone point with two half lines, one in $\{o\}\times\tau^{\phantom{A}}_{1,5}$ and other in $\{o\}\times\tau^{\phantom{A}}_{2,3}$, each extending the relevant geodesic segment, where $\tau^{\phantom{A}}_{i,j}=\O(u_i,u_j)$.}
\begin{enumerate}
\item[($a$)] For any direction $\w_{\tau^{\phantom{A}}_{1,2}}\in\S^2_{\tau^{\phantom{A}}_{1,2}\setminus\{o\}}$, \red{$\Psi_{\tau^{\phantom{A}}_{1,2}}(p,\w_{\tau^{\phantom{A}}_{1,2}};o)$ lies in the plane spanned by the orthant $\tau^{\phantom{A}}_{1,2}$, for any $p$, and, by identifying $u_3$ and $u_5$ with $-u_1$ and $-u_2$ respectively, $\Psi_{\tau^{\phantom{A}}_{1,2}}(p_i,\w_{\tau^{\phantom{A}}_{1,2}};o)=p_i$ for $i=1,2$. Thus,} $\int_{\X^2}\Psi_{\tau^{\phantom{A}}_{1,2}}(p,\w_{\tau^{\phantom{A}}_{1,2}};o)\,d\mu(p)=0$ and so $\Theta_{\tau^{\phantom{A}}_{1,2},\{o\}}(o;\mu)=\S^2_{\tau^{\phantom{A}}_{1,2}\setminus\{o\}}$. \red{Since, in this case, the support of the limiting distribution does not intersect $\{0\}\times\tau^{\phantom{A}}_{1,2}$, this illustrates $(ii)$ above with $\sigma=\{o\}$ and $\tau=\tau^{\phantom{A}}_{1,2}$.}
\item[($b$)] For any direction $\w_{\tau^{\phantom{A}}_{1,5}}\in\S^2_{\tau^{\phantom{A}}_{1,5}\setminus\{o\}}$ \red{such that the angle between $\w_{\tau^{\phantom{A}}_{1,5}}$ and $u_1$-axis is less than or equal $\alpha$, a similar argument shows that}
\[\int_{\X^2}\Psi_{\tau^{\phantom{A}}_{1,5}}(p,\w_{\tau^{\phantom{A}}_{1,5}};o)\,d\mu(p)=0.\]
Hence, \red{such $\w_{\tau^{\phantom{A}}_{1,5}}$ are always contained in $\Theta_{\tau^{\phantom{A}}_{1,5},\{o\}}(o;\mu)$, i.e.} 
\[\Theta_{\tau^{\phantom{A}}_{1,5},\{o\}}(o;\mu)\supseteq\{\theta\in\S^2_{\tau^{\phantom{A}}_{1,5}\setminus\{o\}}\mid\theta\leqslant\alpha\}\]
\red{where $\theta\in\S^2_{\tau^{\phantom{A}}_{1,5}}\setminus\{o\}$ is measured from the $u_1$-axis.} 
\item[($c$)] \red{However,} for any direction $\w_{\tau^{\phantom{A}}_{1,5}}\in\S^2_{\tau^{\phantom{A}}_{1,5}\setminus\{o\}}$ \red{such that the angle between $\w_{\tau^{\phantom{A}}_{1,5}}$ and the $u_1$-axis is greater than $\alpha$, the vector $\Psi_{\tau^{\phantom{A}}_{1,5}}(p_1,\w_{\tau^{\phantom{A}}_{1,5}};o)=p_1$, but the vector $\Psi_{\tau^{\phantom{A}}_{1,5}}(p_2,\w_{\tau^{\phantom{A}}_{1,5}};o)$ lies on the line spanned by the unit vector $\w_{\tau^{\phantom{A}}_{1,5}}$ in $(u_1,u_5)$-plane. Hence, these two vectors do not lie on the same line in the $(u_1,u_5)$-plane through the origin. This gives}
\[\langle\w_{\tau^{\phantom{A}}_{1,5}}, \int_{\X^2}\Psi_{\tau^{\phantom{A}}_{1,5}}(p,\w_{\tau^{\phantom{A}}_{1,5}};o)\,d\mu(p)\rangle<0.\]
\red{Hence, if the angle between $\w_{\tau^{\phantom{A}}_{1,5}}$ and the $u_1$-axis is greater than $\alpha$,} then $\w_{\tau^{\phantom{A}}_{1,5}}\not\in\Theta_{\tau^{\phantom{A}}_{1,5},\{o\}}(o;\mu)$. Combining this with the conclusion $(b)$ shows that $\Theta_{\tau^{\phantom{A}}_{1,5},\{o\}}(o;\mu)=\{\theta\in\S^2_{\tau^{\phantom{A}}_{1,5}\setminus\{o\}}\mid\theta\leqslant\alpha\}$, illustrating $(i)$ and $(iii)$ above with $\sigma=\{o\}$ and $\tau=\tau^{\phantom{A}}_{1,5}$.
\end{enumerate}
\end{example}

\subsection{\red{The limiting distribution}}

To describe the limiting distribution of $\sqrt{n}(\hat\bxi_n-\x^*)$, where the Fr\'echet mean $\x^*$ of $\mu$ lies in a stratum $\sigma=\O(E)$ of \red{local} co-dimension $l\geqslant0$, we continue to regard $\sigma$ as co-bounding itself so that, in this case, the set $F$ of additional axes in the `co-bounding' stratum is empty. Moreover, we shall relate the form of the limiting distribution in the set \eqref{eqn12} for each $\tau$ co-bounding $\sigma$ to a limiting distribution of the Euclidean means of various Euclidean random variables depending on $\tau$:
\begin{enumerate}
\item[$(i)$] for $\tau=\sigma$, corresponding to the set $\mathbb{R}(E)\times\{\mathbf{0}\}$ in \eqref{eqn12}, the relevant Euclidean random variable is $\Phi_\sigma(\bxi_1;\x^*)$; 
\item[$(ii)$] for $\tau\not=\sigma$ the relevant Euclidean random variable is $\Psi_\tau(\bxi_1,\w_\tau;\x^*)$ where, if $l-l'>1$, $\w_\tau$ is any chosen vector in int$\left(\Theta_{\tau,\sigma}(\x^*;\mu)\right)$ if this set is not empty and, if $l-l'=1$ with $\Theta_{\tau,\sigma}(\x^*;\mu)\not=\emptyset$, $\w_\tau$ is its unique element; 
\item[$(iii)$] we take the zero random variable otherwise. 
\end{enumerate}
Note that, by Proposition \ref{prop6}, different choices of $\w_\tau$ in the case $l-l'>1$ of $(ii)$ give random variables that are a.s. equal. Note also that, by Corollary \ref{cor2}, the random variables in the case $(iii)$ play no role in the description of the limiting distribution so that they can be replaced by any other random variables. For simplicity, we denote the relevant random variable above in each case by $\tilde\Psi_\tau(\bxi_1;\x^*)$. With this context and notation write, for each $\tau$ \red{co-bounding $\sigma$}, 
\begin{eqnarray}
A^{-1}_{\sigma,\tau}=\red{U^\top_\tau\left\{I_M-E\left[M_{\x^*}^\sigma(\bxi_1)\right]\right\}U_\tau,}
\label{eqn10}
\end{eqnarray}
\red{where $M_{\x^*}^\sigma(x)$ is defined by \eqref{eqn4a}, $U_\tau$ is the $M\times(m-l')$ matrix whose entries are all zero except for those at $(l_i,i)$ being one, and $u^{\phantom{A}}_{l_1},\cdots,u^{\phantom{A}}_{l_{m-l'}}$ are the ordered axes that span $\jmath(\mathbb R(E\cup F))$.} Note that, since $M_{\x^*}^\sigma(\x)$ is negative semi-definite,  the above inverse is well defined when $E[M_{\x^*}^\sigma(\bxi_1)]$ exists. Then, letting $Z_\tau$ be a random variable in $\mathbb{R}(E\cup F)$ with normal distribution $N(0,A_{\sigma,\tau}^\top V_\tau A_{\sigma,\tau})$, where $V_\tau={\rm cov}(\tilde\Psi_\tau(\bxi_1;\x^*))$, we have the following result.  

\begin{theorem}
Let $\sigma=\O(E)$ be a stratum in $\X^m$ of co-dimension $l(\geqslant0)$. Assume that 
\begin{enumerate}
\item[$(i)$] the Fr\'echet mean $\x^*$ of $\mu$ lies in $\sigma$; 
\item[$(ii)$] $\mu(\mathcal{D}_{\x^*})=0$, where $\mathcal{D}_{\x^*}$ is given by Definition $\ref{def1g}$; 
\item[$(iii)$] $E\left[M^\sigma_{\x^*}(\bxi_1)\right]$ exists, where $M^\sigma_{\x^*}(\x)$ is given by \eqref{eqn4a}; 
\item[$(iv)$] for any stratum $\tau$ in $\X^m$ which co-bounds $\sigma$ and has co-dimension $l'\leqslant l-2$, if $\Theta_{\tau,\sigma}(\x^*;\mu)\not=\emptyset$ then ${\rm int}(\Theta_{\tau,\sigma}(\x^*;\mu))\not=\emptyset$. 
\end{enumerate}
Then, if \red{there exists a random variable $\eta$ on the tangent cone at $\x^*$ such that}
\[\sqrt{n}\{\hat\bxi_n-\x^*\}\buildrel{d}\over\longrightarrow\eta,\quad\hbox{ as }n\rightarrow\infty,\]
\red{then} $\eta$ has the following property: for any stratum $\tau=\O(E\cup F)$ of co-dimension $l'(\leqslant l)$ co-bounding $\sigma$, if ${\bf P}\left(\eta\in\mathbb{R}(E)\times\O(F)\right)>0$ then, for $Z_\tau$ defined as above and $\mathcal{K}_\mu$ by \eqref{eqn13},
\[{\bf P}(\eta\in B)={\bf P}(Z_\tau\in B)\]
for any Borel set $B$ contained in 
\[\begin{cases}
\left({\rm int}\left(\mathbb{R}(E)\times\mathcal{C}(\Theta_{\tau,\sigma}(\x^*;\mu))\right)\right)\setminus\partial\mathcal{K}_\mu&\hbox{ if }l'\leqslant l-2\cr
\left(\mathbb{R}(E)\times\O(F)\right)\setminus\partial\mathcal{K}_\mu&\hbox{ if }l'=l-1\hbox{ or }l.%\cr  
%\mathbb{R}(E)&\hbox{ if }l'=l.
\end{cases}\]
\label{thm4}
\end{theorem}

\begin{proof}
We assume that $l'\leqslant l-2$. The case for $l'=l-1$ can be similarly derived by noting Corollary \ref{cor1}, whereas for $l'=l$ the result  can be derived directly by simplifying the following arguments.

Write $\Xi_\tau=\O(E)\times\mathcal{C}(\hbox{int}(\Theta_{\tau,\sigma}(\x^*;\mu)))$. \red{By Corollary \ref{cor1},} given $\hat\bxi_n\in\tau$, $\hat\bxi_n\in\Xi_\tau$ for sufficiently large $n$ and, by Theorem \ref{thm3}, we also have
\begin{eqnarray*}
\sqrt{n}(\hat\bxi_n-\x^*)&=&\frac{1}{\sqrt{n}}\sum_{i=1}^n\left\{\Phi_\tau(\bxi_i;\hat\bxi_n)-\tilde\Psi_\tau(\bxi_i;\x^*)\right\}\\
&&+\frac{1}{\sqrt{n}}\sum_{i=1}^n\left\{\tilde\Psi_\tau(\bxi_i;\x^*)-\x^*\right\}.
\end{eqnarray*}
For any $\x'\in\tau$ and $\x\in\X^m$, denote the projection $P_{\tau\setminus\sigma}\left(\Phi_\tau(\x;\x')\right)$ of $\Phi_\tau(\x;\x')$ by $\Phi_{\tau\setminus\sigma}(\x;\x')$. Define $\tilde\Psi_{\tau\setminus\sigma}(\x;\x^*)$ similarly. Then, $\tilde\Psi_{\tau\setminus\sigma}(\x;\x^*)=\tilde\Psi_\tau(\x;\x^*)-\Phi_\sigma(\x;\x^*)$ by Corollary \ref{cor00}($ii$). Since $\hat\bxi_n$ is in $\Xi_\tau$ and converges to $\x^*$ a.s., the result of Proposition \ref{prop6} and the argument for the proof of Theorem \ref{thm2} together imply that, for any given $\x$ and all sufficiently large $n$, $\Phi_{\tau\setminus\sigma}(\x;\hat\bxi_n)=\tilde\Psi_{\tau\setminus\sigma}(\x;\x^*)$ a.s.. Hence, in particular, for sufficiently large $n$, 
\[P_{\tau\setminus\sigma}\left(\hat\bxi_n-\x^*\right)=\frac{1}{n}\sum_{i=1}^n\Phi_{\tau\setminus\sigma}(\bxi_i;\hat\bxi_n)\]
is a.s. the Euclidean mean of $\tilde\Psi_{\tau\setminus\sigma}(\bxi_1;\x^*),\cdots,\tilde\Psi_{\tau\setminus\sigma}(\bxi_n;\x^*)$, so that
\[\frac{1}{\sqrt n}\sum_{i=1}^n\left\{\Phi_{\tau\setminus\sigma}(\bxi_i;\hat\bxi_n)-\tilde\Psi_{\tau\setminus\sigma}(\bxi_i;\x^*)\right\}1_{\Xi_\tau}(\hat\bxi_n)\buildrel{\textbf{P}}\over\longrightarrow0.\]
Thus, the limiting distribution of $\sqrt{n}(\hat\bxi_n-\x^*)1_{\Xi_\tau}(\hat\bxi_n)$ is the same as that of
\[\frac{1}{\sqrt{n}}\sum_{i=1}^n\!\left\{\Phi_\sigma(\bxi_i;\hat\bxi_n)-\tilde\Psi_\sigma(\bxi_i;\x^*)\right\}1_{\Xi_\tau}(\hat\bxi_n)+\frac{1}{\sqrt{n}}\sum_{i=1}^n\!\left\{\tilde\Psi_\tau(\bxi_i;\x^*)-\x^*\right\}1_{\Xi_\tau}(\hat\bxi_n).\]
Since $\tilde\Psi_\sigma(\bxi_i;\x^*)=\Phi_\sigma(\bxi_i;\x^*)$, Proposition \ref{prop1} implies that the limiting distribution of $\sqrt{n}(\hat\bxi_n-\x^*)1_{\Xi_\tau}(\hat\bxi_n)$ is equal to that of
\[\sqrt{n}P_\sigma(\hat\bxi_n-\x^*)\,1_{\Xi_\tau}(\hat\bxi_n)\frac{1}{n}\sum_{i=1}^nM_{\x^*}^\sigma(\bxi_i)+1_{\Xi_\tau}(\hat\bxi_n)\frac{1}{\sqrt{n}}\sum_{i=1}^n\left\{\tilde\Psi_\tau(\bxi_i;\x^*)-\x^*\right\}.\]
Hence, by \eqref{eqn11p}, the required result follows from a similar argument to that used in \cite{BLO1} and \cite{BLO2}.
\end{proof}

As for $\Theta_\sigma(\x^*;\mu)$ defined by \eqref{eqn7f}, the convexity in $\w$ of the directional derivative $D(d_{\x}^2)(\w)$ %(cf. \cite{BK}, pp416-417)
implies that $\mathcal{K}_\mu$ is a convex subset of the tangent cone to $\X^m$ at $\x^*$. This, together with the structure of an orthant space, implies that the result of Theorem \ref{thm4} refers to the behaviour of the limiting distribution only within the interior of $\mathcal{K}_\mu$. Its behaviour at the boundaries will depend on how these sets relate to each other and on the shape of the boundary $\partial\mathcal{K}_\mu$.

\vskip 6pt 
The assumption in Theorem \ref{thm4} that $\mu(\mathcal{D}_{\x^*})=0$ ensures that we are able to employ the so-called delta method for the approximate probability distribution of a function of an asymptotically normal statistical estimator. In principle, it is possible to relax this assumption by using directional derivatives and combining that with the use of the law of the total probability. However, it is clear from the definition of $\mathcal{D}_{\x^*}$ that its structure, although conceptually straightforward, is generally more complex than will admit a simple algebraic representation, and the ensuing results will consequently depend heavily on the behaviour of $\mu$ on $\mathcal{D}_{\x^*}$.

\vskip 6pt
To observe special cases of Theorem \ref{thm4}, let $\sigma=\O(E)$ be a stratum in $\X^m$ of co-dimension $l(\geqslant0)$ in which the Fr\'echet mean $\x^*$ of $\mu$ lies, assume that the conditions of Theorem \ref{thm4} are satisfied and write
\[l(\mu)=\inf\{l'\mid l'=\hbox{co-dimension of }\tau,\hbox{ where }\Theta_{\tau,\sigma}(\x^*;\mu)\not=\emptyset\},\]
where we assume that $l(\mu)=l$ if there is no $\tau$ with co-dimension $l'<l$ which satisfies the above required condition. We assume further that, for $\tau=\O(E\cup F)$ of co-dimension $l(\mu)$ co-bounding $\sigma$ and, if $l(\mu)<l$, with $\Theta_{\tau,\sigma}(\x^*;\mu)\not=\emptyset$, $V_\tau=\hbox{cov}(\tilde\Psi_\tau(\bxi_1;\x^*))$ is of full rank $m-l(\mu)$. Then, it is clear from the proof of Theorem \ref{thm4} that ${\bf P}(\eta\in\mathbb{R}(E)\times\O(F))>0$.    

\vskip 6pt
\textit{Case $l(\mu)=l$:} in this case, $\mathcal{K}_\mu=\mathbb{R}(E)$ and the support of the distribution of $\eta$ is contained in the tangent space of $\sigma$. Then, Theorem \ref{thm4} says that $\eta$ is a normal random variable with mean zero and covariance matrix $A_{\sigma,\sigma}^\top\hbox{cov}(\Phi_\sigma(\bxi_1;\x^*))\,A_{\sigma,\sigma}$, \red{where $A_{\sigma,\tau}$ is defined by \eqref{eqn10}.} This generalises the limiting distribution of $\sqrt{n}\{\hat\bxi_n-\x^*\}$ when $\x^*$ lies in a top-dimensional stratum of a tree space obtained in \cite{BLO2}. 

\vskip 6pt
\textit{Case $l(\mu)=l-1$ so that $l\geqslant1$}: if $\tau=\O(E\cup F)$ is a stratum of co-dimension $l'=l-1$ such that $\Theta_{\tau,\sigma}(\x^*;\mu)\not=\emptyset$, then $F$ contains only one axis. By taking the Borel set $B=\mathbb{R}(E)\times\O(F)$, we see that ${\bf P}\left(\eta\in\mathbb{R}(E)\times\O(F)\right)=1/2$ since the corresponding $Z_\tau$ is a normal random variable in $\mathbb{R}^{m-l+1}$ with mean zero. Hence, there are at most two strata of co-dimension $l(\mu)$ co-bounding $\sigma$ on which infinitely many $\hat\bxi_n$ lie. Moreover, in the case of there being only one such a stratum, ${\bf P}(\eta\in\sigma)=1/2$ and, in case of two such strata, ${\bf P}(\eta\in\sigma)=0$. 

\vskip 6pt
\textit{Case that $0\leqslant l(\mu)<l$, that there is a single $\tau_0=\O(E\cup F_0)$ such that the co-dimension of $\tau_0$ is $l(\mu)$ and that $\Theta_{\tau_0,\sigma}(\x^*;\mu)=\mathcal{S}^{l-l(\mu)}_{\tau_0\setminus\sigma}$}: in this case, we have the following full description of the distribution of $\eta$ in terms of $\phi_{\tau_0}$, the probability density function of the random variable $Z_{\tau_0}$ defined prior to Theorem \ref{thm4}. We first note that, since $\mathcal{K}_\mu$ defined by \eqref{eqn13} is convex and closed, the result of Proposition \ref{prop6} implies that, in this case, 
\[\mathcal{K}_\mu=\bigcup_{F\subseteq F_0}\mathbb{R}(E)\times\O(F)\red{=\mathbb{R}(E)\times\overline{\O(F)}}.\]
Then, we extend the projection map $P$ to $\mathbb{R}(E\cup F_0)$ in an obvious fashion and, for any $\tau=\O(E\cup F)$, where $F\subseteq F_0$, and any $\z\in\mathbb{R}(E\cup F_0)$, write $\z_\tau=P_\tau(\z)$ and $\z_{\tau_0\setminus\tau}=P_{\tau_0\setminus\tau}(\z)=\z-\z_\tau$. 

\begin{proposition}
Under the above assumptions and notation, the limiting distribution of $\sqrt{n}\{\hat\bxi_n-\x^*\}$ is given as follows: for any $\tau=\O(E\cup F)$, where $F\subseteq F_0$, and any Borel subset $B\subseteq\mathbb{R}(E)\times\O(F)$, 
\[{\bf P}(\eta\in B)=\int_B\psi_F(\z_\tau)\,d\z_\tau,\]
where
\[\psi_F(\z_\tau)=\int_{-\infty}^0\phi_{\tau_0}(\z)\,d\z_{\tau_0\setminus\tau}.\]
\label{prop7}
\end{proposition}

The special case that $l(\mu)=l-1$ of this Proposition, together with the comments in the previous two paragraphs, generalises the limiting distribution of $\sqrt{n}\{\hat\bxi_n-\x^*\}$ when $\X^m$ is a tree space and $\x^*$ lies in a stratum of co-dimension one obtained in \cite{BLO2}. 

\begin{proof}
By Theorem \ref{thm4}, we only need to consider the case where $F\not=F_0$. Assume that $\tau=\O(E\cup F)$ has co-dimension $l'$ and fix $\w_\tau\in S^{l-l'}_{\tau\setminus\sigma}$. \red{We first show that}
\begin{eqnarray}
P_\tau\left(\tilde\Psi_{\tau_0}(\bxi_1;\x^*)\right)=\tilde\Psi_\tau(\bxi_1;\x^*)\qquad\hbox{ a.s. }
\label{eqn11}
\end{eqnarray}

\red{Recall, from the proof of Theorem \ref{thm2}, that} the geodesics from $\x^*(\lambda,\w_{\tau_0})=\x^*+\lambda\w_{\tau_0}$ to $\x$ have the same support for all $\w_{\tau_0}\in S^{l-l(\mu)}_{\tau_0\setminus\sigma}$ sufficiently close to $\w_\tau$ and all sufficiently small $\lambda>0$. For such $\w_{\tau_0}$ and $\lambda$, by Definition \ref{def2}, Corollary \ref{cor00}$(i)$, Propositions \ref{prop5b} and \ref{prop6}, the sequence $\mathcal{A}=(A_0,\cdots,A_k)$ in the support $(\mathcal{A},\mathcal{B})$ of the geodesics from $\x^*(\lambda,\w_{\tau_0})=\x^*+\lambda\w_{\tau_0}$ to $\bxi_1$ has the property that, if $i>0$ and if $A_i\cap E=\emptyset$, then $A_i$ consists of a single axis \red{in $F_0$} a.s., so that the \red{$P_{A_i}(\x^*(\lambda,\w_{\tau_0}))/\|P_{A_i}(\x^*(\lambda,\w_{\tau_0}))\|$} is independent of the value of $\lambda$ a.s. This, together with the fact implied by Corollary \ref{cor00}$(ii)$ that, if $A_i\cap E\not=\emptyset$, \red{$P_{A_i}(\x^*(\lambda,\w_{\tau_0}))\rightarrow P_{A_i}(\x^*)$}  as $\lambda\rightarrow0$, shows that, with probability one, each \red{$P_{A_i}(\x^*(\lambda,\w_{\tau_0}))/\|P_{A_i}(\x^*(\lambda,\w_{\tau_0}))\|$} in the expression \eqref{eqn0g} for $\Phi_{\tau_0}(\bxi_1;\x^*+\lambda\w_{\tau_0})$  is a continuous function at $\x^*$ in the corresponding Euclidean space. It follows that
\[\lim_{\w_{\tau_0}\rightarrow\w_\tau\atop\lambda\rightarrow0+}\Phi_{\tau_0}(\bxi_1;\x^*+\lambda\w_{\tau_0})\]
exists a.s. and so, in particular,
\[\lim_{\w_{\tau_0}\rightarrow\w_\tau}\lim_{\lambda\rightarrow0+}\Phi_{\tau_0}(\bxi_1;\x^*+\lambda\w_{\tau_0})=\lim_{\lambda\rightarrow0+}\lim_{\w_{\tau_0}\rightarrow\w_\tau}\Phi_{\tau_0}(\bxi_1;\x^*+\lambda\w_{\tau_0})\qquad a.s.\]
Thus, the definition of $\Psi$ gives
\[\lim_{\w_{\tau_0}\rightarrow\w_\tau}\Psi_{\tau_0}(\bxi_1,\w_{\tau_0};\x^*)=\lim_{\lambda\rightarrow0+}\lim_{\w_{\tau_0}\rightarrow\w_\tau}\Phi_{\tau_0}(\bxi_1;\x^*+\lambda\w_{\tau_0})\qquad a.s.\]
Since the limit on the right hand side exists, to find it, we take a particular path for $\w_{\tau_0}$ to approach $\w_\tau$: $\w_{\tau_0}=\sin\alpha\,\w^\perp+\cos\alpha\w_\tau$, where $\langle\w^\perp,\w_\tau\rangle=0$ and $\|\w^\perp\|=1$. Then, writing $\beta=\lambda\sin\alpha$, we have
\begin{eqnarray*}
&&\lim_{\w_{\tau_0}\rightarrow\w_\tau}\Phi_{\tau_0}(\bxi_1,\x^*+\lambda\w_{\tau_0})\\
&=&\lim_{\alpha\rightarrow0+}\Phi_{\tau_0}\left(\bxi_1,\x^*(\lambda,\w_\tau)+\lambda(\sin\alpha\,\w^\perp+(\cos\alpha-1)\w_\tau)\right)\\
&=&\lim_{\beta\rightarrow0+}\Phi_{\tau_0}\left(\bxi_1,\x^*(\lambda,\w_\tau)+\beta\w^\perp\right)\\
&=&\Psi_{\tau_0}(\bxi_1,\w^\perp;\x^*(\lambda,\w_\tau))\qquad a.s.,
\end{eqnarray*}
where the second equality follows from Corollary \ref{cor0a}$(ii)$. Hence, it follows from Corollary \ref{cor00}$(ii)$ that 
\[P_\tau\!\left(\lim_{\w_{\tau_0}\rightarrow\w_\tau}\!\!\Psi_{\tau_0}(\bxi_1,\w_{\tau_0};\x^*)\!\right)=\!\lim_{\lambda\rightarrow0+}\Phi_\tau(\bxi_1;\x^*(\lambda,\w_\tau))=\Psi_\tau(\bxi_1,\w_\tau;\x^*)\quad a.s.\]
as $\x^*(\lambda,\w_\tau)\in\tau$. \red{Hence, \eqref{eqn11} follows.}

\vskip 6pt
Since $\hat\bxi_n$ will lie in $\mathcal{K}_\mu$ for sufficiently large $n$ a.s. \red{by Corollary \ref{cor1}}, without loss of generality, we assume that it is true for all $n$. Let $\hat\bxi^\tau_n$ denote the sample Euclidean mean of $\tilde\Psi_\tau(\bxi_1;\x^*),\cdots,\tilde\Psi_\tau(\bxi_1;\x^*)$. Then, \red{$\hat\bxi^\tau_n\in\mathbb{R}(E\cup F)$ and}, by Corollary \ref{cor0b}, $\hat\bxi^\tau_n\rightarrow\x^*$ a.s. Also, application of \eqref{eqn11} gives \red{$P_{\tau\setminus\sigma}\left(\hat\bxi^\tau_n\right)=P_{\tau\setminus\sigma}\left(\hat\bxi^{\tau_0}_n\right)$.} On the other hand, the argument for the proof of Theorem \ref{thm4} implies that, for all sufficiently large $n$, 
\[\red{1_\tau(\hat\bxi_n)P_{\tau\setminus\sigma}\left(\hat\bxi_n\right)=1_\tau(\hat\bxi_n)P_{\tau\setminus\sigma}\left(\hat\bxi_n-\x^*\right)=1_\tau(\hat\bxi_n)P_{\tau\setminus\sigma}\left(\hat\bxi^\tau_n\right)}\]
so that, for all sufficiently large $n$,
\[\red{1_\tau(\hat\bxi_n)P_{\tau\setminus\sigma}\left(\hat\bxi_n\right)=1_\tau(\hat\bxi_n)P_{\tau\setminus\sigma}\left(\hat\bxi^{\tau_0}_n\right).}\]
However, \red{given that $\hat\bxi_n$ is in $\mathcal K_\mu$, since $P_\sigma(\hat\bxi_n)=P_\sigma(\hat\bxi^{\tau_0})$ by Corollary \ref{cor00}$(ii)$ and Corollary \ref{cor0b},} the fact that $\hat\bxi_n\in\tau$ is equivalent to the fact that \red{$P_{\tau\setminus\sigma}\left(\hat\bxi^{\tau_0}_n\right)$} lies in $\O(F)$ and $-\red{P_{\tau_0\setminus\tau}\left(\hat\bxi^{\tau_0}_n\right)}\in\O(F_0\setminus F)$. Hence, we can re-express the above equality as
\[\red{1_\tau(\hat\bxi_n)P_{\tau\setminus\sigma}\left(\hat\bxi_n\right)=1_{\O(F)}\left(P_{\tau\setminus\sigma}\left(\hat\bxi^{\tau_0}_n\right)\right)1_{\O(F_0\setminus F)}\left(-P_{\tau_0\setminus\tau}\left(\hat\bxi^{\tau_0}_n\right)\right)P_{\tau\setminus\sigma}\left(\hat\bxi^{\tau_0}_n\right).}\]
The required result then follows by a slight modification to the proof of Theorem \ref{thm4}.
\end{proof}

In fact, the argument for the proof of Proposition \ref{prop7}, in particular \eqref{eqn11}, also shows that, if $\tau=\O(E\cup F)$ has co-dimension greater than $l(\mu)$ and if $\mathbb{R}\times\Theta_{\tau,\sigma}(\x^*;\mu)$ is contained in the interior of $\mathcal{K}_\mu$, then ${\bf P}(\eta\in\mathbb{R}\times\O(F))=0$.

\vskip 20pt
\noindent\textit{Aknowledgements.} The authors are indebted to Megan Owen for her continuing helpful discussions, following her collaboration in \cite{BLO1} and \cite{BLO2}. \red{We are indebted to the referees for helpful suggestions to improve the description and presentation of our results.} The second author acknowledges funding from the Engineering and Physical Sciences Research Council.

\end{document}